\pdfoutput=1
\RequirePackage{type1cm}
\documentclass[11pt,letterpaper,leqno]{article}

\usepackage[letterpaper,textwidth=405bp,textheight=646bp,centering]{geometry}
\usepackage[T1]{fontenc}
\usepackage{times}
\usepackage{amsmath,amsthm,amsfonts,amssymb}
\usepackage{mathrsfs}
\usepackage[numbers]{natbib}
\usepackage[font=small,labelfont=bf]{caption}
\usepackage[colorlinks,citecolor=blue,urlcolor=blue,linkcolor=blue]{hyperref}
\usepackage{graphicx}
\usepackage[ruled,vlined,linesnumbered]{algorithm2e}
\usepackage{tikz}
\usetikzlibrary{positioning,arrows.meta}
\usepackage{xcolor}
\usepackage{url}
\usepackage{tikz-3dplot}
\usepackage{bbm}
\newcommand{\1}{\mathbbm{1}}
\usepackage{pgfplots}
\pgfplotsset{compat=1.18}
\usepgfplotslibrary{groupplots,fillbetween,colormaps}
\hypersetup{linktoc=all,
  pdftitle={Shaping of invariant measures of iterated function systems},
  pdfauthor={Allahkaram Shafiei, Jakub Marecek, Robert Shorten}}
\DeclareFontFamily{U}{rsfs}{\skewchar\font127 }
\DeclareFontShape{U}{rsfs}{m}{n}{<-6> rsfs5 <6-8> rsfs7 <8-> rsfs10}{}
\DeclareFontShapeChangeRule{it}{sc}{sc}{sc}
\DeclareFontShapeChangeRule{sl}{sc}{sc}{sc}
\DeclareFontShapeChangeRule{sc}{it}{it}{it}
\DeclareFontShapeChangeRule{sc}{sl}{sl}{sl}

\usepackage[scaled=.9]{helvet}
\usepackage{textcomp}
\makeatletter
\renewcommand\normalsize{%
  \@setfontsize\normalsize\@xipt{13\p@ plus .5\p@ minus .5\p@}%
  \abovedisplayskip 7\p@ \@plus3\p@ \@minus3\p@
  \belowdisplayskip\abovedisplayskip
  \abovedisplayshortskip\abovedisplayskip
  \belowdisplayshortskip\abovedisplayskip
  \let\@listi\@listI}
\normalsize
\renewcommand\small{%
  \@setfontsize\small\@xpt{12\p@ plus .3\p@ minus .3\p@}%
  \abovedisplayskip 6\p@ \@plus2\p@ \@minus2\p@
  \belowdisplayskip\abovedisplayskip
  \abovedisplayshortskip\abovedisplayskip
  \belowdisplayshortskip\abovedisplayskip}
\renewcommand\footnotesize{%
  \@setfontsize\footnotesize\@ixpt{11\p@ plus .2\p@ minus .2\p@}%
  \abovedisplayskip 6\p@ \@plus2\p@ \@minus2\p@
  \belowdisplayskip\abovedisplayskip
  \abovedisplayshortskip\abovedisplayskip
  \belowdisplayshortskip\abovedisplayskip}
\renewcommand\section{\@startsection{section}{1}{\z@}%
  {-3ex \@plus -1ex \@minus -.2ex}{1.5ex \@plus .2ex}{\normalfont\large\bfseries}}
\renewcommand\subsection{\@startsection{subsection}{2}{\z@}%
  {-2.5ex \@plus -1ex \@minus -.2ex}{1ex \@plus .2ex}{\normalfont\normalsize\bfseries}}
\renewcommand\subsubsection{\@startsection{subsubsection}{3}{\z@}%
  {-2ex \@plus -.5ex \@minus -.2ex}{.5ex \@plus .2ex}{\normalfont\normalsize\itshape}}
\makeatother
\newcommand{\bauthor}[1]{{\scshape #1}}
\newcommand{\oauthor}[1]{{\scshape #1}}
\newcommand{\beditor}[1]{#1}
\newcommand{\bsnm}[1]{#1}
\newcommand{\binits}[1]{#1}
\newcommand{\byear}[1]{#1}
\newcommand{\batitle}[1]{#1}
\newcommand{\bctitle}[1]{#1}
\newcommand{\bjtitle}[1]{\textit{#1}}
\newcommand{\bbtitle}[1]{\textit{#1}}
\newcommand{\bsertitle}[1]{\textit{#1}}
\newcommand{\bvolume}[1]{\textbf{#1}}

\newcommand{\bfpage}[1]{#1}
\newcommand{\blpage}[1]{#1}
\newcommand{\blocation}[1]{#1}
\newcommand{\bpublisher}[1]{#1}
\newenvironment{barticle}{}{}
\newenvironment{bbook}{}{}
\newenvironment{bchapter}{}{}
\newenvironment{botherref}{}{}
\let\endbibitem\relax

\numberwithin{equation}{section}
\theoremstyle{plain}
\newtheorem{theorem}{Theorem}[section]
\newtheorem{lemma}[theorem]{Lemma}
\newtheorem{proposition}[theorem]{Proposition}
\newtheorem{corollary}[theorem]{Corollary}
\theoremstyle{definition}
\newtheorem{definition}[theorem]{Definition}
\newtheorem{assumption}[theorem]{Assumption}
\newtheorem{example}[theorem]{Example}
\newtheorem{remark}[theorem]{Remark}
\newtheorem{openproblem}{Open Problem}

\newcommand{\R}{\mathbb{R}}
\newcommand{\N}{\mathbb{N}}
\newcommand{\Z}{\mathbb{Z}}
\newcommand{\E}{\mathbb{E}}
\newcommand{\PP}{\mathbb{P}}
\newcommand{\cX}{\mathcal{X}}
\newcommand{\cP}{\mathcal{P}}
\newcommand{\cA}{\mathcal{A}}
\newcommand{\cAu}{\mathcal{A}_{\mathrm{u}}}
\newcommand{\cU}{\mathcal{U}}
\newcommand{\cD}{\mathcal{D}}
\newcommand{\cM}{\mathcal{M}}
\newcommand{\cC}{\mathcal{C}}
\newcommand{\cE}{\mathcal{E}}
\newcommand{\cF}{\mathcal{F}}
\newcommand{\cK}{\mathcal{K}}
\newcommand{\cN}{\mathcal{N}}
\newcommand{\cB}{\mathcal{B}}
\newcommand{\sD}{\mathscr{D}}
\newcommand{\sM}{\mathscr{M}}
\newcommand{\sF}{\mathscr{F}}
\newcommand{\Inv}{\operatorname{Inv}}
\newcommand{\supp}{\operatorname{supp}}
\newcommand{\diam}{\operatorname{diam}}
\newcommand{\Lip}{\operatorname{Lip}}
\newcommand{\Fix}{\operatorname{Fix}}
\newcommand{\conv}{\operatorname{conv}}
\newcommand{\dist}{\operatorname{dist}}
\newcommand{\cl}{\operatorname{cl}}
\newcommand{\dTV}{d_{\mathrm{TV}}}
\newcommand{\dK}{d_{\mathrm{K}}}
\newcommand{\dW}{d_{\mathrm{W}}}
\newcommand{\eps}{\varepsilon}
\newcommand{\bi}{\mathbf{i}}
\newcommand{\bj}{\mathbf{j}}
\newcommand{\bI}{\mathbf{I}}
\newcommand{\one}{\mathbf{1}}
\newcommand{\mb}{\partial^{\circ}\!\cA}
\newcommand{\uJ}{\underline{J}}
\newcommand{\Jstar}{J^{\star}}
\newcommand{\PiA}{\Pi_{\cA}}
\newcommand{\id}{\mathrm{id}}
\newcommand{\dd}{\mathrm{d}}
\providecommand{\bseriesno}[1]{#1}
\begin{document}

\title{Shaping of invariant measures of iterated function systems}
\author{Allahkaram Shafiei\thanks{Czech Technical University in Prague,
\href{mailto:shafiall@fel.cvut.cz}{\texttt{shafiall@fel.cvut.cz}}}
\and Robert Shorten\thanks{Imperial College London,
\href{mailto:r.shorten@imperial.ac.uk}{\texttt{r.shorten@imperial.ac.uk}}}
\and Jakub Mare\v{c}ek\thanks{Czech Technical University in Prague,
\href{mailto:jakub.marecek@fel.cvut.cz}{\texttt{jakub.marecek@fel.cvut.cz}}}
}
\date{}
\maketitle

\begin{abstract}
We study an inverse problem for iterated function systems: given maps $w_1,\dots,w_m$ on a Polish space
and a target probability measure $\nu$, choose the law $\mu$ with which the maps are sampled so that the
invariant measure $\pi_\mu$ of the resulting Markov chain is as close as possible to $\nu$ in the
Wasserstein-1 distance. The minimization runs over the \emph{admissible set} $\cA$ of sampling laws for
which the chain has a unique invariant measure attracting every initial law; the maps need not be
contractions. On bounded state spaces, the regions certified by a Doeblin or a splitting condition are
convex unions of relatively open faces of the simplex, and they exhaust $\cA$ when the state space is
finite; for affine systems with nonnegative matrices $\cA$ is a spectral sublevel set; in general $\cA$ is
neither open, nor closed, nor convex, nor connected. Existence of an optimal law reduces to the invariant
measures left behind at the missing boundary $\cl(\cA)\setminus\cA$: boundary points without invariant
measures of finite first moment repel minimizing sequences, which yields existence for all systems of
homotheties of $\R^d$, and for compact state spaces a solution exists for every target if and only if the
set of reachable invariant measures is closed; a two-state example shows that existence can fail. A
search over a $\delta$-net of the simplex returns an $\eps$-optimal law over any region on which
$\mu\mapsto\pi_\mu$ is Lipschitz and the chains are uniformly attracting, with explicit constants and, on an
interval, an explicit sample size. Under contraction on average, $\mu\mapsto\pi_\mu$ is differentiable, its
derivative is injective if and only if the push-forwards $(w_i)_\#\pi_\mu$ are linearly independent, and
the objective is weakly convex, so that a projected subgradient method with inexact subgradients converges
to near-stationary points.

\medskip\noindent
\textit{MSC2020 subject classifications:} Primary 60J05, 37H12; secondary 49Q22, 60J20, 37A50,
28A80, 93E20.

\end{abstract}

\section{Introduction}\label{sec:intro}

\subsection{The shaping problem}
Many stochastic systems are built by composing, at random, finitely many deterministic mechanisms: a
queue receives or serves a customer, a battery is charged or discharged, an agent in a shared-resource
network is switched on or off, a pixel of a fractal image is pushed by one of a few affine maps. If the
mechanism is selected independently at each step with law $\mu=(\mu_1,\dots,\mu_m)$ on the index set
$[m]:=\{1,\dots,m\}$, the system is a Markov chain
\begin{equation}\label{eq:chain}
X_{k+1}=w_{I_{k+1}}(X_k),\qquad (I_k)_{k\ge1}\ \text{i.i.d.\ with law }\mu,
\end{equation}
known as a (probabilistic) \emph{iterated function system} (IFS), and its long-run behavior is
described by an invariant probability measure $\pi_\mu$. In the classical theory $\mu$ is given and one
studies $\pi_\mu$. In this paper $\mu$ is the \emph{design variable}: we are given a target measure
$\nu$, typically through a histogram or a sample, and we look for the sampling law whose invariant
measure is closest to it; see Figure~\ref{fig:loop}.

Let $(\cX,\rho)$ be a Polish metric space, let $w_1,\dots,w_m:\cX\to\cX$ be continuous maps, let
$\Delta_m:=\{\mu\in\R_+^m:\sum_{i}\mu_i=1\}$, and let $P^*_\mu\lambda:=\sum_{i=1}^m\mu_i\,(w_i)_\#\lambda$
be the Markov operator of~\eqref{eq:chain} acting on the space $\cP_1(\cX)$ of probability measures with
finite first moment. The \emph{invariant measure shaping problem} is
\begin{equation}
\label{eq:A}
\min_{\mu\in\cA}\ \Big[J(\mu):=W_1(\nu,\pi_\mu)\Big]
\tag{A}
\end{equation}
where $W_1$ is the Wasserstein-1 distance and $\cA\subseteq\Delta_m$ is the \emph{admissible set}
\begin{equation}\label{eq:admissible}
\cA
:=
\left\{
\mu\in\Delta_m:
\begin{array}{l}
P_\mu^*\text{ has a unique invariant measure }\pi_\mu\in\cP_1(\cX),\\[1mm]
\text{and }(P_\mu^*)^k\lambda\xrightarrow[W_1]{}\pi_\mu
\quad\text{for every }\lambda\in\cP_1(\cX)
\end{array}
\right\}.
\end{equation}
We minimize over $\cA$ rather than over $\Delta_m$ because the objective is meaningful exactly there:
for $\mu\in\cA$ the long-run law of the system exists, is unique, and does not depend on how the system
is initialized, so that $J(\mu)$ is the distance between the target and what an observer of the running
system will actually see. Outside $\cA$ the quantity $W_1(\nu,\pi_\mu)$ is either undefined (no invariant
measure), ambiguous (several), or not observable (an invariant measure that is not attracting). We
require attraction in $W_1$ rather than weakly because the computational results of
Section~\ref{sec:computation} approximate $\pi_\mu$ by iterates of $P_\mu^*$ in $W_1$; on bounded state
spaces the two requirements coincide, and on unbounded ones the difference is a first-moment effect
(Remark~\ref{rem:weak}).

Problem~\eqref{eq:A} is a finite-dimensional optimization problem with an infinite-dimensional
equilibrium constraint. Its difficulty is concentrated in the feasible set. The set $\cA$ is defined by
an asymptotic property of a Markov operator and is not given by finitely many inequalities; it is in
general not closed, so the direct method of the calculus of variations does not apply; and the map
$\mu\mapsto\pi_\mu$ is nonlinear, so $J$ is not convex. The two questions addressed in this paper are
therefore:
\begin{itemize}
\item[(Q1)] \emph{What does the admissible set $\cA$ look like?}
\item[(Q2)] \emph{Under which conditions does Problem~\eqref{eq:A} have a solution, and how can
an approximate solution be computed from data?}
\end{itemize}
A guiding principle of the paper is to answer both questions \emph{without assuming that the maps
$w_i$ are contractions}, neither individually nor on average. Contractivity is the standard route to
$\cA=\Delta_m$, but it excludes saturating, monotone, expanding or merely nonexpansive mechanisms that
are common in applications (Section~\ref{sec:numerics} treats a storage model in which every map has
Lipschitz constant exactly one), and it hides the phenomena that make~\eqref{eq:A} interesting.

\begin{figure}[tb]
\centering
\begin{tikzpicture}[
    node distance=9mm and 12mm,
    every node/.style={font=\small},
    block/.style={rectangle, draw, thick, rounded corners=2pt, align=center, minimum height=10mm, inner sep=2mm},
    arrow/.style={-{Latex[length=2mm]}, thick},
    lbl/.style={font=\scriptsize, midway}]
\node[block] (sup) {supervisor\\ broadcasts $\mu\in\Delta_m$};
\node[block, right=of sup] (sys) {population of agents:\\ $X_{k+1}=w_{I_{k+1}}(X_k)$, $I_{k+1}\sim\mu$};
\node[block, right=of sys] (pi) {invariant\\ measure $\pi_\mu$};
\node[block, below=of sys] (cmp) {shaping error $J(\mu)=W_1(\nu,\pi_\mu)$\\ and subgradient $s_\mu(f)$};
\node[block, below=of sup] (upd) {design update\\ (Alg.~\ref{alg:shaping} or \ref{alg:psg})};
\node[block, below=6mm of cmp] (tgt) {target $\nu$ (or samples $\hat\nu_N$)};
\draw[arrow] (sup) -- (sys);
\draw[arrow] (sys) -- node[lbl, above] {$k\to\infty$} (pi);
\draw[arrow] (pi) |- (cmp);
\draw[arrow] (tgt) -- (cmp);
\draw[arrow] (cmp) -- (upd);
\draw[arrow] (upd) -- (sup);
\end{tikzpicture}
\caption{An illustration of possible applications of invariant-measure shaping. The supervisor broadcasts a sampling law $\mu$; each agent switches between the modes $w_1,\dots,w_m$ independently according to $\mu$. The design loop (offline) adjusts $\mu$ so that the long-run distribution $\pi_\mu$ is close to the target $\nu$ in the Wasserstein distance.}
\label{fig:loop}
\end{figure}
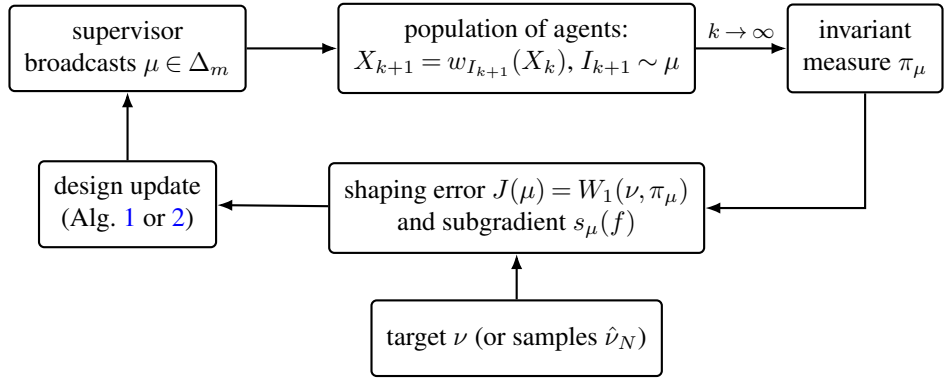
\subsection{Two illustrations}
\emph{A continuous-time analogue.} The viewpoint is not specific to IFS. Consider the
Ornstein--Uhlenbeck family $\dd X_t=-\theta X_t\,\dd t+\sigma\,\dd B_t$ with parameter $\theta$. Solving
the stationary Fokker--Planck equation $\theta\,(x p(x))'+\tfrac{\sigma^2}{2}p''(x)=0$ gives the invariant law
$\pi_\theta=N(0,\sigma^2/(2\theta))$, which exists, is unique and is globally attracting if and only if
$\theta>0$. For the target $\nu=N(0,v)$ one finds
$W_1(\nu,\pi_\theta)=\sqrt{2/\pi}\,\big|\sigma/\sqrt{2\theta}-\sqrt v\big|$, minimized at
$\theta^\dagger=\sigma^2/(2v)$. The admissible set $(0,\infty)$ is not closed, but as $\theta\downarrow0$ the
invariant law spreads out and the objective blows up: the missing boundary point repels minimizing
sequences. Theorem~\ref{thm:universal} shows that this mechanism is general. (The parameter set is not
compact here, and its other end behaves differently: as $\theta\to\infty$, $\pi_\theta\to\delta_0$, which is not
reachable, so that for the target $\nu=\delta_0$ the infimum is not attained.)

\emph{A discrete example.} On $\cX=\R$ take $w_1(x)=x+1$ and $w_2(x)=x/2$. For $\mu=(1-p,p)$ with
$p>0$ the chain~\eqref{eq:chain} has a unique attracting invariant measure with mean $2(1-p)/p$,
whereas for $p=0$ it is a deterministic translation and has no invariant probability measure at all.
Hence $\cA=\Delta_2\setminus\{e_1\}$ is not closed, and it is not clear a priori that~\eqref{eq:A} has a
solution. It does, for every target with finite mean (Theorem~\ref{thm:homothety}), but a change of
maps can destroy this: for the two-state system of Example~\ref{ex:nonexistence} the infimum
in~\eqref{eq:A} is not attained.

\subsection{Contributions}
The paper makes four contributions.

\smallskip\noindent
\textbf{(1) Geometry of the admissible set.} We introduce a hierarchy
$\cAu\subseteq\cA\subseteq\cU$ of uniformly attracting, attracting, and uniquely ergodic sampling laws
and show that, for compact $\cX$, the outer and inner sets are $G_\delta$ subsets of the simplex on
which $\mu\mapsto\pi_\mu$ is continuous (Theorem~\ref{thm:topology}). We then identify three
\emph{certified regions} inside $\cA$, defined by a Doeblin condition, by the Dubins--Freedman splitting
condition for monotone maps, and by contraction on average. Our main structural observation is that the
first two depend on $\mu$ only through its support and are monotone in the support. A combinatorial
lemma (Lemma~\ref{lem:face}) turns this into geometry: the Doeblin region and the splitting region are
convex, relatively open in $\Delta_m$, and their complements are unions of closed faces
(Theorems~\ref{thm:doeblin} and~\ref{thm:splitting}); on bounded state spaces they are contained in $\cA$,
and boundedness cannot be dropped (Example~\ref{ex:doeblin-unbounded}). For finite state spaces the
Doeblin region is all of $\cA$, which gives a complete description (Theorem~\ref{thm:finite}). For
compact state spaces $\cA$ is co-analytic (Remark~\ref{rem:topology}). For affine maps on $\R^d$ we
sandwich $\cA$ between two spectral conditions which coincide for entrywise nonnegative matrices, so
that $\cA=\{\mu:\rho(\sum_i\mu_iA_i)<1\}$; for homotheties $w_i(x)=a_ix+b_i$ with $a_i\ge0$ this is the
slice of the simplex by the open half-space $\{\sum_i\mu_ia_i<1\}$ (Theorem~\ref{thm:affine}).
Examples~\ref{ex:nonclosed}--\ref{ex:twolane} show that nothing of this survives in general: $\cA$ can
fail to be closed, open, convex or connected, even on a compact state space with Lipschitz maps.

\smallskip\noindent
\textbf{(2) Existence theory.} We prove that $J$ is lower semicontinuous on $\cA$ without any
contraction or compactness of $\cX$, and we reduce solvability of~\eqref{eq:A} to its lower
semicontinuous relaxation on the closure of $\cA$ (Theorem~\ref{thm:exact}): a solution exists if and
only if some admissible law does at least as well as every point of the \emph{missing boundary}
$\cl(\cA)\setminus\cA$, and the relaxed objective at a missing boundary point is the distance from $\nu$
to the invariant measures that the admissible laws leave behind there. The decisive object is therefore
this set of \emph{limiting invariant measures}. For compact $\cX$ the closure of the set
$\{\pi_\mu:\mu\in\cA\}$ of reachable invariant measures consists of the reachable ones and the limiting
ones, so that~\eqref{eq:A} has a solution \emph{for every target} if and only if the reachable set is
closed (Theorem~\ref{thm:universal}). Two mechanisms make this verifiable. A boundary point at which the
chain has no invariant measure of finite first moment is automatically repelling, and a boundary point
all of whose limiting invariant measures are reachable from inside is harmless. Combining the two mechanisms we show that for
\emph{every} system of homotheties on $\R^d$ with at least one contraction, Problem~\eqref{eq:A} has a
solution for every target in $\cP_1(\R^d)$, with no condition on the translation parts
(Theorem~\ref{thm:homothety}).

\smallskip\noindent
\textbf{(3) Computation with guarantees.} We show that a search over a $\delta$-net of the simplex
(Lemma~\ref{lem.delta net}) returns, from a finite sample of $\nu$ and finitely many iterations of
$P_\mu^*$, a sampling law that is $\eps$-optimal over any \emph{regular search region}, i.e.\ a region on
which $\mu\mapsto\pi_\mu$ is Lipschitz and the chains are uniformly attracting (Theorem~\ref{thm:eps}); it
is $\eps$-optimal for~\eqref{eq:A} itself whenever the region contains a solution, which the existence
theory guarantees a priori for homotheties (Proposition~\ref{prop:localization}). The constants are
explicit in each certified region, and so is the sample size, with high probability
(Corollary~\ref{cor:samplesize}); they come from a perturbation bound for uniformly ergodic kernels
(Proposition~\ref{prop:unifergodic}) in which contraction of the maps is replaced by contraction of the
Markov operator in total variation or in the Kolmogorov metric. A support-preserving refinement of the
$\delta$-net lemma ensures that the net is compatible with the facial structure of $\cA$. The search is
exponential in $m$, and we regard it as a certificate rather than as a practical method beyond a handful of
modes. For systems that contract on average we prove a linear-response formula for $\mu\mapsto\pi_\mu$ in
terms of a Poisson equation (Theorem~\ref{thm:linear-response}), show that its derivative is injective if
and only if the push-forwards $(w_i)_\#\pi_\mu$ are linearly independent
(Proposition~\ref{prop:injectivity}), show that $J$ is weakly convex with computable subgradients
(Theorem~\ref{thm:structure}), derive an equal-sensitivity optimality condition
(Corollary~\ref{cor:optimality}) and a parameter-recovery bound (Corollary~\ref{cor:recovery}), and analyze
a projected subgradient method with inexact subgradients (Theorem~\ref{thm:psg} and
Proposition~\ref{prop:inexact}).

\smallskip\noindent
\textbf{(4) Numerical evidence.} Section~\ref{sec:numerics} reports experiments on a saturating storage
model (non-contractive, with $\cA$ not closed), on an exactly computable ternary benchmark, on
homotheties near the missing boundary, on the temperature distribution of a population of
thermostatically controlled loads, and on systems with $m=8$ and $m=12$ modes, for which a $\delta$-net is
out of reach; a contractive benchmark is reported in Appendix~\ref{sec:supp-battery}. Code that reproduces the
experiments of Sections~\ref{sec:exp-ternary}, \ref{sec:exp-tcl} and~\ref{sec:exp-many} and of
Appendix~\ref{sec:supp-affine} is available from the authors on request.

\subsection{Novelty and related work}\label{sec:related}
\emph{Iterated function systems.} IFS were introduced by Hutchinson~\cite{Hutchinson1981} and Barnsley
and Demko~\cite{BarnsleyDemko1985}; existence, uniqueness and attractivity of the invariant measure
under contraction on average go back to~\cite{BarnsleyElton1988,BarnsleyEtAl1988}, with ergodic
theorems by Elton~\cite{Elton1987} and a unifying backward-iteration argument by Diaconis and
Freedman~\cite{DiaconisFreedman1999}; see the surveys~\cite{Stenflo2012,GhoshMarecek2022}. Weaker forms of
contractivity are treated in~\cite{Steinsaltz1999,JarnerTweedie2001,LasotaYorke1994,Szarek2003}, and
non-contractive monotone systems by the splitting method of Dubins and
Freedman~\cite{DubinsFreedman1966,BhattacharyaLee1988,BhattacharyaMajumdar2007}; general tools for
unique ergodicity beyond minorization include asymptotic
coupling~\cite{HairerMattinglyScheutzow2011,HairerMattingly2011}. All these works fix the sampling law.
For products of random matrices, the dependence of the top Lyapunov exponent on the weights of the
matrices is classical~\cite{Peres1992}, and it governs almost sure stability. We ask how the \emph{set}
of laws for which unique ergodicity and $W_1$-attraction hold sits inside the simplex, for maps that need
not be linear or contractive, which to our knowledge has not been studied systematically.

\emph{Inverse problems for IFS.} The classical inverse problem asks for maps \emph{and} probabilities
such that the attractor or the invariant measure approximates a target. It is usually approached
through Barnsley's collage theorem~\cite{BarnsleyEtAl1986,BarnsleyHurd1993,Jacquin1992}, and for measures
by moment matching and collage-type minimization~\cite{HandyMantica1990,Vrscay1991,ForteVrscay1995,KunzeEtAl2012}.
These methods assume strictly contractive maps, so that the feasible set is the whole parameter space,
and they minimize the collage distance, an upper bound on the true error. Problem~\eqref{eq:A} is
complementary: the maps are fixed by the application, the true distance $W_1(\nu,\pi_\mu)$ is minimized,
and the feasible set is part of the problem.

\emph{Designing the stationary law of a Markov chain.} Constructing a chain with a prescribed
stationary law is the principle of Markov chain Monte Carlo~\cite{Hastings1970}, and optimizing a chain
subject to a fixed stationary law is the fastest-mixing problem~\cite{BoydDiaconisXiao2004}. For
deterministic maps the corresponding question is the inverse Frobenius--Perron
problem~\cite{Bollt2000}. In control, steering the law of a diffusion is studied through the
Fokker--Planck equation~\cite{AnnunziatoBorzi2013} and through covariance steering and Schr\"odinger
bridges~\cite{ChenGeorgiouPavon2016}. Closest to our motivation is the ergodic control of
ensembles~\cite{FioravantiEtAl2019,KungurtsevEtAl2023}, where IFS model the closed-loop behavior of
large populations of agents and unique ergodicity is the design requirement, and the control of
populations of flexible loads by broadcast randomized policies, whose transition probabilities are
designed to shape the stationary distribution of each load~\cite{MeynEtAl2015}. In all these settings the
family of available kernels is rich enough to reach the target exactly. In~\eqref{eq:A} the family
$\{P_\mu\}$ is $(m-1)$-dimensional, the target is in general not reachable, and the question becomes
one of best approximation over a set of invariant measures that need not be closed.

\emph{Sensitivity of invariant measures.} The dependence of the stationary law on the kernel is classical
for finite chains~\cite{Schweitzer1968,ChoMeyer2001} and for uniformly ergodic
chains~\cite{Mitrophanov2005}; Wasserstein versions are given in~\cite{RudolfSchweizer2018}, and linear
response in~\cite{HairerMajda2010}. Proposition~\ref{prop:unifergodic} is a perturbation bound of this type
written for a family of metrics, which lets us treat the Doeblin, splitting and contractive regimes at
once. Theorem~\ref{thm:linear-response} is a linear-response formula in the spirit of~\cite{HairerMajda2010},
written for the sampling law of an IFS. The inverse question at first order, namely which perturbation
of a system produces a prescribed or an extremal change of its invariant measure, is the \emph{linear
request} or \emph{optimal response} problem~\cite{GalatoloPollicott2017,Kloeckner2018,AntownEtAl2018};
Problem~\eqref{eq:A} is its global, finite-dimensional counterpart, in which the perturbations are
restricted to the simplex of sampling laws and the target need not be reachable.

\emph{Stability under random switching.} That stability of a randomly switched system can depend on the
switching law in a non-monotone way is well documented in continuous
time~\cite{BenaimEtAl2014,LawleyMattinglyReed2014,BenaimEtAl2015}. Example~\ref{ex:nonconvex} is a
discrete-time counterpart for~$\cA$. For affine recursions the roles of the Lyapunov exponent and of
moment conditions are classical~\cite{Kesten1973,Brandt1986,BuraczewskiDamekMikosch2016}; the first-moment
quantity that governs $W_1$-admissibility is the $p$-radius with $p=1$~\cite{JungersProtasov2011}.

\emph{Optimal transport.} We use $W_1$ because it metrizes weak convergence together with convergence
of first moments, is insensitive to the mutual singularity that is typical of invariant measures of IFS,
and is computable~\cite{Villani2009,Santambrogio2015,PeyreCuturi2019}; rates for empirical measures are
taken from~\cite{FournierGuillin2015,WeedBach2019,BobkovLedoux2019}. When the target is realizable,
minimizing $W_1(\hat\nu_N,\pi_\mu)$ is a minimum-Wasserstein estimator of
$\mu$~\cite{BassettiBodiniRegazzini2006,BerntonEtAl2019}; Corollaries~\ref{cor:consistency}
and~\ref{cor:recovery} are results of this type in which the model $\mu\mapsto\pi_\mu$ is only given
implicitly, as a fixed point.

\subsection{Organization and notation}
Section~\ref{sec:setting} fixes the setting. Section~\ref{sec:geometry} studies the geometry of $\cA$
(question (Q1)). Section~\ref{sec:existence} contains the existence theory (first half of (Q2)), and
Section~\ref{sec:computation} the $\delta$-net algorithm and a state-space discretization (second half).
Section~\ref{sec:sensitivity} studies the sensitivity of $\pi_\mu$ in the contractive regime and a
projected subgradient method. Section~\ref{sec:numerics} reports numerical experiments and
Section~\ref{sec:conclusion} lists open problems.

For a nonempty $S\subseteq[m]$ we write $\Delta_S:=\{\mu\in\Delta_m:\supp\mu\subseteq S\}$ for the closed
face of the simplex spanned by $S$ and $\Delta_S^\circ:=\{\mu\in\Delta_m:\supp\mu=S\}$ for its relative
interior, where $\supp\mu:=\{i:\mu_i>0\}$; $e_1,\dots,e_m$ are the vertices of $\Delta_m$, and
$\|\cdot\|_1$ is the $\ell_1$-norm on $\R^m$. Topological notions for subsets of $\Delta_m$ (open,
closed, closure $\cl$) are relative to $\Delta_m$. For a matrix $M$, $\rho(M)$ is its spectral radius.
For a word $\bi=(i_1,\dots,i_k)\in[m]^k$ we put
\[
w_{\bi}:=w_{i_k}\circ\cdots\circ w_{i_1},\qquad \mu_{\bi}:=\mu_{i_1}\cdots\mu_{i_k}.
\]

\section{Setting and preliminaries}\label{sec:setting}

\subsection{Measures, Wasserstein distance, and the Markov operator}
Throughout, $(\cX,\rho)$ is a complete separable metric space with Borel $\sigma$-algebra $\cB(\cX)$,
$\cP(\cX)$ is the set of Borel probability measures on $\cX$, and $x_\circ\in\cX$ is a fixed base point. We
write $M_1(\lambda):=\int_\cX\rho(x,x_\circ)\,\lambda(\dd x)$ and
$\cP_1(\cX):=\{\lambda\in\cP(\cX):M_1(\lambda)<\infty\}$. For $\alpha,\beta\in\cP_1(\cX)$,
\[
W_1(\alpha,\beta):=\inf_{\gamma\in\Gamma(\alpha,\beta)}\int_{\cX\times\cX}\rho(x,y)\,\gamma(\dd x,\dd y)
=\sup\Big\{\int f\,\dd\alpha-\int f\,\dd\beta:\ \Lip(f)\le1\Big\},
\]
where $\Gamma(\alpha,\beta)$ is the set of couplings of $\alpha$ and $\beta$ and the second expression is
the Kantorovich--Rubinstein duality. We shall use the following standard facts
\cite[Chapter~6]{Villani2009}:
\begin{itemize}
\item[(W1)] $(\cP_1(\cX),W_1)$ is a complete separable metric space, and $W_1(\lambda_n,\lambda)\to0$ if and
only if $\lambda_n\to\lambda$ weakly and $M_1(\lambda_n)\to M_1(\lambda)$.
\item[(W2)] $W_1$ is lower semicontinuous for weak convergence: if $\alpha_n\to\alpha$ and
$\beta_n\to\beta$ weakly, then $W_1(\alpha,\beta)\le\liminf_n W_1(\alpha_n,\beta_n)$.
\item[(W3)] $W_1$ is jointly convex: $W_1(\sum_jc_j\alpha_j,\sum_jc_j\beta_j)\le\sum_jc_jW_1(\alpha_j,\beta_j)$
for convex weights $(c_j)$.
\item[(W4)] $M_1(\lambda)=W_1(\lambda,\delta_{x_\circ})$, and $W_1(\alpha,\beta)\le\diam(\cX)\,\dTV(\alpha,\beta)$,
where $\dTV(\alpha,\beta):=\sup_{B\in\cB(\cX)}|\alpha(B)-\beta(B)|$.
\end{itemize}
For a Borel map $T:\cX\to\cX$ and $\lambda\in\cP(\cX)$, the pushforward $T_\#\lambda$ is the law of $T(X)$
when $X$ has law $\lambda$.

\begin{assumption}[Standing assumption]\label{ass:standing}
The maps $w_1,\dots,w_m:\cX\to\cX$ are continuous and have at most linear growth: there is $c_0<\infty$
such that $\rho(w_i(x),x_\circ)\le c_0\,(1+\rho(x,x_\circ))$ for all $x\in\cX$ and $i\in[m]$.
\end{assumption}

The growth condition is void when $\cX$ is bounded and holds whenever the maps are Lipschitz. No
contraction property is assumed. For $\mu\in\Delta_m$ the transition kernel of the
chain~\eqref{eq:chain} and the associated Markov operator are
\[
P_\mu(x,B):=\sum_{i=1}^m\mu_i\,\one_B(w_i(x)),\qquad
P_\mu^*\lambda:=\int_\cX P_\mu(x,\cdot)\,\lambda(\dd x)=\sum_{i=1}^m\mu_i\,(w_i)_\#\lambda ,
\]
so that $(P_\mu^*)^k\lambda=\sum_{\bi\in[m]^k}\mu_{\bi}\,(w_{\bi})_\#\lambda$ is the law of $X_k$ when $X_0$
has law $\lambda$.

\begin{lemma}[Basic properties of the Markov operator]\label{lem:basic}
Under Assumption~\ref{ass:standing}:
\begin{enumerate}
\item[\rm(i)] $P_\mu^*$ maps $\cP_1(\cX)$ into itself and $M_1(P_\mu^*\lambda)\le c_0(1+M_1(\lambda))$.
\item[\rm(ii)] If $\mu_n\to\mu$ in $\Delta_m$ and $\lambda_n\to\lambda$ weakly, then
$P_{\mu_n}^*\lambda_n\to P_\mu^*\lambda$ weakly.
\item[\rm(iii)] For all $\mu,\eta\in\Delta_m$ and $\lambda\in\cP_1(\cX)$,
\begin{equation}\label{eq:mixture}
W_1\big(P_\mu^*\lambda,P_\eta^*\lambda\big)\le\frac12\,\|\mu-\eta\|_1\,
\max_{i,j\in[m]}\int_\cX\rho\big(w_i(x),w_j(x)\big)\,\lambda(\dd x).
\end{equation}
\end{enumerate}
\end{lemma}

\begin{proof}
(i) is immediate from the growth condition. For (ii), let $f$ be bounded and continuous. Then
$f\circ w_i$ is bounded and continuous, so
$\int f\,\dd(P_{\mu_n}^*\lambda_n)=\sum_i\mu_{n,i}\int f\circ w_i\,\dd\lambda_n\to\sum_i\mu_i\int f\circ w_i\,\dd\lambda$.
For (iii), put $c:=\frac12\|\mu-\eta\|_1=\sum_i(\mu_i-\eta_i)^+=\sum_i(\eta_i-\mu_i)^+$ and assume $c>0$. Then
$\mu=(\mu\wedge\eta)+c\,\alpha$ and $\eta=(\mu\wedge\eta)+c\,\beta$ with $\alpha,\beta\in\Delta_m$, where
$\mu\wedge\eta$ is the coordinatewise minimum. Hence $P_\mu^*\lambda$ and $P_\eta^*\lambda$ are mixtures
with a common component of mass $1-c$, and two applications of (W3) give
\[
W_1(P_\mu^*\lambda,P_\eta^*\lambda)\le c\,W_1(P_\alpha^*\lambda,P_\beta^*\lambda)
\le c\sum_{i,j}\alpha_i\beta_j\,W_1\big((w_i)_\#\lambda,(w_j)_\#\lambda\big).
\]
Finally $W_1((w_i)_\#\lambda,(w_j)_\#\lambda)\le\int\rho(w_i(x),w_j(x))\,\lambda(\dd x)$, by coupling through
the same $x$.
\end{proof}

\begin{lemma}[Contraction and unique invariant measure]\label{lem:contraction}
Assume that each $w_i$ is Lipschitz continuous with constant $L_i\ge0$, and define
$L(\mu):=\sum_{i=1}^m\mu_iL_i$. If $L(\mu)<1$, then $P_\mu^*$ is a strict contraction on
$(\cP_1(\cX),W_1)$ with constant $L(\mu)$. Consequently, $P_\mu^*$ admits a unique invariant measure
$\pi_\mu\in\cP_1(\cX)$, and
\[
W_1\big((P_\mu^*)^k\lambda,\pi_\mu\big)\le L(\mu)^k\,W_1(\lambda,\pi_\mu)\qquad\text{for all }\lambda\in\cP_1(\cX),\ k\ge0.
\]
In particular, $\mu\in\cA$.
\end{lemma}

\begin{proof}
Let $\lambda,\lambda'\in\cP_1(\cX)$. Pushing an optimal coupling of $(\lambda,\lambda')$ forward by
$w_i\times w_i$ gives $W_1((w_i)_\#\lambda,(w_i)_\#\lambda')\le L_i\,W_1(\lambda,\lambda')$, so by (W3)
\[
W_1(P_\mu^*\lambda,P_\mu^*\lambda')
\le\sum_{i=1}^m\mu_i\,W_1\big((w_i)_\#\lambda,(w_i)_\#\lambda'\big)
\le\sum_{i=1}^m\mu_iL_i\,W_1(\lambda,\lambda')=L(\mu)\,W_1(\lambda,\lambda').
\]
Since $(\cP_1(\cX),W_1)$ is complete by (W1), the Banach fixed point theorem yields a unique fixed point
$\pi_\mu\in\cP_1(\cX)$ of $P_\mu^*$, that is, a unique invariant measure with finite first moment. By
induction,
$W_1((P_\mu^*)^k\lambda,\pi_\mu)=W_1((P_\mu^*)^k\lambda,(P_\mu^*)^k\pi_\mu)\le L(\mu)^k\,W_1(\lambda,\pi_\mu)$,
which tends to zero; hence $\mu\in\cA$.
\end{proof}

\begin{lemma}[Lipschitz stability of the invariant measure]\label{lem:lipschitz-pi}
Let $\cX$ be compact with diameter $D:=\diam(\cX)<\infty$, let the maps $w_i$ and $L(\mu)$ be as in
Lemma~\ref{lem:contraction}, and for $r\in[0,1)$ put
\[
\cA_r:=\{\mu\in\Delta_m:L(\mu)\le r\}.
\]
Then $\cA_r\subseteq\cA$, and for all $\mu,\eta\in\cA_r$,
\[
W_1(\pi_\mu,\pi_\eta)\le\frac{D}{2(1-r)}\,\|\mu-\eta\|_1.
\]
\end{lemma}

\begin{proof}
The inclusion $\cA_r\subseteq\cA$ is Lemma~\ref{lem:contraction}. Since $\pi_\mu$ and $\pi_\eta$ are fixed
points,
\[
\begin{aligned}
W_1(\pi_\mu,\pi_\eta)&=W_1(P_\mu^*\pi_\mu,P_\eta^*\pi_\eta)
\le W_1(P_\mu^*\pi_\mu,P_\mu^*\pi_\eta)+W_1(P_\mu^*\pi_\eta,P_\eta^*\pi_\eta)\\
&\le r\,W_1(\pi_\mu,\pi_\eta)+W_1\Big(\sum_i\mu_i(w_i)_\#\pi_\eta,\ \sum_i\eta_i(w_i)_\#\pi_\eta\Big).
\end{aligned}
\]
The last term is the distance between two mixtures of the same probability measures
$(w_i)_\#\pi_\eta$. Their total variation distance is at most $\frac12\|\mu-\eta\|_1$, so by (W4) the term
is at most $\frac D2\|\mu-\eta\|_1$. Rearranging gives the claim.
\end{proof}

The same three steps (invariance, contraction of one of the two operators, and a bound on the distance
between two mixtures of the same push-forwards) give all perturbation bounds of this paper:
Proposition~\ref{prop:unifergodic}(ii) replaces contraction in $W_1$ by contraction in a compatible metric,
Proposition~\ref{prop:localization}(iii) and~\eqref{eq:stability-general} replace $D$ by a moment of the
invariant measure.

\subsection{Invariant measures and the admissible set}
For $\mu\in\Delta_m$ let
\[
\Inv(\mu):=\{\pi\in\cP_1(\cX):P_\mu^*\pi=\pi\}
\]
be the set of invariant measures with finite first moment, and let
\[
\cU:=\{\mu\in\Delta_m:\Inv(\mu)\text{ is a singleton}\}
\]
be the set of \emph{uniquely ergodic} sampling laws. For $\mu\in\cU$ we denote by $\pi_\mu$ the unique
element of $\Inv(\mu)$. The admissible set $\cA$ is defined by~\eqref{eq:admissible}, and
\[
\PiA:=\{\pi_\mu:\mu\in\cA\}\subseteq\cP_1(\cX)
\]
is the set of \emph{reachable} invariant measures. With this notation, Problem~\eqref{eq:A} asks for a point
of $\PiA$ nearest to $\nu$ in $W_1$, together with a parameter that realizes it.

\begin{lemma}\label{lem:attraction}
Let $\mu\in\Delta_m$ and $\pi\in\cP_1(\cX)$ be such that $(P_\mu^*)^k\lambda\to\pi$ in $W_1$ for every
$\lambda\in\cP_1(\cX)$. Then $\Inv(\mu)=\{\pi\}$. Consequently
\[
\cA=\big\{\mu\in\Delta_m:\ \exists\,\pi\in\cP_1(\cX)\ \text{with}\ (P_\mu^*)^k\lambda\to\pi\ \text{in }W_1\
\text{for all }\lambda\in\cP_1(\cX)\big\}\subseteq\cU .
\]
\end{lemma}

\begin{proof}
Fix $\lambda\in\cP_1(\cX)$ and put $\lambda_k:=(P_\mu^*)^k\lambda$. By Lemma~\ref{lem:basic}(ii),
$\lambda_{k+1}=P_\mu^*\lambda_k\to P_\mu^*\pi$ weakly; since also $\lambda_{k+1}\to\pi$, we get $P_\mu^*\pi=\pi$. If $\pi'\in\Inv(\mu)$, then
$\pi'=(P_\mu^*)^k\pi'\to\pi$, so $\pi'=\pi$.
\end{proof}

Thus the uniqueness requirement in~\eqref{eq:admissible} is a consequence of global attraction; we keep it
in the definition for emphasis. The inclusion $\cA\subseteq\cU$ can be strict: a periodic chain has a
unique invariant law which does not attract (Example~\ref{ex:nonexistence}). The sets $\cA$ and $\cU$
and the measures $\pi_\mu$ do not depend on the base point $x_\circ$.

\begin{remark}[Attraction in $W_1$ versus weak attraction]\label{rem:weak}
Let $\cA_{\rm w}\supseteq\cA$ be the set of $\mu\in\cU$ such that $(P_\mu^*)^k\lambda\to\pi_\mu$ weakly for
every $\lambda\in\cP_1(\cX)$. By (W1), $W_1$-convergence is weak convergence together with convergence of
first moments, and on a bounded space the latter follows from the former, because $x\mapsto\rho(x,x_\circ)$
is then bounded and continuous. Hence $\cA_{\rm w}=\cA$ whenever $\cX$ is bounded, which covers all state
spaces of Section~\ref{sec:geometry} except those of Example~\ref{ex:doeblin-unbounded},
Proposition~\ref{prop:contraction} and Section~\ref{sec:affine}. On unbounded spaces the
inclusion can be strict: in Example~\ref{ex:nonconvex} the chain driven by two nilpotent maps is absorbed at
the origin almost surely for every sampling law, so the whole edge lies in $\cA_{\rm w}$, while its first
moments do not converge to zero on a closed sub-interval, which is therefore missing from $\cA$. The
existence theory of Section~\ref{sec:existence} only uses that the feasible set is contained in $\cU$
(Remark~\ref{rem:feasible}) and applies verbatim to $\cA_{\rm w}$. We nevertheless work with $\cA$, because
the computational results of Section~\ref{sec:computation} approximate $\pi_\mu$ by the iterates
$(P_\mu^*)^k\lambda_0$, and the resulting truncation error is controlled in $W_1$ only under
$W_1$-attraction.
\end{remark}

\begin{remark}[General parametrized families]\label{rem:general}
An IFS with sampling law $\mu$ is a parametrized stochastic dynamical system whose kernel
$P_\mu=\sum_i\mu_iP_{e_i}$ depends affinely on the parameter. The existence theory of
Section~\ref{sec:existence} uses only two properties of the family $(P_\mu)_{\mu\in\Delta_m}$: the
parameter set is a compact metric space, and $(\mu,\lambda)\mapsto P_\mu^*\lambda$ is jointly weakly
continuous (Lemma~\ref{lem:basic}(ii)). It therefore applies verbatim to any family
$(P_\theta)_{\theta\in\Theta}$ of Feller kernels with these two properties. The geometric results of
Section~\ref{sec:geometry}, in contrast, exploit the simplex and the affine dependence on $\mu$.
\end{remark}

\section{Geometry of the admissible set}\label{sec:geometry}

\subsection{General topological structure}\label{sec:topology}
We begin with what can be said without any structural hypothesis on the maps. Let $\cX$ be compact
with diameter $D$, so that $\cP_1(\cX)=\cP(\cX)$ is compact for $W_1$. For $k\ge1$ define
\begin{equation}\label{eq:Phik}
\Phi_k(\mu):=\sup_{x,y\in\cX}W_1\big(P_\mu^k(x,\cdot),P_\mu^k(y,\cdot)\big),
\qquad
\cAu:=\big\{\mu\in\Delta_m:\lim_{k\to\infty}\Phi_k(\mu)=0\big\}.
\end{equation}

\begin{theorem}[Topological structure]\label{thm:topology}
Let $\cX$ be compact and let Assumption~\ref{ass:standing} hold.
\begin{enumerate}
\item[\rm(i)] For every $\mu\in\Delta_m$ the set $\Inv(\mu)$ is nonempty, convex and compact, and the
graph $\{(\mu,\pi):\pi\in\Inv(\mu)\}$ is closed in $\Delta_m\times\cP(\cX)$.
\item[\rm(ii)] $\cAu\subseteq\cA\subseteq\cU$. Moreover $\mu\in\cAu$ if and only if $\mu\in\cA$ and
$\sup_{\lambda\in\cP(\cX)}W_1\big((P_\mu^*)^k\lambda,\pi_\mu\big)\to0$ as $k\to\infty$, and in that case this
supremum is at most $\Phi_k(\mu)$.
\item[\rm(iii)] $\cU$ and $\cAu$ are $G_\delta$ subsets of $\Delta_m$.
\item[\rm(iv)] The map $\mu\mapsto\pi_\mu$ is continuous from $\cU$ to $(\cP(\cX),W_1)$. In particular
$J$ is continuous on $\cU$, hence on $\cA$.
\end{enumerate}
\end{theorem}

\begin{proof}
(i) Existence of an invariant measure for a Feller kernel on a compact space is the Krylov--Bogolyubov
theorem \cite[Theorem~3.1.1]{DaPratoZabczyk1996}. Convexity is clear, the closed graph property is
Lemma~\ref{lem:basic}(ii), and compactness of $\Inv(\mu)$ follows.

(ii) Let $f$ be $1$-Lipschitz and write $P_\mu^kf(x):=\int f(z)\,P_\mu^k(x,\dd z)$. By duality,
$P_\mu^kf(x)-P_\mu^kf(y)\le\Phi_k(\mu)$ for all $x,y$, hence for all $\lambda,\lambda'\in\cP(\cX)$
\begin{equation}\label{eq:Phibound}
W_1\big((P_\mu^*)^k\lambda,(P_\mu^*)^k\lambda'\big)
=\sup_{\Lip(f)\le1}\iint\big(P_\mu^kf(x)-P_\mu^kf(y)\big)\,\lambda(\dd x)\,\lambda'(\dd y)\le\Phi_k(\mu).
\end{equation}
Let $\mu\in\cAu$ and $\pi\in\Inv(\mu)$. Taking $\lambda'=\pi$ in~\eqref{eq:Phibound} gives
$W_1((P_\mu^*)^k\lambda,\pi)\le\Phi_k(\mu)\to0$ for every $\lambda$, so $\mu\in\cA$ by
Lemma~\ref{lem:attraction}, with the stated uniform bound. Conversely, if $\mu\in\cA$ and the supremum tends to zero, then
$\Phi_k(\mu)\le2\sup_\lambda W_1((P_\mu^*)^k\lambda,\pi_\mu)\to0$ by the triangle inequality applied to Dirac
masses. The inclusion $\cA\subseteq\cU$ is Lemma~\ref{lem:attraction}.

(iii) Let $g(\mu)$ be the $W_1$-diameter of $\Inv(\mu)$. If $\mu_n\to\mu$, choose
$\pi_n,\pi_n'\in\Inv(\mu_n)$ with $W_1(\pi_n,\pi_n')=g(\mu_n)$. Along a subsequence realizing
$\limsup_ng(\mu_n)$ and such that $\pi_n\to\pi$ and $\pi_n'\to\pi'$, part~(i) gives $\pi,\pi'\in\Inv(\mu)$, so
$g(\mu)\ge W_1(\pi,\pi')=\limsup_ng(\mu_n)$. Thus $g$ is upper semicontinuous and
$\cU=\bigcap_{j\ge1}\{g<1/j\}$ is a countable intersection of open sets.

For $\cAu$, note first that each $\Phi_k$ is Lipschitz. Indeed
$P_\mu^k(x,\cdot)=\sum_{\bi\in[m]^k}\mu_{\bi}\delta_{w_{\bi}(x)}$ and
$\sum_{\bi\in[m]^k}|\mu_{\bi}-\eta_{\bi}|\le k\|\mu-\eta\|_1$, so (W4) yields
$W_1(P_\mu^k(x,\cdot),P_\eta^k(x,\cdot))\le\frac{Dk}{2}\|\mu-\eta\|_1$ for every $x$, and therefore
$|\Phi_k(\mu)-\Phi_k(\eta)|\le Dk\|\mu-\eta\|_1$. Next, $\Phi_{k+1}\le\Phi_k$: for $1$-Lipschitz $f$,
\[
P_\mu^{k+1}f(x)-P_\mu^{k+1}f(y)=\iint\big(P_\mu^kf(u)-P_\mu^kf(v)\big)\,P_\mu(x,\dd u)\,P_\mu(y,\dd v)\le\Phi_k(\mu).
\]
Taking the supremum over $f$, $x$ and $y$ gives $\Phi_{k+1}(\mu)\le\Phi_k(\mu)$. Since $(\Phi_k(\mu))_{k\ge1}$ is nonincreasing and bounded below by $0$, its limit exists and equals its infimum. Therefore

\[
\lim_{k\to\infty}\Phi_k(\mu)=0
\iff \inf_{k\ge1}\Phi_k(\mu)=0
\iff \forall j\ge1\ \exists k\ge1:\ \Phi_k(\mu)<\tfrac1j ,
\]
so that $\cAu=\bigcap_{j\ge1}\bigcup_{k\ge1}\{\Phi_k<1/j\}$, a countable intersection of open sets because each
$\Phi_k$ is continuous, is a $G_\delta$ set.

(iv) Let $\mu_n\to\mu$ in $\cU$. By compactness of $\cP(\cX)$ and part~(i), every limit point of
$(\pi_{\mu_n})$ lies in $\Inv(\mu)=\{\pi_\mu\}$, so $\pi_{\mu_n}\to\pi_\mu$. Continuity of $J$ follows from
$|J(\mu)-J(\eta)|\le W_1(\pi_\mu,\pi_\eta)$.
\end{proof}

\begin{remark}\label{rem:topology}
(a) The inclusions in (ii) are strict in general. The swap of two points has a unique invariant law that
does not attract, so it lies in $\cU\setminus\cA$. For $\cA\setminus\cAu$, let $\cX=\R/\Z$ and let $w$ be
the map induced by $\theta\mapsto\theta+\frac1{10}\sin^2(\pi\theta)$. Every orbit converges monotonically
to the unique fixed point $0$, so $w^k_\#\lambda\to\delta_0$ for every $\lambda$ by dominated
convergence; but $w^k$ is a homeomorphism of the circle, hence onto, and
$\sup_x\rho(w^k(x),0)=\frac12$ for every $k$.

(b) We do not know whether $\cA$ itself is always a $G_\delta$, or even a Borel, set: pointwise attraction
is not a countable condition; see Open Problem~\ref{op:borel}. It is \emph{co-analytic}, and in particular
universally measurable~\cite[Theorem~21.10]{Kechris1995}. Indeed, on the compact space $\cX$, attraction
of every $\lambda\in\cP(\cX)$ follows from attraction of every Dirac mass, by dominated convergence
applied to $\int P_\mu^kf\,\dd\lambda$ for bounded continuous $f$, so that
\[
\cA=\big\{\mu\in\cU:\ (\mu,x)\in B\ \text{for every }x\in\cX\big\},\qquad
B:=\big\{(\mu,x)\in\cU\times\cX:\ F_k(\mu,x)\to0\big\},
\]
with $F_k(\mu,x):=W_1(P_\mu^k(x,\cdot),\pi_\mu)$. Each $F_k$ is continuous on $\cU\times\cX$ by part~(iv)
and because $P_\mu^k(x,\cdot)=\sum_{\bi}\mu_{\bi}\delta_{w_{\bi}(x)}$, so $B$ is a Borel subset of
$\cU\times\cX$, which is Polish by part~(iii). Hence $\cA$ is the complement in $\cU$ of the projection of
the Borel set $(\cU\times\cX)\setminus B$, that is, of an analytic set.
\end{remark}

\subsection{Uniformly ergodic sampling laws}\label{sec:unifergodic}
Quantitative information requires a rate, and a rate requires a metric in which the Markov operator
contracts. Contraction of the maps in $\rho$ is one way to obtain it, but not the only one. In this
subsection $\cX$ is bounded with diameter $D<\infty$, so that $\cP_1(\cX)=\cP(\cX)$.

\begin{definition}[Compatible metric]\label{def:compatible}
A metric $d$ on $\cP(\cX)$ is \emph{compatible} with the IFS if
\begin{enumerate}
\item[\rm(d1)] $(\cP(\cX),d)$ is complete and $d\le1$;
\item[\rm(d2)] $W_1(\alpha,\beta)\le D\,d(\alpha,\beta)$ for all $\alpha,\beta\in\cP(\cX)$;
\item[\rm(d3)] $d(P_\mu^*\lambda,P_\eta^*\lambda)\le\frac12\|\mu-\eta\|_1$ for all $\mu,\eta\in\Delta_m$ and
$\lambda\in\cP(\cX)$.
\end{enumerate}
For a compatible metric $d$, $n\in\N$ and $\eps\in(0,1]$, we let $\cE_d(n,\eps)$ be the set of all
$\mu\in\Delta_m$ such that $P_\mu^*$ is nonexpansive for $d$ and
\[
d\big((P_\mu^*)^n\lambda,(P_\mu^*)^n\lambda'\big)\le(1-\eps)\,d(\lambda,\lambda')\qquad\text{for all }\lambda,\lambda'\in\cP(\cX).
\]
\end{definition}

Three compatible metrics will be used:
\begin{itemize}
\item the total variation distance $\dTV$, for which (d2) is (W4) and (d3) holds because
$(P_\mu^*\lambda-P_\eta^*\lambda)(B)=\sum_i(\mu_i-\eta_i)\lambda(w_i^{-1}(B))$ with
$\lambda(w_i^{-1}(B))\in[0,1]$ and $\sum_i(\mu_i-\eta_i)=0$;
\item on $\cX=[a,b]$, the Kolmogorov distance
$\dK(\alpha,\beta):=\sup_{y\in[a,b]}|\alpha([a,y])-\beta([a,y])|$, for which (d2) follows from
$W_1(\alpha,\beta)=\int_a^b|\alpha([a,y])-\beta([a,y])|\,\dd y$ \cite[Proposition~2.17]{Santambrogio2015} and
(d3) as for $\dTV$;
\item the normalized Wasserstein distance $\dW:=W_1/D$, for which (d3) is~\eqref{eq:mixture}.
\end{itemize}

\begin{proposition}[Uniform ergodicity and perturbation]\label{prop:unifergodic}
Let $\cX$ be bounded, let $d$ be a compatible metric, and let $\mu\in\cE_d(n,\eps)$.
\begin{enumerate}
\item[\rm(i)] $\mu\in\cA$; $\pi_\mu$ is the only invariant probability measure of $P_\mu$; and for all
$\lambda\in\cP(\cX)$ and $k\ge0$,
\[
W_1\big((P_\mu^*)^k\lambda,\pi_\mu\big)\le D\,(1-\eps)^{\lfloor k/n\rfloor}.
\]
\item[\rm(ii)] For every $\eta\in\Delta_m$ and every invariant probability measure $\pi'$ of $P_\eta$,
\begin{equation}\label{eq:perturbation}
W_1(\pi_\mu,\pi')\le D\,d(\pi_\mu,\pi')\le\frac{D\,n}{2\,\eps}\,\|\mu-\eta\|_1 .
\end{equation}
\end{enumerate}
\end{proposition}

\begin{proof}
(i) By the Banach fixed point theorem in the complete space $(\cP(\cX),d)$, the map $(P_\mu^*)^n$ has a
unique fixed point $\pi$. Since $P_\mu^*\pi$ is also fixed by $(P_\mu^*)^n$, we get $P_\mu^*\pi=\pi$, and
every invariant measure of $P_\mu$ is a fixed point of $(P_\mu^*)^n$, hence equal to $\pi$. Writing
$k=qn+r$ with $0\le r<n$ and using $d\le1$,
$d((P_\mu^*)^k\lambda,\pi)=d((P_\mu^*)^{qn}(P_\mu^*)^r\lambda,(P_\mu^*)^{qn}\pi)\le(1-\eps)^q$. The claim follows
from (d2) and Lemma~\ref{lem:attraction}.

(ii) Since $\pi_\mu=(P_\mu^*)^n\pi_\mu$ and $\pi'=(P_\eta^*)^n\pi'=P_\eta^*\pi'$,
\begin{align*}
d(\pi_\mu,\pi')&\le d\big((P_\mu^*)^n\pi_\mu,(P_\mu^*)^n\pi'\big)+d\big((P_\mu^*)^n\pi',\pi'\big)\\
&\le(1-\eps)\,d(\pi_\mu,\pi')+\sum_{j=0}^{n-1}d\big((P_\mu^*)^{j+1}\pi',(P_\mu^*)^{j}\pi'\big).
\end{align*}
By nonexpansiveness of $P_\mu^*$ and (d3), each term of the sum is at most
$d(P_\mu^*\pi',\pi')=d(P_\mu^*\pi',P_\eta^*\pi')\le\frac12\|\mu-\eta\|_1$. Rearranging gives the second
inequality in~\eqref{eq:perturbation}; the first is (d2).
\end{proof}

Only the reference law $\mu$ needs to be uniformly ergodic in~\eqref{eq:perturbation}: the bound controls
\emph{all} invariant measures of all nearby kernels, admissible or not. This is the quantitative
counterpart of the closed graph property in Theorem~\ref{thm:topology}(i).

\subsection{Support-determined certificates and the face lemma}\label{sec:faces}
The decomposition $\Delta_m=\bigsqcup_{\emptyset\neq S\subseteq[m]}\Delta_S^\circ$ of the simplex into its
relatively open faces is the natural coordinate system for conditions that depend on which maps are
used but not on how often.

\begin{lemma}[Face lemma]\label{lem:face}
Let $\sF$ be a family of nonempty subsets of $[m]$ that is an \emph{up-set}: $S\in\sF$ and
$S\subseteq T\subseteq[m]$ imply $T\in\sF$. Let $\cF:=\bigcup_{S\in\sF}\Delta_S^\circ=\{\mu\in\Delta_m:\supp\mu\in\sF\}$.
Then:
\begin{enumerate}
\item[\rm(i)] $\cF$ is convex.
\item[\rm(ii)] $\Delta_m\setminus\cF=\bigcup_{S\notin\sF}\Delta_S$ is a finite union of closed faces; in
particular $\cF$ is relatively open in $\Delta_m$.
\item[\rm(iii)] If $\sF\neq\emptyset$, then $\cF\supseteq\Delta_{[m]}^\circ$; thus $\cF$ is dense in $\Delta_m$
and its complement lies in the relative boundary of the simplex.
\item[\rm(iv)] $\cF$ is closed if and only if $\cF\in\{\emptyset,\Delta_m\}$, and $\cF=\Delta_m$ if and
only if $\{i\}\in\sF$ for every $i\in[m]$.
\end{enumerate}
\end{lemma}

\begin{proof}
(i) If $\mu\in\Delta_S^\circ$ and $\eta\in\Delta_T^\circ$ with $S,T\in\sF$, then for $t\in(0,1)$ the support of
$(1-t)\mu+t\eta$ is $S\cup T\supseteq S$, which belongs to $\sF$. (ii) The complement of $\cF$ is
$\bigcup_{S\notin\sF}\Delta_S^\circ$. If $S\notin\sF$ and $\emptyset\ne T\subseteq S$, then $T\notin\sF$ because
$\sF$ is an up-set, so $\Delta_S=\bigcup_{\emptyset\ne T\subseteq S}\Delta_T^\circ$ is contained in the
complement. (iii) A nonempty up-set contains $[m]$. (iv) The simplex is connected and $\cF$ is open, so
$\cF$ is closed only if it is empty or everything; the last claim follows from the up-set property.
\end{proof}

We now exhibit two up-sets that certify admissibility.

\begin{definition}[Point-accessible supports]\label{def:pa}
A nonempty $S\subseteq[m]$ is \emph{point-accessible} if there are $n\in\N$ and $z\in\cX$ such that for
every $x\in\cX$ there is a word $\bi\in S^n$ with $w_{\bi}(x)=z$. We call $(n,z)$ a \emph{witness}, write
$\sD$ for the family of point-accessible supports, and set $\cD:=\{\mu\in\Delta_m:\supp\mu\in\sD\}$.
\end{definition}

\begin{theorem}[Doeblin region]\label{thm:doeblin}
Let Assumption~\ref{ass:standing} hold.
\begin{enumerate}
\item[\rm(i)] $\sD$ is an up-set. Hence $\cD$ is convex and relatively open in $\Delta_m$, its complement
is a union of closed faces, and $\cD$ is either empty or dense in $\Delta_m$.
\item[\rm(ii)] $\mu\in\cD$ if and only if $P_\mu$ satisfies the Doeblin condition: there are $n_0\in\N$,
$\eps\in(0,1]$ and $\nu_0\in\cP(\cX)$ such that
\begin{equation}\label{eq:Doeblin}
P_\mu^{n_0}(x,B)\ge\eps\,\nu_0(B)\qquad\text{for all }x\in\cX,\ B\in\cB(\cX).
\end{equation}
\item[\rm(iii)] Let $\cX$ be bounded. If $S\in\sD$ has witness $(n,z)$ and $\supp\mu\supseteq S$, then
$\mu\in\cE_{\dTV}(n,\eps)$ with $\eps=(\min_{i\in S}\mu_i)^n$. Consequently $\cD\subseteq\cA$, and for
$\mu\in\cD$ the conclusions of Proposition~\ref{prop:unifergodic} hold with $d=\dTV$.
\end{enumerate}
\end{theorem}

\begin{proof}
(i) If $S\subseteq T$ then $S^n\subseteq T^n$, so a witness for $S$ is a witness for $T$; apply
Lemma~\ref{lem:face}.

(ii) Let $S:=\supp\mu\in\sD$ with witness $(n,z)$. For each $x$ there is $\bi\in S^n$ with $w_{\bi}(x)=z$,
so $P_\mu^n(x,\{z\})\ge\mu_{\bi}\ge(\min_{i\in S}\mu_i)^n=:\eps$, which is~\eqref{eq:Doeblin} with
$\nu_0=\delta_z$. Conversely, assume~\eqref{eq:Doeblin}. For any fixed $x_1$ the measure
$P_\mu^{n_0}(x_1,\cdot)=\sum_{\bi}\mu_{\bi}\delta_{w_{\bi}(x_1)}$ is carried by a finite set, so $\nu_0$ is
carried by the same finite set and has an atom $z$. Then $P_\mu^{n_0}(x,\{z\})\ge\eps\nu_0(\{z\})>0$ for
every $x$, i.e.\ for every $x$ there is a word $\bi$ of length $n_0$ with $\mu_{\bi}>0$ and
$w_{\bi}(x)=z$. Thus $(n_0,z)$ is a witness for $\supp\mu$.

(iii) Total variation is nonexpansive under every Markov operator. Let $\lambda\ne\lambda'$ and
$c:=\dTV(\lambda,\lambda')$. By the Jordan decomposition, $\lambda-\lambda'=c\,(\alpha-\alpha')$ for two
probability measures $\alpha,\alpha'$. As in (ii), $(P_\mu^*)^n\alpha\ge\eps\delta_z$ and
$(P_\mu^*)^n\alpha'\ge\eps\delta_z$, so
$(P_\mu^*)^n\alpha-(P_\mu^*)^n\alpha'$ is the difference of two nonnegative measures of mass $1-\eps$, whence
$\dTV((P_\mu^*)^n\lambda,(P_\mu^*)^n\lambda')\le c\,(1-\eps)$. Therefore $\mu\in\cE_{\dTV}(n,\eps)$.
\end{proof}

Part~(ii) shows that for an IFS the Doeblin condition is a purely combinatorial property of the maps
indexed by the support of $\mu$: some point must be reachable from everywhere in a common number of
steps. For injective maps on an infinite space this never happens, since a word sends at most one point
to $z$; the Doeblin region is tailored to maps that merge trajectories, such as maps with saturation.

Boundedness of $\cX$ in part~(iii) cannot be dropped: on an unbounded space the Doeblin condition still
gives a unique invariant probability measure and convergence in total variation, but not a finite first
moment.

\begin{example}[Doeblin without admissibility]\label{ex:doeblin-unbounded}
Let $\cX=[0,\infty)$, $w_1\equiv0$ and $w_2(x)=2x+1$; both maps are Lipschitz. The support $\{1\}$ is
point-accessible and $\{2\}$ is not, so $\cD=\Delta_2\setminus\{e_2\}$. For $\mu=(1-t,t)$ with $t<1$ the
chain is reset to $0$ with probability $1-t$, and $j$ steps after the last reset it is at $2^j-1$; its
unique invariant probability measure is the law of $2^G-1$ with $\PP(G=j)=(1-t)t^j$, $j\ge0$, which has
finite mean if and only if $t<\frac12$. Hence $\Inv(\mu)=\emptyset$ for $t\ge\frac12$, while
$\theta_1(\mu)=2t<1$ for $t<\frac12$, and Proposition~\ref{prop:contraction} below gives
$\cA=\{\mu\in\Delta_2:\mu_2<\frac12\}$, which is not a union of relatively open faces.
\end{example}

For order-preserving maps on an interval, a weaker requirement suffices.

\begin{definition}[Splitting supports]\label{def:splitting}
Let $\cX=[a,b]\subset\R$ and let every $w_i$ be nondecreasing. A nonempty $S\subseteq[m]$ is
\emph{splitting} if there are $n\in\N$ and words $\bi,\bj\in S^n$ such that
\begin{equation}\label{eq:splitting}
w_{\bi}(b)\le w_{\bj}(a).
\end{equation}
We write $\sM$ for the family of splitting supports and $\cM:=\{\mu\in\Delta_m:\supp\mu\in\sM\}$.
\end{definition}

Since $w_{\bi}$ and $w_{\bj}$ are nondecreasing, \eqref{eq:splitting} says that $w_{\bi}$ maps the whole
interval to the left of the point $x_0:=w_{\bj}(a)$ and $w_{\bj}$ maps it to the right of $x_0$. This is the
splitting condition of Dubins and Freedman~\cite{DubinsFreedman1966}; see
also~\cite{BhattacharyaLee1988,BhattacharyaMajumdar2007}. A point-accessible support of monotone maps is
splitting (take $\bi,\bj$ with $w_{\bi}(b)=z=w_{\bj}(a)$), so $\sD\subseteq\sM$ in this setting.

\begin{theorem}[Splitting region]\label{thm:splitting}
Let $\cX=[a,b]$ and let $w_1,\dots,w_m$ be continuous and nondecreasing.
\begin{enumerate}
\item[\rm(i)] $\sM$ is an up-set. Hence $\cM$ is convex and relatively open in $\Delta_m$, its complement
is a union of closed faces, and $\cM$ is either empty or dense in $\Delta_m$.
\item[\rm(ii)] If $S\in\sM$ with words $\bi,\bj\in S^n$ as in~\eqref{eq:splitting} and $\supp\mu\supseteq S$,
then $\mu\in\cE_{\dK}(n,\chi)$ with $\chi=\min\{\mu_{\bi},\mu_{\bj}\}\ge(\min_{i\in S}\mu_i)^n$. Consequently
$\cD\subseteq\cM\subseteq\cA$, and for $\mu\in\cM$ the conclusions of Proposition~\ref{prop:unifergodic}
hold with $d=\dK$ and $D=b-a$.
\end{enumerate}
\end{theorem}

\begin{proof}
(i) is as in Theorem~\ref{thm:doeblin}. For (ii), let $g:[a,b]\to[a,b]$ be nondecreasing and
$y\in[a,b]$. The set $g^{-1}([a,y])$ is an initial segment of $[a,b]$: it is empty, or all of $[a,b]$,
or of the form $[a,c]$ or $[a,c)$. In each case
\begin{equation}\label{eq:initial}
\big|\lambda(g^{-1}([a,y]))-\lambda'(g^{-1}([a,y]))\big|\le\dK(\lambda,\lambda'),
\end{equation}
with left-hand side equal to zero in the first two cases; for $[a,c)$ one uses that
$\lambda([a,c))=\lim_{y'\uparrow c}\lambda([a,y'])$.
Applying~\eqref{eq:initial} to $g=w_i$ and averaging shows that $P_\mu^*$ is nonexpansive for $\dK$. Now fix
$y$ and expand
\[
(P_\mu^*)^n\lambda([a,y])-(P_\mu^*)^n\lambda'([a,y])
=\sum_{\mathbf k\in[m]^n}\mu_{\mathbf k}\Big(\lambda\big(w_{\mathbf k}^{-1}([a,y])\big)-\lambda'\big(w_{\mathbf k}^{-1}([a,y])\big)\Big).
\]
If $y<w_{\bj}(a)$, then $w_{\bj}^{-1}([a,y])=\emptyset$ and the term $\mathbf k=\bj$ vanishes. If
$y\ge w_{\bj}(a)$, then $y\ge w_{\bi}(b)$ by~\eqref{eq:splitting}, so $w_{\bi}^{-1}([a,y])=[a,b]$ and the term
$\mathbf k=\bi$ vanishes. By~\eqref{eq:initial} the remaining terms sum to at most
$(1-\chi)\,\dK(\lambda,\lambda')$. Taking the supremum over $y$ gives $\mu\in\cE_{\dK}(n,\chi)$, and
completeness of $(\cP([a,b]),\dK)$ is elementary.
\end{proof}

The classical regime fits the same scheme, with a half-space in place of a union of faces.

\begin{proposition}[Contraction on average]\label{prop:contraction}
Let Assumption~\ref{ass:standing} hold and suppose that every $w_i$ is Lipschitz. For $n\ge1$ put
$\theta_n(\mu):=\sum_{\bi\in[m]^n}\mu_{\bi}\Lip(w_{\bi})$.
\begin{enumerate}
\item[\rm(i)] If $\theta_n(\mu)<1$ for some $n$, then $\mu\in\cA$.
\item[\rm(ii)] The set $\cC:=\{\mu\in\Delta_m:\theta_1(\mu)<1\}$ is the intersection of $\Delta_m$ with an
open half-space; it is convex and relatively open, and $\cC=\Delta_m$ if and only if $\max_i\Lip(w_i)<1$.
\item[\rm(iii)] If $\cX$ is bounded and $\theta_1(\mu)<1$, then $\mu\in\cE_{\dW}(1,1-\theta_1(\mu))$.
\end{enumerate}
\end{proposition}

\begin{proof}
Coupling $(P_\mu^*)^n\lambda$ and $(P_\mu^*)^n\lambda'$ through the same word and an optimal coupling of
$(\lambda,\lambda')$ gives $W_1((P_\mu^*)^n\lambda,(P_\mu^*)^n\lambda')\le\theta_n(\mu)W_1(\lambda,\lambda')$. If
$\theta_n(\mu)<1$, the Banach fixed point theorem in $(\cP_1(\cX),W_1)$ gives a unique fixed point $\pi$ of
$(P_\mu^*)^n$, which is invariant for $P_\mu$ as in the proof of Proposition~\ref{prop:unifergodic}, and
\[
W_1\big((P_\mu^*)^{qn+r}\lambda,\pi\big)\le\theta_n(\mu)^q\max_{r'<n}W_1\big((P_\mu^*)^{r'}\lambda,\pi\big)\xrightarrow[q\to\infty]{}0 .
\]
This proves (i). Parts (ii) and (iii) are immediate since $\theta_1(\mu)=\sum_i\mu_i\Lip(w_i)$ is linear.
\end{proof}

In summary, when $\cX$ is bounded,
\[
\cD\cup\cM\cup\cC\ \subseteq\ \cA ,
\]
where $\cM$ is defined for monotone maps on an interval, and each of the three certified regions is
convex and relatively open. The first two are insensitive to the size of the Lipschitz constants and
are the relevant ones when contraction is not available.

\subsection{Finite state spaces: a complete description}\label{sec:finite}

\begin{theorem}[Finite state spaces]\label{thm:finite}
Let $\cX$ be finite. Then $\cA=\cAu=\cD$. Consequently:
\begin{enumerate}
\item[\rm(i)] Whether $\mu$ is admissible depends only on $\supp\mu$, and the family of admissible supports
is an up-set.
\item[\rm(ii)] $\cA$ is convex and relatively open in $\Delta_m$, and $\Delta_m\setminus\cA$ is a union of
closed faces of the simplex.
\item[\rm(iii)] $\cA\neq\emptyset$ if and only if $[m]$ is point-accessible, in which case $\cA$ contains the
relative interior of $\Delta_m$. Moreover $\cA$ is closed if and only if $\cA\in\{\emptyset,\Delta_m\}$, and
$\cA=\Delta_m$ if and only if every $w_i$ has a constant iterate.
\item[\rm(iv)] On $\cA$, the map $\mu\mapsto\pi_\mu$ is a rational function of $\mu$ whose denominator does
not vanish on $\cA$.
\end{enumerate}
\end{theorem}

\begin{proof}
By Theorem~\ref{thm:doeblin}(iii) and Proposition~\ref{prop:unifergodic}(i), $\cD\subseteq\cAu\subseteq\cA$.
Conversely let $\mu\in\cA$ and pick $z$ with $\pi_\mu(\{z\})>0$. On a finite space, $W_1$-convergence of
$P_\mu^k(x,\cdot)$ to $\pi_\mu$ implies $P_\mu^k(x,\{z\})\to\pi_\mu(\{z\})$, so for each $x$ there is $k_x$ with
$P_\mu^k(x,\{z\})>0$ for all $k\ge k_x$. With $n:=\max_xk_x$, every $x$ admits a word $\bi\in(\supp\mu)^n$
with $w_{\bi}(x)=z$; thus $\supp\mu\in\sD$. Statements (i)--(iii) now follow from
Theorem~\ref{thm:doeblin}(i) and Lemma~\ref{lem:face}; for the last claim of (iii), $\{i\}\in\sD$ means
precisely that some iterate $w_i^n$ is constant. For (iv), $\pi_\mu$ is the unique solution of the linear
system $M(\mu)\pi=b$ obtained by stacking $(I-P_\mu^\top)\pi=0$ and $\one^\top\pi=1$; the matrix $M(\mu)$ is
affine in $\mu$ and has full column rank for $\mu\in\cA$, so
$\pi_\mu=(M(\mu)^\top M(\mu))^{-1}M(\mu)^\top b$ is rational in $\mu$, with denominator
$\det(M(\mu)^\top M(\mu))\neq0$ on $\cA$.
\end{proof}

\begin{remark}\label{rem:finite}
(a) In the language of finite Markov chains, $S$ is point-accessible if and only if the directed graph
with edges $x\to w_i(x)$, $i\in S$, has exactly one closed communicating class and this class is
aperiodic \cite[Chapter~4]{Seneta2006}; equivalently, $P_\mu$ is stochastic, indecomposable and aperiodic
in the sense of Wolfowitz~\cite{Wolfowitz1963} for one, hence every, $\mu\in\Delta_S^\circ$. Admissibility of
a support can therefore be decided by a graph search.

(b) The set $\cU$ has the same facial structure: $\mu\in\cU$ if and only if the graph of $\supp\mu$ has
exactly one closed class. This family is again an up-set, because a set closed under the maps of
$T\supseteq S$ is closed under the maps of $S$, so that enlarging the support cannot increase the number
of minimal closed sets.

(c) Figure~\ref{fig:geometry}(a,b) shows two instances with $m=3$. In (a) the maps are those of the
storage model with hold of Example~\ref{ex:hold}, restricted to the lattice $\{0,0.1,\dots,1\}$: the supports $\{1\}$ and $\{2\}$ are
point-accessible and $\{3\}$ is not, so $\cA=\Delta_3\setminus\{e_3\}$. In (b), $\cX=\{1,2,3\}$, $w_1=\id$,
$w_2$ exchanges $1$ and $2$ and fixes $3$, and $w_3\equiv1$: a support is point-accessible if and only
if it contains $3$, so $\cA=\Delta_3\setminus\Delta_{\{1,2\}}$ is the simplex minus a closed edge.
\end{remark}

\begin{figure}[t]
\centering
\includegraphics[width=\textwidth]{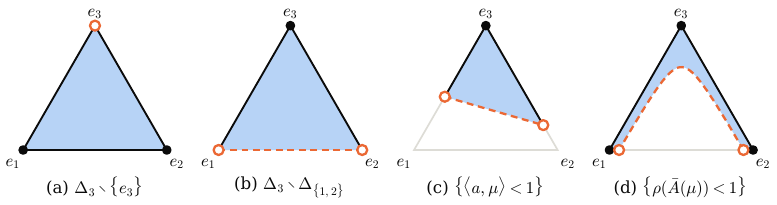}
\caption{Four admissible sets in the simplex $\Delta_3$ (shaded; solid boundary pieces and filled
vertices belong to $\cA$, dashed pieces and hollow points do not). (a), (b): finite state spaces,
Remark~\ref{rem:finite}(c). (c): homotheties with $a=(1.6,1.2,0.2)$, Theorem~\ref{thm:affine}(iv). (d): the
entrywise nonnegative linear system of Example~\ref{ex:nonconvex}, for which $\cA$ is connected but not
convex; its intersection with the edge $\Delta_{\{1,2\}}$ is disconnected.}
\label{fig:geometry}
\end{figure}

\subsection{Affine systems: spectral description}\label{sec:affine}
Let $\cX=\R^d$ with a norm $\|\cdot\|$ and $w_i(x)=A_ix+b_i$ with $A_i\in\R^{d\times d}$ and $b_i\in\R^d$. Put
\[
\bar A(\mu):=\sum_{i=1}^m\mu_iA_i,\qquad\bar b(\mu):=\sum_{i=1}^m\mu_ib_i,\qquad
\varrho_1(\mu):=\lim_{k\to\infty}\Big(\sum_{\bi\in[m]^k}\mu_{\bi}\,\|A_{\bi}\|\Big)^{1/k},
\]
where $A_{\bi}:=A_{i_k}\cdots A_{i_1}$ and $\|A_{\bi}\|$ is the operator norm. The limit exists and equals
the infimum over $k$ by submultiplicativity, and it does not depend on the norm; $\varrho_1$ is the
$p$-radius of the weighted family $(A_i,\mu_i)$ for $p=1$~\cite{JungersProtasov2011}. Since
$\|\bar A(\mu)^k\|=\|\sum_{\bi}\mu_{\bi}A_{\bi}\|\le\sum_{\bi}\mu_{\bi}\|A_{\bi}\|$, we always have
$\rho(\bar A(\mu))\le\varrho_1(\mu)$.

\begin{theorem}[Affine systems]\label{thm:affine}
For the affine IFS $w_i(x)=A_ix+b_i$ on $\R^d$:
\begin{enumerate}
\item[\rm(i)] $\{\mu:\varrho_1(\mu)<1\}\subseteq\cA\subseteq\{\mu:\rho(\bar A(\mu))<1\}$, and for
$\mu\in\cA$ the mean of $\pi_\mu$ is $\bar x(\mu)=(I-\bar A(\mu))^{-1}\bar b(\mu)$.
\item[\rm(ii)] If $b_i=0$ for all $i$, then $\cA=\{\mu:\varrho_1(\mu)<1\}$.
\item[\rm(iii)] If all matrices $A_i$ are entrywise nonnegative, then $\varrho_1(\mu)=\rho(\bar A(\mu))$ and
\[
\cA=\big\{\mu\in\Delta_m:\rho(\bar A(\mu))<1\big\},
\]
which is relatively open in $\Delta_m$.
\item[\rm(iv)] If $A_i=a_iI$ with $a_i\ge0$, then, with $a:=(a_1,\dots,a_m)$,
\[
\cA=\big\{\mu\in\Delta_m:\langle a,\mu\rangle<1\big\}
\]
is the intersection of the simplex with an open half-space. It is convex and relatively open; it is
closed if and only if $\max_ia_i<1$ (then $\cA=\Delta_m$) or $\min_ia_i\ge1$ (then $\cA=\emptyset$).
\end{enumerate}
\end{theorem}

\begin{proof}
(i) Since $\Lip(w_{\bi})=\|A_{\bi}\|$, the condition $\varrho_1(\mu)<1$ means $\theta_n(\mu)<1$ for some $n$,
and Proposition~\ref{prop:contraction}(i) gives $\mu\in\cA$. For the second inclusion let $\mu\in\cA$ and
let $X_k^x$ denote the chain started at $x$. Then $X_k^x-X_k^0=A_{\bI}x$ where $\bI=(I_1,\dots,I_k)$, so
$\E X_k^x-\E X_k^0=\bar A(\mu)^kx$ by independence. For any coupling $(Y,Y')$ of the laws of $X_k^x$ and
$X_k^0$ we have $\E\|Y-Y'\|\ge\|\E Y-\E Y'\|$, hence
\[
\|\bar A(\mu)^kx\|\le W_1\big((P_\mu^*)^k\delta_x,(P_\mu^*)^k\delta_0\big)\xrightarrow[k\to\infty]{}0
\qquad\text{for every }x\in\R^d,
\]
which forces $\rho(\bar A(\mu))<1$. Integrating $x$ against $\pi_\mu=P_\mu^*\pi_\mu$ gives
$\bar x(\mu)=\bar A(\mu)\bar x(\mu)+\bar b(\mu)$, and $I-\bar A(\mu)$ is invertible.

(ii) If all $b_i$ vanish, $\delta_0$ is invariant for every $\mu$. If $\mu\in\cA$, then $\pi_\mu=\delta_0$,
and therefore $\E\|A_{\bI}x\|=W_1((P_\mu^*)^k\delta_x,\delta_0)$ tends to zero for each $x$. Applying this to the vectors of a basis
gives $\sum_{\bi\in[m]^k}\mu_{\bi}\|A_{\bi}\|\to0$, so this sum is smaller than one for some $k$ and
$\varrho_1(\mu)<1$.

(iii) Use the $\ell_1$-norm on $\R^d$. For a nonnegative matrix $M$, the induced norm is the largest column
sum, so $\|M\|\le\one^\top M\one\le d\,\|M\|$. Hence
$\sum_{\bi\in[m]^k}\mu_{\bi}\|A_{\bi}\|\le\one^\top\bar A(\mu)^k\one\le d\,\|\bar A(\mu)^k\|$, and taking $k$-th
roots gives $\varrho_1(\mu)\le\rho(\bar A(\mu))$. With (i) this proves the identity. The set is open
because the spectral radius is continuous.

(iv) is the case $\bar A(\mu)=\langle a,\mu\rangle I$ of (iii). If $\min_ia_i<1\le\max_ia_i$, then $\cA$ is
a nonempty proper open subset of the connected set $\Delta_m$, hence not closed.
\end{proof}

\begin{example}[The admissible set is not closed]\label{ex:nonclosed}
Let $\cX=\R$, $w_1(x)=x+1$ and $w_2(x)=x/2$. By Theorem~\ref{thm:affine}(iv) with $a=(1,\frac12)$,
\[
\cA=\{\mu\in\Delta_2:\mu_1+\tfrac12\mu_2<1\}=\Delta_2\setminus\{e_1\}.
\]
At $\mu=e_1$ the chain is the translation $x\mapsto x+1$, which has no invariant probability measure:
an invariant $\pi$ would give the same mass $c$ to every interval $[j,j+1)$, $j\in\Z$, and $\sum_{j}c=1$ is
impossible. The laws $(1-\frac1j,\frac1j)\in\cA$ converge to $e_1\notin\cA$.
\end{example}

\begin{example}[A non-convex admissible set with a disconnected edge]\label{ex:nonconvex}
Let $\cX=\R^2$ and consider the linear maps $w_i(x)=A_ix$ with
\[
A_1=\begin{pmatrix}0&4\\0&0\end{pmatrix},\qquad A_2=\begin{pmatrix}0&0\\4&0\end{pmatrix},\qquad
A_3=\tfrac12\begin{pmatrix}1&0\\0&1\end{pmatrix}.
\]
The matrix $\bar A(\mu)$ has eigenvalues $\frac12\mu_3\pm4\sqrt{\mu_1\mu_2}$, so by
Theorem~\ref{thm:affine}(iii)
\[
\cA=\big\{\mu\in\Delta_3:\ \tfrac12\mu_3+4\sqrt{\mu_1\mu_2}<1\big\},
\]
shown in Figure~\ref{fig:geometry}(d). The vertices $e_1$ and $e_2$ are admissible (each of $A_1,A_2$ is
nilpotent) and their midpoint is not, so $\cA$ is not convex. It is connected, since the edges
$[e_1,e_3]$ and $[e_2,e_3]$ lie in $\cA$, but the IFS made of $w_1,w_2$ alone has the disconnected admissible set
$\{(1-t,t):t(1-t)<\frac1{16}\}=\{t<t_-\}\cup\{t>t_+\}$ with $t_\pm=\frac12\pm\frac{\sqrt3}{4}$. Mixing two
stable mechanisms can thus be destabilizing, as in the continuous-time switched systems
of~\cite{BenaimEtAl2014,LawleyMattinglyReed2014}. Here the instability is one of first moments: for
every $t\in(0,1)$ the chain driven by $w_1,w_2$ is absorbed at the origin almost surely (it survives only
as long as the two maps alternate), but $\E X_{2k}=(16\,t(1-t))^kX_0$, which does not tend to zero when
$t\in[t_-,t_+]$: convergence holds in probability and fails in $W_1$. The same argument applies to every
$\mu\in\Delta_3$, since $A_1A_3^jA_1=A_2A_3^jA_2=0$ for all $j$, so with weak attraction in place of
$W_1$-attraction the whole simplex would be admissible (Remark~\ref{rem:weak}). The non-convexity here is
a first-moment phenomenon; Example~\ref{ex:twolane} below shows that it is not the only one.
\end{example}

\subsection{Compact state spaces: no general regularity}\label{sec:pathology}
One might hope that the first-moment effect in Example~\ref{ex:nonconvex} is the only obstruction to
convexity and that, for compact $\cX$, the facial structure of Theorem~\ref{thm:finite} persists. It does
not.

\begin{example}[Two isolated vertices]\label{ex:twolane}
Let $u_1=(1,0)$, $u_2=(0,1)$ and
\[
\cX:=\{0\}\cup\Big\{\tfrac{1}{h+1}\,u_l:\ l\in\{1,2\},\ h\in\N_0\Big\}\subset\R^2
\]
with the $\ell_1$-distance; $\cX$ is compact and countable. We write $(l,h)$ for the point
$\frac1{h+1}u_l$ and call $h$ its height. Define
\begin{align*}
w_1(1,h)&=(1,h+1), & w_1(2,h)&=(1,\lfloor h/2\rfloor),\\
w_2(2,h)&=(2,h+1), & w_2(1,h)&=(2,\lfloor h/2\rfloor),
\end{align*}
and $w_1(0)=w_2(0)=0$. Both maps are Lipschitz (with constant $4$). We claim that
\[
\cA=\{e_1,e_2\}.
\]
\emph{The vertices are admissible.} After one application of $w_1$ every point other than $0$ lies on the
first branch, where $w_1$ raises the height by one; thus $\rho(w_1^k(x),0)\le1/k$ for all $x$ and
$W_1((w_1^k)_\#\lambda,\delta_0)\le1/k$. The same holds for $w_2$.

\emph{No other law is admissible.} Let $\mu=(1-t,t)$ with $t\in(0,1)$ and $\theta:=\min\{t,1-t\}$. The
origin is a common fixed point, so $\delta_0\in\Inv(\mu)$, and if $\mu$ were admissible we would have
$\pi_\mu=\delta_0$. Let $H_k$ be the height of $X_k$ for a chain started at $x\neq0$ with height $H_0$. From
a point of the first branch, the height increases by one with probability $1-t$ and is halved with
probability $t$; on the second branch the roles of $t$ and $1-t$ are exchanged. Hence
$\E[H_{k+1}\mid X_k]\le(1-\theta/2)H_k+1$ and $\sup_k\E H_k\le H_0+2/\theta$. By Jensen's inequality,
\[
W_1\big((P_\mu^*)^k\delta_x,\delta_0\big)=\E\Big[\frac1{H_k+1}\Big]\ge\frac1{H_0+1+2/\theta}\qquad\text{for all }k,
\]
so $\delta_0$ does not attract $\delta_x$, and $\mu\notin\cA$.
\end{example}

In this example each map sends everything to the common fixed point, yet any random mixture of the two
keeps the chain away from it: after a change of branch the height is halved. The admissible set is
closed, not relatively open, not convex and not connected, and the family of admissible supports
$\{\{1\},\{2\}\}$ is not an up-set. Together with Example~\ref{ex:nonclosed} and Remark~\ref{rem:finite}(c) this proves:

\begin{proposition}[No general regularity]\label{prop:noregularity}
Under Assumption~\ref{ass:standing} alone, the admissible set $\cA$ need not be closed, nor relatively
open, nor convex, nor connected, and each of these properties can fail on a compact state space with
Lipschitz maps.
\end{proposition}

The contrast between Proposition~\ref{prop:noregularity} and Theorems~\ref{thm:doeblin}--\ref{thm:affine}
delimits what a general theory can assert about $\cA$: regular geometry is a property of certified
regions and of structured classes of maps, not of admissibility as such. The existence theory of the next
section is designed to work without it.

\section{Existence of optimal sampling laws}\label{sec:existence}

We now turn to question (Q2): when does Problem~\eqref{eq:A} have a solution? Since $\cA$ may fail to be
closed, a minimizing sequence may converge to a point outside $\cA$. The idea of this section is to
analyze what the invariant measures do along such a sequence. Either their first moments blow up, and
then the objective blows up as well, or they converge to an invariant measure of the limiting kernel,
and then everything depends on whether this limit can also be produced by an admissible law.

Throughout this section, Assumption~\ref{ass:standing} is in force, $\nu\in\cP_1(\cX)$ is fixed,
$\cA\neq\emptyset$, and
\[
\Jstar:=\inf_{\mu\in\cA}J(\mu)\in[0,\infty).
\]

\begin{assumption}\label{ass:proper}
The metric space $(\cX,\rho)$ is proper: closed bounded subsets of $\cX$ are compact.
\end{assumption}

Assumption~\ref{ass:proper} holds for compact spaces and for closed subsets of $\R^d$. Neither
contractivity of the maps nor compactness of $\cX$ is required.

\subsection{Lower semicontinuity and the relaxed problem}

\begin{lemma}[Compactness and closedness along bounded sequences]\label{lem:closed}
Let Assumptions~\ref{ass:standing} and~\ref{ass:proper} hold. Let $\mu_n\to\mu$ in $\Delta_m$ and
$\pi_n\in\Inv(\mu_n)$ with $\sup_nM_1(\pi_n)<\infty$. Then $(\pi_n)$ is tight, every weak limit point $\pi$
of $(\pi_n)$ belongs to $\Inv(\mu)$, and if $\pi_{n_j}\to\pi$ weakly then
$W_1(\nu,\pi)\le\liminf_{j}W_1(\nu,\pi_{n_j})$.
\end{lemma}

\begin{proof}
By Markov's inequality, $\pi_n(\{x:\rho(x,x_\circ)>R\})\le R^{-1}\sup_nM_1(\pi_n)$, and closed balls are compact,
so $(\pi_n)$ is tight and Prokhorov's theorem applies. Let $\pi_{n_j}\to\pi$ weakly. By
Lemma~\ref{lem:basic}(ii), $\pi_{n_j}=P_{\mu_{n_j}}^*\pi_{n_j}\to P_\mu^*\pi$ weakly, hence $P_\mu^*\pi=\pi$.
Since $x\mapsto\rho(x,x_\circ)$ is continuous and nonnegative, $M_1(\pi)\le\liminf_jM_1(\pi_{n_j})<\infty$, so
$\pi\in\Inv(\mu)$. The last claim is (W2).
\end{proof}

\begin{proposition}[Lower semicontinuity of the objective]\label{prop:lsc}
Under Assumptions~\ref{ass:standing} and~\ref{ass:proper}, the function $J(\mu)=W_1(\nu,\pi_\mu)$ is lower
semicontinuous on $\cU$, and hence on $\cA$.
\end{proposition}

\begin{proof}
Let $\mu_n\to\mu$ in $\cU$ and $\ell:=\liminf_nJ(\mu_n)$. If $\ell=\infty$ there is nothing to prove;
otherwise pass to a subsequence along which $J(\mu_n)\to\ell$. By (W4),
$M_1(\pi_{\mu_n})\le M_1(\nu)+J(\mu_n)$ is bounded. By Lemma~\ref{lem:closed} a further subsequence converges
weakly to some $\pi\in\Inv(\mu)=\{\pi_\mu\}$, and $J(\mu)=W_1(\nu,\pi_\mu)\le\ell$.
\end{proof}

On a non-compact space $J$ need not be continuous on $\cA$, because $W_1$ is only lower semicontinuous
for weak convergence; Proposition~\ref{prop:lsc} is what the direct method needs.

\begin{definition}[Missing boundary, relaxed objective, limiting invariant measures]\label{def:relaxed}
The \emph{missing boundary} of the admissible set is $\mb:=\cl(\cA)\setminus\cA$. For $\mu_0\in\cl(\cA)$ let
\[
\uJ(\mu_0):=\inf\Big\{\liminf_{n\to\infty}J(\eta_n):\ \eta_n\in\cA,\ \eta_n\to\mu_0\Big\}\in[0,\infty],
\]
and let $\Lambda(\mu_0)$ be the set of all $\pi\in\cP(\cX)$ for which there is a sequence $\eta_n\in\cA$
with $\eta_n\to\mu_0$, $\sup_nM_1(\pi_{\eta_n})<\infty$ and $\pi_{\eta_n}\to\pi$ weakly.
\end{definition}

The function $\uJ$ is the lower semicontinuous envelope of $J$, extended to the closure of $\cA$; the set
$\Lambda(\mu_0)$ collects the invariant measures that the admissible laws leave behind at $\mu_0$.

\begin{theorem}[Relaxation and solvability]\label{thm:exact}
Let Assumptions~\ref{ass:standing} and~\ref{ass:proper} hold, let $\nu\in\cP_1(\cX)$ and
$\cA\neq\emptyset$.
\begin{enumerate}
\item[\rm(i)] $\uJ$ is lower semicontinuous on the compact set $\cl(\cA)$, $\uJ=J$ on $\cA$, and
\[
\min_{\mu\in\cl(\cA)}\uJ(\mu)=\inf_{\mu\in\cA}J(\mu)=\Jstar .
\]
\item[\rm(ii)] For every $\mu_0\in\mb$ we have $\Lambda(\mu_0)\subseteq\Inv(\mu_0)$, and if
$\uJ(\mu_0)<\infty$ there is $\pi\in\Lambda(\mu_0)$ with $W_1(\nu,\pi)\le\uJ(\mu_0)$. In particular
$\uJ(\mu_0)=+\infty$ whenever $\Lambda(\mu_0)=\emptyset$, and a fortiori whenever $\Inv(\mu_0)=\emptyset$.
If $\cX$ is compact, then $\Lambda(\mu_0)$ is nonempty and compact and
$\uJ(\mu_0)=\min_{\pi\in\Lambda(\mu_0)}W_1(\nu,\pi)$.
\item[\rm(iii)] The following are equivalent:
\begin{enumerate}
\item[\rm(a)] Problem~\eqref{eq:A} has a solution;
\item[\rm(b)] there is $\mu^\dagger\in\cA$ such that $J(\mu^\dagger)\le\uJ(\mu_0)$ for every $\mu_0\in\mb$;
\item[\rm(c)] some minimizing sequence of~\eqref{eq:A} has a cluster point in $\cA$.
\end{enumerate}
In that case the set of solutions of~\eqref{eq:A} is $\cA\cap\operatorname*{arg\,min}_{\cl(\cA)}\uJ$.
\end{enumerate}
\end{theorem}

\begin{proof}
(i) A diagonal argument shows that the infimum defining $\uJ(\mu_0)$ is attained along some sequence and
that $\uJ$ is lower semicontinuous. On $\cA$, the constant sequence gives $\uJ\le J$, and
Proposition~\ref{prop:lsc} gives $\uJ\ge J$. Since every value $\liminf_nJ(\eta_n)$ is at least $\Jstar$, we
have $\uJ\ge\Jstar$ on $\cl(\cA)$, while $\inf_{\cA}\uJ=\inf_{\cA}J=\Jstar$. A lower semicontinuous function
attains its infimum on a compact set.

(ii) The inclusion $\Lambda(\mu_0)\subseteq\Inv(\mu_0)$ is Lemma~\ref{lem:closed}. If $\uJ(\mu_0)<\infty$,
choose $\eta_n\in\cA$ with $\eta_n\to\mu_0$ and $J(\eta_n)\to\uJ(\mu_0)$. Then
$M_1(\pi_{\eta_n})\le M_1(\nu)+J(\eta_n)$ is bounded, and Lemma~\ref{lem:closed} yields a weak limit point
$\pi\in\Lambda(\mu_0)$ with $W_1(\nu,\pi)\le\uJ(\mu_0)$. If $\cX$ is compact, weak convergence and
$W_1$-convergence coincide and first moments are bounded, so $\Lambda(\mu_0)$ is the set of $W_1$-limit
points of $\pi_\eta$ as $\eta\to\mu_0$ in $\cA$; it is nonempty, closed, hence compact, and
$\uJ(\mu_0)\le\lim_nW_1(\nu,\pi_{\eta_n})=W_1(\nu,\pi)$ for every $\pi\in\Lambda(\mu_0)$ realized by $(\eta_n)$.

(iii) (a)$\Rightarrow$(b): a solution $\mu^\dagger$ satisfies $J(\mu^\dagger)=\Jstar\le\uJ(\mu_0)$ by (i).
(b)$\Rightarrow$(a): let $\hat\mu$ minimize $\uJ$ over $\cl(\cA)$. If $\hat\mu\in\cA$ then
$J(\hat\mu)=\uJ(\hat\mu)=\Jstar$. Otherwise $\hat\mu\in\mb$ and $J(\mu^\dagger)\le\uJ(\hat\mu)=\Jstar$. In both
cases the infimum is attained. (a)$\Rightarrow$(c) by taking a constant sequence, and (c)$\Rightarrow$(a)
by Proposition~\ref{prop:lsc}: if $\eta_n\to\bar\mu\in\cA$ along a minimizing sequence, then
$J(\bar\mu)\le\liminf_nJ(\eta_n)=\Jstar$. The description of the solution set follows from (i).
\end{proof}

Theorem~\ref{thm:exact} reduces the existence question to a comparison at the missing boundary. The
relaxed problem $\min_{\cl(\cA)}\uJ$ always has a solution and has the same value as~\eqref{eq:A};
Problem~\eqref{eq:A} itself has a solution exactly when the relaxation does not strictly prefer a
missing boundary point. Part~(iii) is a reformulation rather than a test, since condition~(b) involves
$\uJ$ on $\mb$. What makes it usable is part~(ii): $\uJ(\mu_0)$ is the distance from $\nu$ to the limiting
invariant measures at $\mu_0$, which are invariant measures of the limiting kernel, and this turns~(b) into
the verifiable sufficient conditions of the next subsection.

\begin{remark}\label{rem:feasible}
The proofs of Proposition~\ref{prop:lsc} and Theorem~\ref{thm:exact} use only that the feasible set is
contained in $\cU$, and the same holds for Theorem~\ref{thm:universal} below. These three results
therefore remain valid, with the same proofs, if $\cA$ is replaced by any nonempty set $\cF\subseteq\cU$,
for instance $\cU$ itself, $\cAu$, or one of the certified regions of Section~\ref{sec:geometry}; the
examples and the results on homotheties in the rest of this section use the specific form of $\cA$.
\end{remark}

\subsection{Solvability for every target}

\begin{theorem}[Universal solvability]\label{thm:universal}
Let Assumptions~\ref{ass:standing} and~\ref{ass:proper} hold and let $\cA\neq\emptyset$.
\begin{enumerate}
\item[\rm(i)] \emph{(Sufficient condition.)} If
\begin{equation}\label{eq:reach}
\Lambda(\mu_0)\subseteq\PiA\qquad\text{for every }\mu_0\in\mb ,
\end{equation}
then Problem~\eqref{eq:A} has a solution for every target $\nu\in\cP_1(\cX)$. In
particular, this is the case
\begin{enumerate}
\item[\rm(a)] if $\cA$ is closed, for instance if $\cA=\Delta_m$; or
\item[\rm(b)] if $\Inv(\mu_0)=\emptyset$ for every $\mu_0\in\mb$; then, in addition, $J(\eta_n)\to\infty$ for
every sequence $\eta_n\in\cA$ converging to a point of $\mb$.
\end{enumerate}
\item[\rm(ii)] \emph{(Necessary and sufficient condition for compact $\cX$.)} If $\cX$ is compact, the
following are equivalent:
\begin{enumerate}
\item[\rm(1)] Problem~\eqref{eq:A} has a solution for every target $\nu\in\cP(\cX)$;
\item[\rm(2)] the reachable set $\PiA=\{\pi_\mu:\mu\in\cA\}$ is closed in $(\cP(\cX),W_1)$;
\item[\rm(3)] condition~\eqref{eq:reach} holds.
\end{enumerate}
\end{enumerate}
\end{theorem}

\begin{proof}
(i) Let $\hat\mu$ minimize $\uJ$ over $\cl(\cA)$; then $\uJ(\hat\mu)=\Jstar<\infty$ by
Theorem~\ref{thm:exact}(i). If $\hat\mu\in\cA$ we are done. Otherwise Theorem~\ref{thm:exact}(ii) provides
$\pi\in\Lambda(\hat\mu)$ with $W_1(\nu,\pi)\le\Jstar$, and by~\eqref{eq:reach} $\pi=\pi_{\mu'}$ for some
$\mu'\in\cA$, so $J(\mu')\le\Jstar$. Case~(a) is the case $\mb=\emptyset$. In case~(b),
$\Lambda(\mu_0)\subseteq\Inv(\mu_0)=\emptyset$, so~\eqref{eq:reach} holds trivially and $\uJ\equiv+\infty$ on
$\mb$ by Theorem~\ref{thm:exact}(ii), which is the stated divergence.

(ii) We first note that, for compact $\cX$,
\begin{equation}\label{eq:closure}
\cl_{W_1}(\PiA)=\PiA\cup\bigcup_{\mu_0\in\mb}\Lambda(\mu_0).
\end{equation}
Indeed, if $\pi_{\eta_n}\to\pi$ with $\eta_n\in\cA$, pass to a subsequence with $\eta_n\to\mu_0\in\cl(\cA)$;
if $\mu_0\in\cA$ then $\pi=\pi_{\mu_0}$ by Theorem~\ref{thm:topology}(iv), and otherwise
$\pi\in\Lambda(\mu_0)$. The converse inclusion is clear. Hence (2)$\Leftrightarrow$(3).
(2)$\Rightarrow$(1): a closed subset of the compact space $(\cP(\cX),W_1)$ is compact and
$\pi\mapsto W_1(\nu,\pi)$ is continuous. (1)$\Rightarrow$(2): if $\pi\in\cl_{W_1}(\PiA)\setminus\PiA$, then for
the target $\nu:=\pi$ we have $\Jstar=0$, while $J(\mu)=W_1(\pi,\pi_\mu)>0$ for every $\mu\in\cA$.
\end{proof}

The equivalence (1)$\Leftrightarrow$(2) is an instance of the elementary fact that a subset of a compact
metric space contains a nearest point to every point if and only if it is closed. The content specific to
the shaping problem is the identity~\eqref{eq:closure}, which identifies the points missing from $\PiA$ as
limiting invariant measures at the missing boundary, and condition~\eqref{eq:reach}, which can be verified
through the dynamics at $\mb$ alone, without knowledge of $\PiA$.

Condition~\eqref{eq:reach} says that whatever the admissible laws leave behind at the missing boundary
can also be produced from inside. Case~(b) is the mechanism seen in the Ornstein--Uhlenbeck illustration
of the Introduction: where the limiting kernel has no invariant measure of finite first moment, mass
escapes, and the missing boundary repels every minimizing sequence. No verification of a Lyapunov
condition or of a moment estimate is needed: non-existence of an invariant measure at the boundary is
enough. A second mechanism that produces~\eqref{eq:reach} is the presence of an idle mechanism.

\begin{lemma}[Laziness]\label{lem:lazy}
Suppose that $w_j=\id$ for some $j\in[m]$. Let $\eta\in\cA$ with $\eta_j=0$ and $s\in(0,1]$. Then
$\mu_s:=(1-s)e_j+s\,\eta\in\cA$ and $\pi_{\mu_s}=\pi_\eta$. In particular $J$ is constant on the half-open
segment $\{\mu_s:s\in(0,1]\}$.
\end{lemma}

\begin{proof}
Since $P_{\mu_s}^*=(1-s)\,\mathrm{Id}+s\,P_\eta^*$, we have
$(P_{\mu_s}^*)^k\lambda=\sum_{l=0}^k\binom kl s^l(1-s)^{k-l}(P_\eta^*)^l\lambda$, and by (W3)
\[
W_1\big((P_{\mu_s}^*)^k\lambda,\pi_\eta\big)\le\sum_{l=0}^k\binom kl s^l(1-s)^{k-l}\,W_1\big((P_\eta^*)^l\lambda,\pi_\eta\big).
\]
The sequence $W_1((P_\eta^*)^l\lambda,\pi_\eta)$ is bounded and tends to zero, and the binomial weights of any
fixed finite set of indices $l$ tend to zero as $k\to\infty$. Hence the right-hand side tends to zero, and
Lemma~\ref{lem:attraction} applies.
\end{proof}

\subsection{Two examples}

\begin{example}[Non-existence on a two-point space]\label{ex:nonexistence}
Let $\cX=\{0,1\}\subset\R$, let $w_1$ be the swap $w_1(x)=1-x$ and $w_2\equiv0$, and write $\mu=(1-t,t)$. For
$t>0$ the support of $\mu$ is point-accessible (with witness $n=1$, $z=0$), so $\mu\in\cA$ by
Theorem~\ref{thm:finite}, and a direct computation gives
\[
\pi_\mu=\Big(\frac{1}{2-t},\ \frac{1-t}{2-t}\Big).
\]
For $t=0$ the chain is the deterministic swap: its unique invariant law is $(\frac12,\frac12)$, but
$(P_{e_1}^*)^k\delta_0$ oscillates between $\delta_0$ and $\delta_1$. Hence $\cA=\Delta_2\setminus\{e_1\}$,
$\cU=\Delta_2$ and $\Lambda(e_1)=\{(\frac12,\frac12)\}\not\subseteq\PiA$. For the target $\nu=(\frac12,\frac12)$,
\[
J(\mu)=\Big|\frac{1-t}{2-t}-\frac12\Big|=\frac{t}{2(2-t)}>0\quad\text{for }t\in(0,1],\qquad\inf_{\cA}J=0 ,
\]
and Problem~\eqref{eq:A} has no solution. The target is the long-run occupation law of the periodic
chain, which every admissible law approximates and none attains. If the attraction requirement is
dropped and one minimizes over $\cU$ instead (Remark~\ref{rem:feasible}), the minimizer is $e_1$.
\end{example}

\begin{example}[Storage with an idle mode: a non-closed, non-coercive, solvable case]\label{ex:hold}
Let $\cX=[0,1]$ be the normalized level of a storage device with the three mechanisms
\[
w_1(x)=\min\{1,x+0.3\},\qquad w_2(x)=\max\{0,x-0.2\},\qquad w_3(x)=x
\]
(charge, discharge, hold).
Every map is nondecreasing and has Lipschitz constant exactly one, so $\theta_1(\mu)=1$ for all $\mu$ and
the contraction region $\cC$ is empty. However $w_1^4\equiv1$ and $w_2^5\equiv0$, so $\{1\}$ and $\{2\}$ are
point-accessible, and by Theorem~\ref{thm:doeblin} every $\mu\neq e_3$ is admissible. At $e_3$ the kernel
is the identity and every probability measure is invariant. Hence
\[
\cA=\cD=\Delta_3\setminus\{e_3\},\qquad\mb=\{e_3\},
\]
as in Figure~\ref{fig:geometry}(a). Since $\cX$ is compact, $J$ is bounded and there is no coercivity at
$e_3$. Nevertheless~\eqref{eq:reach} holds. Every $\eta\in\cA$ can be written as
$\eta=(1-s)e_3+s\,\eta'$ with $s\in(0,1]$ and $\eta'\in\Delta_{\{1,2\}}\subseteq\cA$, and
$\pi_\eta=\pi_{\eta'}$ by Lemma~\ref{lem:lazy}. Since $\Delta_{\{1,2\}}$ is compact and $\eta'\mapsto\pi_{\eta'}$ is
continuous on it (Theorem~\ref{thm:topology}(iv)), we get
$\Lambda(e_3)\subseteq\{\pi_{\eta'}:\eta'\in\Delta_{\{1,2\}}\}\subseteq\PiA$. By Theorem~\ref{thm:universal},
Problem~\eqref{eq:A} has a solution for every target. Solutions are never unique here: with $\mu^\dagger$,
the whole half-open segment from $\mu^\dagger$ towards $e_3$ consists of solutions, so the solution set
is not closed (Figure~\ref{fig:landscapes}(a)).
\end{example}

\subsection{Homotheties: unconditional existence}\label{sec:homothety}
We now show that for the class of Theorem~\ref{thm:affine}(iv) the two mechanisms of
Theorem~\ref{thm:universal} always combine. Let $\cX=\R^d$ with the Euclidean norm and
\begin{equation}\label{eq:homothety}
w_i(x)=a_ix+b_i,\qquad a_i\ge0,\quad b_i\in\R^d,\qquad i\in[m],
\end{equation}
so that $\cA=\{\mu\in\Delta_m:\langle a,\mu\rangle<1\}$. We call $\mu_0\in\Delta_m$ \emph{critical} if
$\langle a,\mu_0\rangle=1$, and write $\Fix(w_i)$ for the set of fixed points of $w_i$.

\begin{lemma}[Invariant measures at critical parameters]\label{lem:critical}
Let $\mu_0$ be critical for the system~\eqref{eq:homothety}. Then
\[
\Inv(\mu_0)=\Big\{\pi\in\cP_1(\R^d):\ \pi\big(F(\mu_0)\big)=1\Big\},\qquad
F(\mu_0):=\bigcap_{i\in\supp\mu_0}\Fix(w_i).
\]
In particular $\Inv(\mu_0)=\emptyset$ if the maps $w_i$, $i\in\supp\mu_0$, have no common fixed point, and
this is the case whenever $\bar b(\mu_0)\neq0$.
\end{lemma}

\begin{proof}
A measure carried by common fixed points is invariant. Conversely let $\pi\in\Inv(\mu_0)$ and run the
stationary chain $X_{k+1}=w_{I_{k+1}}(X_k)$ with $X_0\sim\pi$ independent of the i.i.d.\ sequence $(I_k)$ with
law $\mu_0$. With $\cF_k:=\sigma(X_0,I_1,\dots,I_k)$,
\[
\E[X_{k+1}\mid\cF_k]=\sum_i\mu_{0,i}(a_iX_k+b_i)=X_k+\bar b(\mu_0).
\]
Taking expectations and using stationarity and $\E\|X_0\|<\infty$ gives $\bar b(\mu_0)=0$. Thus $(X_k)$ is a
martingale, bounded in $L^1$ because $\E\|X_k\|=M_1(\pi)$, and by Doob's theorem it converges almost surely.
Hence $X_{k+1}-X_k\to0$ almost surely. By stationarity the law of $X_{k+1}-X_k$ does not depend on $k$, so
$X_1=X_0$ almost surely, that is, $\sum_i\mu_{0,i}\,\pi(\Fix(w_i))=1$. Therefore $\pi(\Fix(w_i))=1$ for each
$i\in\supp\mu_0$. Finally, a common fixed point $x$ satisfies $x=\sum_i\mu_{0,i}w_i(x)=x+\bar b(\mu_0)$.
\end{proof}

\begin{theorem}[Existence for homotheties]\label{thm:homothety}
Consider the system~\eqref{eq:homothety} on $\R^d$ and assume $\min_ia_i<1$, so that $\cA\neq\emptyset$. Then
Problem~\eqref{eq:A} has a solution for every target $\nu\in\cP_1(\R^d)$.
\end{theorem}

\begin{proof}
\emph{Step 1: identity maps can be discarded.} Let $I_0:=\{i:w_i=\id\}$ and $\Delta':=\Delta_{[m]\setminus I_0}$,
which is nonempty because some $a_i<1$. For $\eta\in\cA$ put $s:=\sum_{i\notin I_0}\eta_i$; then $s>0$,
since $\langle a,\eta\rangle<1$, and $\eta=(1-s)\iota+s\,\eta'$ with $\iota\in\Delta_{I_0}$ (if $s<1$) and
$\eta'\in\Delta'$. From $\langle a,\eta\rangle=(1-s)+s\langle a,\eta'\rangle<1$ we get
$\langle a,\eta'\rangle<1$, so $\eta'\in\cA$, and $\pi_{\eta'}$ is invariant for
$P_\eta^*=(1-s)\,\mathrm{Id}+s\,P_{\eta'}^*$, so $\pi_\eta=\pi_{\eta'}$. Consequently there is a minimizing
sequence $(\eta_n)$ of~\eqref{eq:A} in $\cA\cap\Delta'$.

\emph{Step 2: limit of a minimizing sequence.} Passing to a subsequence, $\eta_n\to\hat\mu\in\Delta'$, and
$\langle a,\hat\mu\rangle\le1$. Since $M_1(\pi_{\eta_n})\le M_1(\nu)+J(\eta_n)$ is bounded,
Lemma~\ref{lem:closed} gives, along a further subsequence, a weak limit $\pi\in\Inv(\hat\mu)$ with
$W_1(\nu,\pi)\le\Jstar$. If $\langle a,\hat\mu\rangle<1$, then $\hat\mu\in\cA$, $\pi=\pi_{\hat\mu}$, and
$\hat\mu$ is a solution.

\emph{Step 3: the critical case.} Let $\langle a,\hat\mu\rangle=1$. By Lemma~\ref{lem:critical},
$\pi(F(\hat\mu))=1$, so the maps $w_i$, $i\in T:=\supp\hat\mu$, have a common fixed point. Not all $a_i$,
$i\in T$, equal one: otherwise these maps would be translations with a fixed point, hence identities,
contradicting $T\cap I_0=\emptyset$. Since the $\hat\mu$-average of the $a_i$ equals one, there is $i\in T$
with $a_i<1$. Its only fixed point is $x_i^*:=b_i/(1-a_i)$, so $F(\hat\mu)=\{x_i^*\}$ and
$\pi=\delta_{x_i^*}$. But $e_i\in\cA$ and $\pi_{e_i}=\delta_{x_i^*}$. Hence $J(e_i)=W_1(\nu,\pi)\le\Jstar$, and
$e_i$ is a solution.
\end{proof}

The proof shows what happens at a critical parameter: either the maps in its support have no common
fixed point, and then the invariant measures escape to infinity and the objective diverges
(Theorem~\ref{thm:universal}(i)(b)); or they share a fixed point, and then the only possible limit is the
Dirac mass at that point, which is reachable from a vertex. In the running example $w_1(x)=x+1$,
$w_2(x)=x/2$ of Example~\ref{ex:nonclosed}, the single critical parameter $e_1$ is a translation without
fixed point, so the minimum of $J$ over the non-closed set $\cA=\Delta_2\setminus\{e_1\}$ is attained for
every target with finite mean. No condition on the translation vectors is needed.

We do not know whether Theorem~\ref{thm:homothety} extends to the class of Theorem~\ref{thm:affine}(iii),
entrywise nonnegative matrices $A_i$, whose critical parameters satisfy $\rho(\bar A(\mu_0))=1$. If $v\ge0$ is a
left Perron vector of $\bar A(\mu_0)$, the argument of Lemma~\ref{lem:critical} applies to the scalar process
$v^\top X_k$ and shows that $v^\top\bar b(\mu_0)=0$ and that every $\pi\in\Inv(\mu_0)$ is carried by the
intersection of the affine sets $\{x:v^\top((A_i-I)x+b_i)=0\}$, $i\in\supp\mu_0$, but this no longer
identifies $\pi$ as a measure that is reachable from a vertex.

When the boundary drift does not vanish, coercivity can be quantified, which localizes the minimizers a
priori.

\begin{proposition}[Quantitative coercivity and localization]\label{prop:localization}
Consider the system~\eqref{eq:homothety} with $a_i\in[0,1]$ for all $i$, let
$S:=\{i:a_i=1\}$ satisfy $\emptyset\neq S\neq[m]$, so that $\cA=\Delta_m\setminus\Delta_S$, and assume
\begin{equation}\label{eq:C}
\kappa_{\rm b}:=\dist\big(0,\conv\{b_i:i\in S\}\big)>0 .
\end{equation}
Put $b_{\max}:=\max_i\|b_i\|$, $\theta_0:=1-\max_{i\notin S}a_i$, $\tau(\mu):=\sum_{i\in S}\mu_i$, and let
$m_\nu$ be the mean of $\nu\in\cP_1(\R^d)$. Then:
\begin{enumerate}
\item[\rm(i)] For every $\mu\in\cA$,
\[
J(\mu)\ \ge\ \|\bar x(\mu)\|-\|m_\nu\|\ \ge\ \frac{\kappa_{\rm b}\,\tau(\mu)}{1-\tau(\mu)}-b_{\max}-\|m_\nu\| ,
\qquad\bar x(\mu)=\frac{\bar b(\mu)}{1-\langle a,\mu\rangle}.
\]
\item[\rm(ii)] For any reference law $\mu_1\in\cA$, every solution $\mu^\dagger$ of~\eqref{eq:A} lies in the
compact set
\[
\cK_R:=\Big\{\mu\in\Delta_m:\tau(\mu)\le\frac{R}{1+R}\Big\}\subset\cA,\qquad
R:=\frac{J(\mu_1)+b_{\max}+\|m_\nu\|}{\kappa_{\rm b}},
\]
on which $\langle a,\mu\rangle\le1-\theta_0/(1+R)$.
\item[\rm(iii)] (This part uses only $a_i\in[0,1]$ for all $i$, and neither~\eqref{eq:C} nor the
assumption on $S$.) For all $\mu,\eta\in\Delta_m$ with $\langle a,\mu\rangle\le r<1$ and
$\langle a,\eta\rangle\le r'<1$,
\begin{gather*}
M_1(\pi_\eta)\le\frac{b_{\max}}{1-r'},\qquad
W_1\big((P_\eta^*)^k\delta_0,\pi_\eta\big)\le\frac{b_{\max}\,(r')^k}{1-r'},\\
W_1(\pi_\mu,\pi_\eta)\le\frac{M_1(\pi_\eta)+2b_{\max}}{2(1-r)}\,\|\mu-\eta\|_1 .
\end{gather*}
\end{enumerate}
\end{proposition}

\begin{proof}
(i) The formula for $\bar x(\mu)$ is Theorem~\ref{thm:affine}(i). For any coupling $(Y,Y')$ of $\nu$ and
$\pi_\mu$, $\E\|Y-Y'\|\ge\|m_\nu-\bar x(\mu)\|$, which gives the first inequality. Write $\tau=\tau(\mu)$. If
$\tau>0$, then $\sum_{i\in S}\mu_ib_i=\tau\,c$ with $c\in\conv\{b_i:i\in S\}$, so
$\|\bar b(\mu)\|\ge\tau\kappa_{\rm b}-(1-\tau)b_{\max}$, and this also holds for $\tau=0$. Moreover
$0<1-\langle a,\mu\rangle=\sum_{i\notin S}\mu_i(1-a_i)\le1-\tau$. Hence
$\|\bar x(\mu)\|\ge\tau\kappa_{\rm b}/(1-\tau)-b_{\max}$ (the bound is trivial when its right-hand side is negative).

(ii) A solution satisfies $J(\mu^\dagger)\le J(\mu_1)$, so (i) gives $\tau/(1-\tau)\le R$ for
$\tau=\tau(\mu^\dagger)$. On $\cK_R$, $1-\langle a,\mu\rangle\ge\theta_0(1-\tau(\mu))\ge\theta_0/(1+R)$.

(iii) With base point $x_\circ=0$, invariance gives
\[
M_1(\pi_\eta)=\sum_i\eta_i\int\|a_ix+b_i\|\,\pi_\eta(\dd x)\le\langle a,\eta\rangle M_1(\pi_\eta)+b_{\max},
\]
which is the first bound, and the second follows from
$W_1((P_\eta^*)^k\delta_0,\pi_\eta)\le\langle a,\eta\rangle^kM_1(\pi_\eta)$. For the third, by invariance, the
contraction estimate of Proposition~\ref{prop:contraction} and~\eqref{eq:mixture},
\begin{align*}
W_1(\pi_\mu,\pi_\eta)&\le W_1(P_\mu^*\pi_\mu,P_\mu^*\pi_\eta)+W_1(P_\mu^*\pi_\eta,P_\eta^*\pi_\eta)\\
&\le r\,W_1(\pi_\mu,\pi_\eta)+\tfrac12\|\mu-\eta\|_1\big(M_1(\pi_\eta)+2b_{\max}\big),
\end{align*}
because $\|w_i(x)-w_j(x)\|\le|a_i-a_j|\,\|x\|+\|b_i-b_j\|\le\|x\|+2b_{\max}$.
\end{proof}

Condition~\eqref{eq:C} is not needed for existence (Theorem~\ref{thm:homothety}), only for the explicit
rate in (i): for $w_1(x)=x+1$, $w_2(x)=x-1$, $w_3(x)=x/2$ the origin lies in $\conv\{b_1,b_2\}$, the mean of
$\pi_\mu$ vanishes along $\mu=(\frac{1-p}2,\frac{1-p}2,p)$, and yet $J(\mu)\to\infty$ as $p\downarrow0$, more
slowly than in the presence of a boundary drift; see Section~\ref{sec:exp-affine}.

\section{Computation: search over a \texorpdfstring{$\delta$}{delta}-net}\label{sec:computation}

In practice the target is known through a sample, the invariant measure through finitely many
iterations of the Markov operator, and the simplex through finitely many of its points. This section
shows that the three approximations can be controlled simultaneously.

\subsection{The \texorpdfstring{$\delta$}{delta}-net}

\begin{lemma}[$\delta$-net of the simplex]\label{lem.delta net}
For $\delta>0$ let $q:=\lceil 2m/\delta\rceil$ and
\[
\mathcal N_\delta:=\left\{\left(k_1/q,\dots,k_m/q\right):k_i\in\mathbb Z_{\ge0},\ \textstyle\sum_{i=1}^m k_i=q\right\}\subset\Delta_m.
\]
Then $|\mathcal N_\delta|=\binom{q+m-1}{m-1}$, and for every $\mu\in\Delta_m$ there is $\mu'\in\mathcal N_\delta$ with $\|\mu-\mu'\|_1\le\delta$.
Moreover, if $\delta\le2$, then $\mu'$ can be chosen such that $\supp\mu'=\supp\mu$.
\end{lemma}

\begin{proof}
The cardinality of $\cN_\delta$ is the number of nonnegative integer solutions of $\sum_{i=1}^mk_i=q$,
which is $\binom{q+m-1}{m-1}$ (stars and bars). If $\delta>2$ the approximation claim is trivial, because
any two points of $\Delta_m$ are at $\ell_1$-distance at most $2$. Let $\delta\le2$, so that $q\ge m$, let
$\mu\in\Delta_m$, $T:=\supp\mu$ and $s:=|T|$. Put $k_i^0:=\lceil q\mu_i\rceil\ge1$ for $i\in T$. Then
$\sum_{i\in T}k_i^0=q+r$ for an integer $r$ with $0\le r\le s-1$, because
$q=\sum_{i\in T}q\mu_i\le\sum_{i\in T}k_i^0<q+s$. Since $\sum_{i\in T}(k_i^0-1)=q+r-s\ge r$, we can choose
integers $0\le c_i\le k_i^0-1$ with $\sum_{i\in T}c_i=r$. Set $k_i:=k_i^0-c_i\ge1$ for $i\in T$ and $k_i:=0$
for $i\notin T$, and $\mu':=(k_1/q,\dots,k_m/q)$. Then $\mu'\in\cN_\delta$, $\supp\mu'=T$, and
\[
\|\mu-\mu'\|_1\le\sum_{i\in T}\Big(\frac{k_i^0}{q}-\mu_i\Big)+\sum_{i\in T}\frac{c_i}{q}=\frac rq+\frac rq
\le\frac{2(m-1)}{q}<\frac{2m}{q}\le\delta .
\]
\end{proof}

The support-preserving refinement makes the net compatible with the facial structure found in
Section~\ref{sec:geometry}.

\begin{corollary}[Nets of face-structured sets]\label{cor:facenet}
Let $\cF\subseteq\Delta_m$ be a nonempty union of relatively open faces of the simplex, for instance $\cD$, $\cM$,
or $\cA$ itself when $\cX$ is finite. Then for every $\delta\le2$ the finite set $\cN_\delta\cap\cF$ is a
$\delta$-net of $\cF$. If moreover $\cX$ is compact and $\cF\subseteq\cA$, then
\[
\min_{\mu\in\cN_\delta\cap\cF}J(\mu)\ \xrightarrow[\delta\to0]{}\ \inf_{\mu\in\cF}J(\mu).
\]
\end{corollary}

\begin{proof}
The first claim is Lemma~\ref{lem.delta net}, since membership in $\cF$ depends only on the support. For the
second, fix $\eps>0$ and $\mu_\eps\in\cF$ with $J(\mu_\eps)\le\inf_\cF J+\eps$, and let $\mu'_\delta\in\cN_\delta\cap\cF$
with $\|\mu'_\delta-\mu_\eps\|_1\le\delta$. By continuity of $J$ on $\cA$ (Theorem~\ref{thm:topology}(iv)),
$J(\mu'_\delta)\to J(\mu_\eps)$ as $\delta\to0$.
\end{proof}

\subsection{The algorithm and its error decomposition}
Let $x_1,\dots,x_N$ be i.i.d.\ with law $\nu$, let $\hat\nu_N:=\frac1N\sum_{j=1}^N\delta_{x_j}$, fix
$\lambda_0\in\cP_1(\cX)$, and for $k\ge0$ put
\[
\pi_\mu^{(k)}:=(P_\mu^*)^k\lambda_0,\qquad\hat J_{N,k}(\mu):=W_1\big(\hat\nu_N,\pi_\mu^{(k)}\big).
\]
Algorithm~\ref{alg:shaping} minimizes $\hat J_{N,k}$ over a finite set $\mathcal G$ of grid points.

\begin{proposition}[Error decomposition]\label{prop:decomposition}
Let $\mu\in\cA$. Then
\begin{equation}\label{eq:pointwise}
\big|\hat J_{N,k}(\mu)-J(\mu)\big|\le W_1(\hat\nu_N,\nu)+W_1\big(\pi_\mu^{(k)},\pi_\mu\big).
\end{equation}
Consequently, for every finite $\mathcal G\subseteq\cA$ and every
$\hat\mu\in\operatorname*{arg\,min}_{\mathcal G}\hat J_{N,k}$,
\begin{equation}\label{eq:oracle}
J(\hat\mu)\le\min_{\mu\in\mathcal G}J(\mu)+2\,W_1(\hat\nu_N,\nu)+2\max_{\mu\in\mathcal G}W_1\big(\pi_\mu^{(k)},\pi_\mu\big).
\end{equation}
\end{proposition}

\begin{proof}
By the triangle inequality,
\begin{align*}
W_1(\hat\nu_N,\pi_\mu^{(k)})&\le W_1(\hat\nu_N,\nu)+W_1(\nu,\pi_\mu)+W_1(\pi_\mu,\pi_\mu^{(k)}),\\
W_1(\nu,\pi_\mu)&\le W_1(\nu,\hat\nu_N)+W_1(\hat\nu_N,\pi_\mu^{(k)})+W_1(\pi_\mu^{(k)},\pi_\mu),
\end{align*}
which is~\eqref{eq:pointwise}. Let $e$ denote the sum of the last two terms of~\eqref{eq:oracle} divided by two and
let $\mu^\circ$ minimize $J$ over $\mathcal G$. Then
$J(\hat\mu)\le\hat J_{N,k}(\hat\mu)+e\le\hat J_{N,k}(\mu^\circ)+e\le J(\mu^\circ)+2e$.
\end{proof}

The bound~\eqref{eq:oracle} requires neither contraction nor compactness: the statistical error, the
truncation error, and the discretization error (hidden in $\min_{\mathcal G}J$) are separated, and each
can be estimated in the regime at hand.

\begin{algorithm}[!htb]
\caption{Invariant measure shaping by search over a $\delta$-net}
\label{alg:shaping}
\KwIn{maps $w_1,\dots,w_m$; sample $x_1,\dots,x_N$ of the target $\nu$; initial law $\lambda_0$;
number of iterations $k$; mesh $\delta>0$; search region $\cK'\subseteq\cA$ (default: $\cA$).}
\KwOut{a sampling law $\hat\mu\in\cN_\delta\cap\cK'$.}
$\hat\nu_N\leftarrow\frac1N\sum_{j=1}^N\delta_{x_j}$;\quad $q\leftarrow\lceil2m/\delta\rceil$\;
$\mathcal G\leftarrow\cN_\delta\cap\cK'$, where
$\cN_\delta=\{(k_1/q,\dots,k_m/q):k_i\in\Z_{\ge0},\ \sum_ik_i=q\}$\;
$\hat J\leftarrow+\infty$\;
\For{$\mu\in\mathcal G$}{
  $\pi\leftarrow\lambda_0$\;
  \For{$t=1,\dots,k$}{$\pi\leftarrow\sum_{i=1}^m\mu_i\,(w_i)_\#\pi$\tcp*{one application of $P_\mu^*$}}
  $\hat J_{N,k}(\mu)\leftarrow W_1(\hat\nu_N,\pi)$\tcp*{discrete optimal transport}
  \If{$\hat J_{N,k}(\mu)<\hat J$}{$\hat J\leftarrow\hat J_{N,k}(\mu)$;\quad $\hat\mu\leftarrow\mu$\;}
}
\Return{$\hat\mu$}
\end{algorithm}

\subsection{\texorpdfstring{$\eps$}{epsilon}-optimal shaping}

\begin{definition}[Regular search region]\label{def:regular}
A pair of sets $\cK\subseteq\cK'\subseteq\cA$ is a \emph{regular search region} with constants $\bar\delta>0$,
$L_\pi<\infty$ and a sequence $r_k\downarrow0$ if
\begin{enumerate}
\item[\rm(R1)] for every $\delta\le\bar\delta$ and $\mu\in\cK$, every $\mu'\in\cN_\delta$ with
$\|\mu-\mu'\|_1\le\delta$ belongs to $\cK'$;
\item[\rm(R2)] $W_1(\pi_\mu,\pi_\eta)\le L_\pi\,\|\mu-\eta\|_1$ for all $\mu\in\cK$ and $\eta\in\cK'$;
\item[\rm(R3)] $W_1\big((P_\eta^*)^k\lambda_0,\pi_\eta\big)\le r_k$ for all $\eta\in\cK'$ and $k\ge0$.
\end{enumerate}
\end{definition}

\begin{theorem}[$\eps$-optimal sampling law]\label{thm:eps}
Let Assumption~\ref{ass:standing} hold, let $\nu\in\cP_1(\cX)$, and let $\cK\subseteq\cK'$ be a regular
search region with constants $(\bar\delta,L_\pi,(r_k))$. Let $\hat\mu$ be the output of
Algorithm~\ref{alg:shaping} with mesh $\delta\le\bar\delta$ and search region $\cK'$. Then
\begin{equation}\label{eq:main-bound}
J(\hat\mu)\le\inf_{\mu\in\cK}J(\mu)+L_\pi\,\delta+2\,r_k+2\,W_1(\hat\nu_N,\nu).
\end{equation}
In particular, for $\eps>0$, if
\[
\delta\le\min\Big\{\bar\delta,\frac{\eps}{3L_\pi}\Big\},\qquad r_k\le\frac\eps6,\qquad W_1(\hat\nu_N,\nu)\le\frac\eps6,
\]
then $J(\hat\mu)\le\inf_{\cK}J+\eps$, and $J(\hat\mu)\le\Jstar+\eps$ whenever $\cK$ contains a solution
of Problem~\eqref{eq:A}, or more generally a point $\mu$ with $J(\mu)\le\Jstar+\eps'$, in which case
$J(\hat\mu)\le\Jstar+\eps+\eps'$. Almost surely, the third condition holds for all sufficiently large $N$;
an explicit $N$ for which it holds with prescribed probability is given in
Corollary~\ref{cor:samplesize}. The algorithm evaluates $\hat J_{N,k}$ at no more than
$|\cN_\delta|=\binom{q+m-1}{m-1}$ points, $q=\lceil2m/\delta\rceil$. Finally, the approximation of the objective
is uniform over the search region,
\begin{equation}\label{eq:uniform}
\sup_{\mu\in\cK'}\big|\hat J_{N,k}(\mu)-J(\mu)\big|\le W_1(\hat\nu_N,\nu)+r_k=:\eps_{N,k},
\end{equation}
so that every $\mu'\in\cK'$ with $\hat J_{N,k}(\mu')\le\inf_{\cK'}\hat J_{N,k}+\eps_{\rm opt}$, found by any method,
satisfies $J(\mu')\le\inf_{\cK'}J+2\eps_{N,k}+\eps_{\rm opt}$.
\end{theorem}

\begin{proof}
Fix $\mu\in\cK$. By Lemma~\ref{lem.delta net} there is $\mu'\in\cN_\delta$ with $\|\mu-\mu'\|_1\le\delta$, and
$\mu'\in\cK'$ by (R1). By (R2), $J(\mu')\le J(\mu)+W_1(\pi_\mu,\pi_{\mu'})\le J(\mu)+L_\pi\delta$. Hence
$\min_{\cN_\delta\cap\cK'}J\le\inf_\cK J+L_\pi\delta$, and~\eqref{eq:main-bound} follows from~\eqref{eq:oracle}
with $\mathcal G=\cN_\delta\cap\cK'$ and (R3). The almost sure convergence $W_1(\hat\nu_N,\nu)\to0$ for
$\nu\in\cP_1(\cX)$ follows from Varadarajan's theorem and the strong law of large numbers for
$\rho(x_j,x_\circ)$, by (W1). The bound~\eqref{eq:uniform} is~\eqref{eq:pointwise} combined with (R3), and it gives
$J(\mu')\le\hat J_{N,k}(\mu')+\eps_{N,k}\le\inf_{\cK'}\hat J_{N,k}+\eps_{\rm opt}+\eps_{N,k}\le\inf_{\cK'}J+2\eps_{N,k}+\eps_{\rm opt}$.
\end{proof}

On a bounded state space the statistical term concentrates around its mean, which gives an explicit
sample size.

\begin{corollary}[Explicit sample size]\label{cor:samplesize}
Let $\cX$ be bounded with diameter $D$ and let $\alpha\in(0,1)$. With probability at least $1-\alpha$,
\[
W_1(\hat\nu_N,\nu)\le\E\,W_1(\hat\nu_N,\nu)+D\sqrt{\frac{\log(1/\alpha)}{2N}} .
\]
In particular, if $\cX=[a,b]$, then with probability at least $1-\alpha$,
$W_1(\hat\nu_N,\nu)\le(b-a)\big(1+\sqrt{2\log(1/\alpha)}\big)/(2\sqrt N)$, and the third condition of
Theorem~\ref{thm:eps} holds with probability at least $1-\alpha$ as soon as
\[
N\ge\frac{9\,(b-a)^2\big(1+\sqrt{2\log(1/\alpha)}\big)^2}{\eps^2}.
\]
\end{corollary}

\begin{proof}
Replacing one sample point $x_j$ by $x_j'$ moves mass $1/N$ from $x_j$ to $x_j'$, which changes
$W_1(\hat\nu_N,\nu)$ by at most $\rho(x_j,x_j')/N\le D/N$. McDiarmid's bounded-difference
inequality~\cite{McDiarmid1989} gives $\PP\big(W_1(\hat\nu_N,\nu)\ge\E\,W_1(\hat\nu_N,\nu)+t\big)\le
\exp(-2Nt^2/D^2)$, which is the first claim. On $[a,b]$,
$\E\,W_1(\hat\nu_N,\nu)=\int_a^b\E|\hat\nu_N([a,y])-\nu([a,y])|\,\dd y\le(b-a)/(2\sqrt N)$ by Jensen's
inequality, since the variance of $\hat\nu_N([a,y])$ is at most $1/(4N)$. The last claim is the condition
$(b-a)(1+\sqrt{2\log(1/\alpha)})/(2\sqrt N)\le\eps/6$.
\end{proof}

In $\R^d$ the expectation $\E\,W_1(\hat\nu_N,\nu)$ decays at a dimension-dependent rate, for instance
like $N^{-1/d}$ for $d\ge3$ on bounded sets~\cite{FournierGuillin2015,WeedBach2019}.

The next corollary lists regular search regions for the certified regions of Section~\ref{sec:geometry}.
For a nonempty $S\subseteq[m]$ and $\beta\in(0,1/|S|]$ let
\[
\cK_\beta(S):=\{\mu\in\Delta_m:\mu_i\ge\beta\ \text{for all }i\in S\},
\]
a compact convex set of laws whose support contains $S$.

\begin{corollary}[Explicit constants]\label{cor:constants}
In each of the following cases $\cK\subseteq\cK'$ is a regular search region with the stated constants.
In {\rm(a)--(c)}, $\cX$ is bounded with diameter $D$ and $\lambda_0\in\cP(\cX)$ is arbitrary.
\begin{enumerate}
\item[\rm(a)] \emph{Doeblin regime.} $S\in\sD$ with witness $(n,z)$, $\cK=\cK_\beta(S)$, $\cK'=\cK_{\beta/2}(S)$:
\[
\bar\delta=\beta,\qquad L_\pi=\frac{D\,n}{2\,\beta^{n}},\qquad r_k=D\Big(1-\frac{\beta^n}{2^n}\Big)^{\lfloor k/n\rfloor}.
\]
\item[\rm(b)] \emph{Splitting regime.} $\cX=[a,b]$, nondecreasing maps, $S\in\sM$ with words of length $n$
in~\eqref{eq:splitting}: the same sets and constants as in {\rm(a)}, with $D=b-a$.
\item[\rm(c)] \emph{Contraction on average.} Lipschitz maps with $L_i:=\Lip(w_i)$,
$L_{\max}:=\max_iL_i$, $r\in(0,1)$, $\cK=\{\mu:\sum_i\mu_iL_i\le r\}$ and
$\cK'=\{\mu:\sum_i\mu_iL_i\le\frac{1+r}{2}\}$:
\[
\bar\delta=\frac{1-r}{L_{\max}},\qquad L_\pi=\frac{D}{2(1-r)},\qquad r_k=D\Big(\frac{1+r}{2}\Big)^k .
\]
If $L_{\max}<1$ one may take $\cK=\cK'=\Delta_m$, $\bar\delta=2$, $L_\pi=D/(2(1-L_{\max}))$ and $r_k=D\,L_{\max}^k$.
\item[\rm(d)] \emph{Homotheties on $\R^d$.} The system~\eqref{eq:homothety} with $a_i\in[0,1]$,
$\lambda_0$ the Dirac mass at the origin, $r\in(0,1)$, $\cK=\{\mu:\langle a,\mu\rangle\le r\}$ and
$\cK'=\{\mu:\langle a,\mu\rangle\le\frac{1+r}2\}$:
\[
\bar\delta=1-r,\qquad L_\pi=\frac{(2-r)\,b_{\max}}{(1-r)^2},\qquad r_k=\frac{2\,b_{\max}}{1-r}\Big(\frac{1+r}{2}\Big)^k .
\]
\end{enumerate}
\end{corollary}

\begin{proof}
(a) If $\mu\in\cK_\beta(S)$ and $\|\mu-\mu'\|_1\le\delta\le\beta$, then $|\mu_i-\mu'_i|\le\frac12\|\mu-\mu'\|_1\le\beta/2$
for every $i$, because $\sum_i(\mu_i-\mu'_i)=0$; hence $\mu'_i\ge\beta/2$ for $i\in S$, which is (R1). By Theorem~\ref{thm:doeblin}(iii), $\cK_\beta(S)\subseteq\cE_{\dTV}(n,\beta^n)$ and
$\cK_{\beta/2}(S)\subseteq\cE_{\dTV}(n,(\beta/2)^n)$, so (R2) and (R3) are the two parts of
Proposition~\ref{prop:unifergodic}. (b) is identical, with Theorem~\ref{thm:splitting}(ii). (c) Since
$|\sum_i(\mu_i-\mu_i')L_i|\le\frac12\|\mu-\mu'\|_1L_{\max}$, (R1) holds for $\delta\le(1-r)/L_{\max}$;
(R2) and (R3) follow from Propositions~\ref{prop:contraction}(iii) and~\ref{prop:unifergodic}. (d) Here
$|\langle a,\mu-\mu'\rangle|\le\frac12\|\mu-\mu'\|_1$, which gives (R1), and (R2), (R3) are
Proposition~\ref{prop:localization}(iii) with $r'=(1+r)/2$.
\end{proof}

\begin{corollary}[Consistency]\label{cor:consistency}
Let $\cK\subseteq\cK'$ be a regular search region with $\cK'$ compact, let Assumption~\ref{ass:proper}
hold, and suppose that $\cK$ contains a solution of Problem~\eqref{eq:A}. Let $\delta_j\to0$, $k_j\to\infty$ and
$N_j\to\infty$, and let $\hat\mu_j$ be the corresponding outputs of Algorithm~\ref{alg:shaping}. Then, almost
surely, $J(\hat\mu_j)\to\Jstar$ and every accumulation point of $(\hat\mu_j)$ is a solution
of~\eqref{eq:A}.
\end{corollary}

\begin{proof}
By~\eqref{eq:main-bound}, $\Jstar\le J(\hat\mu_j)\le\Jstar+L_\pi\delta_j+2r_{k_j}+2W_1(\hat\nu_{N_j},\nu)\to\Jstar$
almost surely. If $\hat\mu_j\to\bar\mu$ along a subsequence, then $\bar\mu\in\cK'\subseteq\cA$ and
$J(\bar\mu)\le\liminf_jJ(\hat\mu_j)=\Jstar$ by Proposition~\ref{prop:lsc}.
\end{proof}

\begin{remark}[Scope and limitations]\label{rem:scope}
(a) \emph{Existence and computation fit together.} Theorem~\ref{thm:eps} needs a region $\cK$ that contains a
solution. Section~\ref{sec:existence} provides it: under the hypotheses of
Proposition~\ref{prop:localization}, all solutions lie in the explicit set $\cK_R$, on which
Corollary~\ref{cor:constants}(d) applies with $r=1-\theta_0/(1+R)$. Without localization, the regions
$\cK_\beta(S)$ exhaust the corresponding face-structured sets as $\beta\downarrow0$, and
$\inf_{\cK_\beta(S)}J$ decreases to the infimum over $\{\mu:\supp\mu\supseteq S\}$.

(b) \emph{Cost.} The net has $O\big((2m/\delta)^{m-1}\big)$ points: polynomial in $1/\eps$ for fixed $m$ and
exponential in $m$. In the Doeblin regime of Corollary~\ref{cor:constants}(a), the choice $\delta=\eps/(3L_\pi)$
and the condition $r_k\le\eps/6$ require
\[
|\cN_\delta|=O\Big(\Big(\frac{3mDn}{\eps\,\beta^n}\Big)^{m-1}\Big)\qquad\text{and}\qquad
k\ge n\,\Big\lceil\Big(\frac2\beta\Big)^n\log\frac{6D}{\eps}\Big\rceil ,
\]
the latter because $-\log(1-x)\ge x$. The guarantee of Theorem~\ref{thm:eps} is therefore a
certificate of finiteness with explicit constants rather than a practical method beyond a handful of modes;
Section~\ref{sec:exp-many} shows what can be done for $m=12$. The objective is in general not convex: in Example~\ref{ex:hold} it is continuous on $\cA$, constant on
each half-open segment issued from $e_3$, and not constant, which is incompatible with convexity (let one
endpoint of a chord tend to $e_3$ along one such segment while the other endpoint stays on a segment
with a larger value). An exhaustive search is therefore a natural baseline with a worst-case guarantee; exploiting more structure is the subject of Section~\ref{sec:sensitivity} and of Open Problem~\ref{op:complexity}. The measure
$\pi_\mu^{(k)}$ has up to $m^k$ atoms when $\lambda_0$ is a Dirac mass; it can be computed exactly when the
maps preserve a finite lattice (Section~\ref{sec:exp-storage}), and its atoms do not depend on $\mu$, so
they can be enumerated once (Appendix~\ref{sec:supp-battery}); in the contractive regime the state space can
also be quantized (Section~\ref{sec:quantization}). Otherwise $\pi_\mu^{(k)}$ may be replaced by the
empirical measure of independent runs of the chain, which adds the corresponding Monte Carlo error
to~\eqref{eq:oracle}.

(c) \emph{Conservativeness.} The constants of the Doeblin and splitting regimes depend on $\beta^{-n}$ and
are pessimistic; they certify finiteness and the correct qualitative dependence. The
bound~\eqref{eq:oracle} is sharper, but it involves the unknown quantities $W_1(\hat\nu_N,\nu)$ and
$W_1(\pi_\mu^{(k)},\pi_\mu)$; it becomes computable once these are replaced by the high-probability bound of
Corollary~\ref{cor:samplesize} and by $r_k$. In the experiments of Section~\ref{sec:numerics}, where $\nu$
and $\pi_\mu$ are known, we evaluate it exactly as a consistency check of the implementation.
\end{remark}

\subsection{Discretization of the state space}\label{sec:quantization}
When $\cX$ is a continuum and the maps contract on average, Problem~\eqref{eq:A} can be replaced by a
finite-state problem at a controlled cost, in the spirit of Ulam's
method~\cite{Ulam1960,Li1976,DellnitzJunge1999}. Let $\cX$ be totally bounded, let
$\cX=\bigcup_{j=1}^nC_j$ be a Borel partition with $\diam C_j\le h$, pick $x_j\in C_j$, let
$\cX_n:=\{x_1,\dots,x_n\}$ and let $Q_h:\cX\to\cX_n$ be the quantizer $Q_h(x):=x_j$ for $x\in C_j$, so that
$\rho(x,Q_h(x))\le h$. The quantized maps $\widetilde w_i:=Q_h\circ w_i|_{\cX_n}$ define an IFS on the finite
set $\cX_n$ with Markov operator $\widetilde P_\mu^*\lambda:=\sum_i\mu_i(\widetilde w_i)_\#\lambda$.

On a box $\cX=\prod_{l=1}^d[a_l,b_l]\subset\R^d$ with the Euclidean norm, maps $w_i$ that leave the box
invariant, and a uniform grid $\cX_h$ of step $h$ aligned with the box ($(b_l-a_l)/h\in\N$), one can instead round \emph{randomly and without bias}: let $R_h(y,\cdot)$ be the law
of the grid point obtained by rounding each coordinate $y_l$ independently to the neighboring grid value
above it with probability equal to its fractional position in the grid cell, and below it otherwise, so
that $\int z\,R_h(y,\dd z)=y$ and $\int\|z-y\|^2R_h(y,\dd z)\le dh^2/4$. The randomized quantized chain on
$\cX_h$ has kernel $\widehat P_\mu(x,\cdot):=\sum_i\mu_iR_h(w_i(x),\cdot)$. This is the model used in
Section~\ref{sec:exp-tcl}.

\begin{proposition}[Quantization error]\label{prop:quantization}
Let the maps $w_i$ be Lipschitz and let $\mu\in\Delta_m$.
\begin{enumerate}
\item[\rm(i)] If $\theta_1(\mu)=\sum_i\mu_i\Lip(w_i)<1$ and $\widetilde\pi$ is \emph{any} invariant distribution
of the quantized chain $\widetilde P_\mu$, then
\[
W_1(\pi_\mu,\widetilde\pi)\le\frac{h}{1-\theta_1(\mu)} .
\]
Consequently $|W_1(\nu,\pi_\mu)-W_1(\nu,\widetilde\pi)|\le h/(1-\theta_1(\mu))$ for every $\nu\in\cP_1(\cX)$, and
any two invariant distributions of $\widetilde P_\mu$ are at $W_1$-distance at most $2h/(1-\theta_1(\mu))$.
\item[\rm(ii)] On a box as above, if $\theta_2(\mu):=\sum_i\mu_i\Lip(w_i)^2<1$ and $\widehat\pi$ is any
invariant distribution of the randomized quantized chain $\widehat P_\mu$, then
\[
W_1(\pi_\mu,\widehat\pi)\le\frac{h}{2}\sqrt{\frac{d}{1-\theta_2(\mu)}} .
\]
\end{enumerate}
\end{proposition}

\begin{proof}
(i) For every $\lambda\in\cP(\cX_n)$, coupling $P_\mu^*\lambda$ and $\widetilde P_\mu^*\lambda$ through the same
index $i$ and the same point $x$ gives $W_1(P_\mu^*\lambda,\widetilde P_\mu^*\lambda)\le h$, because
$\rho(w_i(x),Q_h(w_i(x)))\le h$. By invariance of both measures and Proposition~\ref{prop:contraction},
\[
W_1(\pi_\mu,\widetilde\pi)\le W_1(P_\mu^*\pi_\mu,P_\mu^*\widetilde\pi)+W_1(P_\mu^*\widetilde\pi,\widetilde P_\mu^*\widetilde\pi)
\le\theta_1(\mu)\,W_1(\pi_\mu,\widetilde\pi)+h .
\]
The remaining claims follow from the triangle inequality.

(ii) Let $\widehat X_0=X_0\sim\widehat\pi$, let $(I_k)$ be i.i.d.\ with law $\mu$, and put
$X_{k+1}:=w_{I_{k+1}}(X_k)$ and $\widehat X_{k+1}:=w_{I_{k+1}}(\widehat X_k)+\xi_{k+1}$, where, conditionally on
$\cF_k:=\sigma(X_0,I_1,\xi_1,\dots,I_k,\xi_k)$ and $I_{k+1}$, the point $w_{I_{k+1}}(\widehat X_k)+\xi_{k+1}$ has law
$R_h(w_{I_{k+1}}(\widehat X_k),\cdot)$. Then $\widehat X_k\sim\widehat\pi$ and $X_k\sim(P_\mu^*)^k\widehat\pi$ for all
$k$, and $\E[\xi_{k+1}\mid\cF_k,I_{k+1}]=0$, $\E[\|\xi_{k+1}\|^2\mid\cF_k,I_{k+1}]\le dh^2/4$. The error
$e_k:=\widehat X_k-X_k$ satisfies $e_{k+1}=\big(w_{I_{k+1}}(\widehat X_k)-w_{I_{k+1}}(X_k)\big)+\xi_{k+1}$, where the
first term is measurable with respect to $(\cF_k,I_{k+1})$. Hence the cross term vanishes in
$\E\|e_{k+1}\|^2$, and since $I_{k+1}$ is independent of $\cF_k$,
\[
\E\|e_{k+1}\|^2\le\sum_i\mu_i\Lip(w_i)^2\,\E\|e_k\|^2+\frac{dh^2}4=\theta_2(\mu)\,\E\|e_k\|^2+\frac{dh^2}4 .
\]
With $e_0=0$ this gives $\E\|e_k\|^2\le dh^2/(4(1-\theta_2(\mu)))$ for all $k$, so that
$W_1(\widehat\pi,(P_\mu^*)^k\widehat\pi)\le\E\|e_k\|\le\frac h2\sqrt{d/(1-\theta_2(\mu))}$. Finally
$\theta_1(\mu)\le\theta_2(\mu)^{1/2}<1$ by Jensen's inequality, so $(P_\mu^*)^k\widehat\pi\to\pi_\mu$ in $W_1$ by
Lemma~\ref{lem:contraction}, and the claim follows by letting $k\to\infty$.
\end{proof}

Unbiased rounding errors average out instead of accumulating: for maps with a common Lipschitz constant
$L$ close to one, (ii) gives an error of order $h/\sqrt{2(1-L)}$ instead of $h/(1-L)$. The price is the
second-moment condition $\theta_2(\mu)<1$, which is stronger than $\theta_1(\mu)<1$: for Lipschitz constants
$0$ and $1.9$ with equal weights, $\theta_1=0.95$ but $\theta_2\approx1.8$.

The bound is uniform on $\{\theta_1\le r\}$ and holds for every invariant distribution of the
quantized chain, so the non-uniqueness that quantization may create is harmless in the contractive
regime: minimizing $W_1(\nu,\widetilde\pi)$ over pairs $(\mu,\widetilde\pi)$ with $\theta_1(\mu)\le r$ and
$\widetilde\pi$ invariant for $\widetilde P_\mu$ yields a law that is $2h/(1-r)$-optimal for~\eqref{eq:A}
restricted to $\{\theta_1\le r\}$. The quantized problem is a finite-state shaping problem, to which
Theorem~\ref{thm:finite} and Open Problem~\ref{op:complexity} apply. Without contraction no such bound
is available in general, since quantization can change the long-run behavior of the chain; the storage
models of Section~\ref{sec:exp-storage} do not need it because they leave a finite lattice invariant.

\section{Sensitivity of the invariant measure and a projected subgradient method}
\label{sec:sensitivity}

The grid search of Algorithm~\ref{alg:shaping} evaluates the objective at $|\mathcal N_\delta|=\binom{q+m-1}{m-1}$ points, and is therefore limited to a handful of modes. In this section we show that $\mu\mapsto\pi_\mu$ is differentiable, with a derivative given by the solution of a Poisson equation, and that the shaping objective $J$ is weakly convex on the sets $\mathcal A_r$. This yields computable subgradients, first-order optimality conditions, a local identifiability criterion, and a projected subgradient method with non-asymptotic guarantees.

\smallskip\noindent\emph{Setting and notation.} Throughout this section, $(\mathcal X,\rho)$ is a complete separable metric space with the base point $x_\circ$ of Section~\ref{sec:setting}. Each $w_i$ is Lipschitz with constant $L_i$, and we set $L_{\max}:=\max_iL_i$ and $b:=\max_i\rho(w_i(x_\circ),x_\circ)$. For $r\in[0,1)$ we use the set $\mathcal A_r=\{\mu\in\Delta_m:L(\mu)\le r\}$ of Lemma~\ref{lem:lipschitz-pi}; it is a convex polytope, and $\mathcal A_r\subseteq\mathcal A$ by Lemma~\ref{lem:contraction}. The optimization results of this section concern Problem~\eqref{eq:A} restricted to $\mathcal A_r$, that is, the feasible set $\cF=\mathcal A_r$ of Remark~\ref{rem:feasible}; they say nothing about minimizers in $\cA\setminus\mathcal A_{<1}$, which exist, for instance, in Example~\ref{ex:hold}, where $\mathcal A_{<1}=\emptyset$. We write $\mathcal A_{<1}:=\{\mu\in\Delta_m:L(\mu)<1\}$, and $T_0:=\{h\in\R^m:\sum_ih_i=0\}$ for the tangent space of the simplex. For a Lipschitz function $f:\mathcal X\to\R$, $\operatorname{Lip}(f)$ denotes its Lipschitz constant, $\operatorname{Lip}_1:=\{f:\operatorname{Lip}(f)\le1\}$, and $\lambda(f):=\int f\,d\lambda$ for $\lambda\in\mathcal P_1(\mathcal X)$. The transition operator acts on functions by
\[
(P_\mu f)(x):=\int f(y)\,P_\mu(x,dy)=\sum_{i=1}^m\mu_i f\big(w_i(x)\big),
\]
so that $(P_\mu^*\lambda)(f)=\lambda(P_\mu f)$ and $\operatorname{Lip}(P_\mu f)\le L(\mu)\operatorname{Lip}(f)$. We use the Kantorovich--Rubinstein duality \cite[Thm.~5.10 and Particular Case~5.16]{Villani2009}: for $\alpha,\beta\in\mathcal P_1(\mathcal X)$,
\begin{equation}\label{eq:KR}
W_1(\alpha,\beta)=\max_{f\in\operatorname{Lip}_1}\big(\alpha(f)-\beta(f)\big),
\end{equation}
where the maximum is attained. Finally, for $\mu\in\mathcal A_{<1}$ we define the \emph{mode spread}
\[
D_\mu:=\max_{i,j\in[m]}\int_{\mathcal X}\rho\big(w_i(x),w_j(x)\big)\,\pi_\mu(dx).
\]
Integrating the bound $\rho(w_i(x),x_\circ)\le L_i\rho(x,x_\circ)+b$ against $\pi_\mu$ and using invariance gives
$\int\rho(x,x_\circ)\,\pi_\mu(dx)\le b/(1-L(\mu))$. Hence, by the triangle inequality,
\begin{equation}\label{eq:Dr}
D_\mu\le D_r:=\min\Big\{\operatorname{diam}(\mathcal X),\ 2b\Big(1+\frac{L_{\max}}{1-r}\Big)\Big\}\qquad\text{for all }\mu\in\mathcal A_r .
\end{equation}

\subsection{Linear response}

\begin{lemma}[Poisson equation]\label{lem:poisson}
Let $\mu\in\mathcal A_{<1}$ and let $f$ be Lipschitz. The series
\[
u_f^\mu(x):=\sum_{k=0}^\infty\Big((P_\mu^kf)(x)-\pi_\mu(f)\Big)
\]
converges absolutely for every $x\in\mathcal X$. The function $u_f^\mu$ is Lipschitz with $\operatorname{Lip}(u_f^\mu)\le\operatorname{Lip}(f)/(1-L(\mu))$, satisfies $\pi_\mu(u_f^\mu)=0$, and solves the Poisson equation
\begin{equation}\label{eq:poisson}
u_f^\mu-P_\mu u_f^\mu=f-\pi_\mu(f).
\end{equation}
\end{lemma}

\begin{proof}
By duality and Lemma~\ref{lem:contraction},
$|(P_\mu^kf)(x)-\pi_\mu(f)|=|((P_\mu^*)^k\delta_x)(f)-\pi_\mu(f)|\le\operatorname{Lip}(f)\,L(\mu)^k\int\rho(x,y)\,\pi_\mu(dy)$, and the integral is finite. Hence the series converges absolutely. Since $\operatorname{Lip}(P_\mu^kf)\le L(\mu)^k\operatorname{Lip}(f)$, every partial sum, and therefore the pointwise limit, is Lipschitz with constant at most $\operatorname{Lip}(f)/(1-L(\mu))$. Invariance gives $\pi_\mu(P_\mu^kf)=\pi_\mu(f)$ for all $k$. The bound above shows that the series may be integrated term by term against $\pi_\mu$, because $\int\!\!\int\rho(x,y)\,\pi_\mu(dy)\,\pi_\mu(dx)<\infty$; this yields $\pi_\mu(u_f^\mu)=0$. Finally, $P_\mu$ is a finite combination of compositions, so it acts term by term, and $P_\mu u_f^\mu=\sum_{k\ge1}(P_\mu^kf-\pi_\mu(f))=u_f^\mu-(f-\pi_\mu(f))$.
\end{proof}

\begin{lemma}[Mixtures of push-forwards]\label{lem:mixture}
Let $\mu\in\mathcal A_{<1}$, $h\in T_0$, and let $g$ be Lipschitz. Then
\[
\Big|\sum_{i=1}^mh_i\int g\big(w_i(x)\big)\,\pi_\mu(dx)\Big|\le\tfrac12\,D_\mu\operatorname{Lip}(g)\,\|h\|_1 .
\]
Consequently, for all $\mu,\eta\in\mathcal A_{<1}$,
\begin{equation}\label{eq:stability-general}
W_1(\pi_\mu,\pi_\eta)\le\frac{D_\mu}{2(1-L(\eta))}\,\|\mu-\eta\|_1 .
\end{equation}
\end{lemma}

\begin{proof}
Write $h=h^+-h^-$ with $h^\pm\ge0$. Since $\sum_ih_i=0$, we have $\|h^+\|_1=\|h^-\|_1=\tfrac12\|h\|_1=:s$, and we may assume $s>0$. With $a_i:=\int g\circ w_i\,d\pi_\mu$, $\sum_ih_ia_i=s^{-1}\sum_{i,j}h_i^+h_j^-(a_i-a_j)$, and $|a_i-a_j|\le\operatorname{Lip}(g)\int\rho(w_i(x),w_j(x))\,\pi_\mu(dx)\le\operatorname{Lip}(g)D_\mu$. This proves the first claim. For the second, as in the proof of Lemma~\ref{lem:lipschitz-pi},
$W_1(\pi_\eta,\pi_\mu)\le L(\eta)\,W_1(\pi_\eta,\pi_\mu)+W_1(P_\eta^*\pi_\mu,P_\mu^*\pi_\mu)$. By \eqref{eq:KR} and the first claim with $h=\eta-\mu$,
$W_1(P_\eta^*\pi_\mu,P_\mu^*\pi_\mu)=\max_{f\in\operatorname{Lip}_1}\sum_i(\eta_i-\mu_i)\,\pi_\mu(f\circ w_i)\le\tfrac12D_\mu\|\eta-\mu\|_1$.
\end{proof}

Since $D_\mu\le\operatorname{diam}(\mathcal X)$, the bound \eqref{eq:stability-general} contains Lemma~\ref{lem:lipschitz-pi}, and it does not require $\mathcal X$ to be bounded.

\begin{proposition}[Exact perturbation identity]\label{prop:identity}
For all $\mu,\eta\in\mathcal A_{<1}$ and every Lipschitz $f$,
\begin{equation}\label{eq:identity}
\pi_\eta(f)-\pi_\mu(f)=\sum_{i=1}^m(\eta_i-\mu_i)\int u_f^\mu\big(w_i(x)\big)\,\pi_\eta(dx).
\end{equation}
\end{proposition}

\begin{proof}
Integrate \eqref{eq:poisson} against $\pi_\eta$; this is possible because $u_f^\mu$ is Lipschitz and $\pi_\eta\in\mathcal P_1(\mathcal X)$. Using $\pi_\eta(u_f^\mu)=\pi_\eta(P_\eta u_f^\mu)$, we obtain
$\pi_\eta(f)-\pi_\mu(f)=\pi_\eta(u_f^\mu)-\pi_\eta(P_\mu u_f^\mu)=\pi_\eta\big((P_\eta-P_\mu)u_f^\mu\big)$, and $(P_\eta-P_\mu)u=\sum_i(\eta_i-\mu_i)\,u\circ w_i$.
\end{proof}

\begin{definition}[Sensitivity vector]\label{def:sensitivity}
For $\mu\in\mathcal A_{<1}$ and Lipschitz $f$, the \emph{sensitivity vector} $s_\mu(f)\in\R^m$ has the entries
\[
s_\mu(f)_i:=\int u_f^\mu\big(w_i(x)\big)\,\pi_\mu(dx)=\big((w_i)_\#\pi_\mu\big)(u_f^\mu),\qquad i\in[m].
\]
\end{definition}

Only the differences $s_\mu(f)_i-s_\mu(f)_j$ matter below, since the vector is always paired with $h\in T_0$.

\begin{theorem}[Linear response]\label{thm:linear-response}
Let $\mu,\eta\in\mathcal A_{<1}$ and $h:=\eta-\mu$. For every Lipschitz $f$,
\begin{equation}\label{eq:linear-response}
\big|\pi_\eta(f)-\pi_\mu(f)-\langle s_\mu(f),h\rangle\big|\le C_{\mu,\eta}\operatorname{Lip}(f)\,\|h\|_1^2,
\qquad C_{\mu,\eta}:=\frac{L_{\max}D_\mu}{2(1-L(\mu))(1-L(\eta))},
\end{equation}
and $|\langle s_\mu(f),h\rangle|\le D_\mu\operatorname{Lip}(f)\|h\|_1/(2(1-L(\mu)))$. Hence the linear map $h\mapsto D\pi_\mu[h]$, defined by $\langle D\pi_\mu[h],f\rangle:=\langle s_\mu(f),h\rangle$, is the derivative of $\mu\mapsto\pi_\mu$ relative to $\mathcal A_{<1}$, in the sense that
\[
\sup_{f\in\operatorname{Lip}_1}\big|\pi_\eta(f)-\pi_\mu(f)-\langle D\pi_\mu[\eta-\mu],f\rangle\big|=O(\|\eta-\mu\|_1^2).
\]
Moreover,
\begin{equation}\label{eq:neumann}
\langle D\pi_\mu[h],f\rangle=\sum_{k=0}^\infty\zeta_h\big(P_\mu^kf\big),\qquad\zeta_h:=\sum_{i=1}^mh_i\,(w_i)_\#\pi_\mu,
\end{equation}
that is, $D\pi_\mu[h]=\sum_{k\ge0}(P_\mu^*)^k\zeta_h$ in duality with Lipschitz functions.
\end{theorem}

\begin{proof}
By Proposition~\ref{prop:identity}, the left-hand side of \eqref{eq:linear-response} equals $|(\pi_\eta-\pi_\mu)(g_h)|$ with $g_h:=\sum_ih_i\,u_f^\mu\circ w_i$. By Lemma~\ref{lem:poisson}, $\operatorname{Lip}(g_h)\le\sum_i|h_i|L_i\operatorname{Lip}(u_f^\mu)\le L_{\max}\operatorname{Lip}(f)\|h\|_1/(1-L(\mu))$. Therefore, by \eqref{eq:KR} and \eqref{eq:stability-general},
\[
|(\pi_\eta-\pi_\mu)(g_h)|\le\operatorname{Lip}(g_h)\,W_1(\pi_\eta,\pi_\mu)\le C_{\mu,\eta}\operatorname{Lip}(f)\|h\|_1^2 .
\] The bound on $\langle s_\mu(f),h\rangle$ is Lemma~\ref{lem:mixture} with $g=u_f^\mu$. For \eqref{eq:neumann}, insert the series of Lemma~\ref{lem:poisson} into $\langle s_\mu(f),h\rangle=\sum_ih_i\,\pi_\mu(u_f^\mu\circ w_i)$. The constants $\pi_\mu(f)$ cancel because $\sum_ih_i=0$, and the order of summation may be exchanged because the double series converges absolutely, by the bound in the proof of Lemma~\ref{lem:poisson}.
\end{proof}

\subsection{Structure of the shaping objective}

Fix $\nu\in\mathcal P_1(\mathcal X)$ and $J(\mu)=W_1(\nu,\pi_\mu)$. For $\mu\in\mathcal A_{<1}$, let
\[
\Phi_\mu:=\big\{f\in\operatorname{Lip}_1:\ f(x_\circ)=0,\ \pi_\mu(f)-\nu(f)=J(\mu)\big\}
\]
be the set of optimal Kantorovich potentials, which is nonempty by \eqref{eq:KR}. Set
\[
K_r:=\frac{L_{\max}D_r}{2(1-r)^2},\qquad \omega_r:=2mK_r,\qquad G_r:=\frac{\sqrt m\,D_r}{1-r}.
\]
For a function $\varphi$ that is $\omega$-weakly convex on a convex set $C$, i.e.\ $\varphi+\frac\omega2\|\cdot\|_2^2$ is convex on $C$, we write $\partial\varphi(\mu)$ for the set of $v\in\R^m$ such that $\varphi(\eta)\ge\varphi(\mu)+\langle v,\eta-\mu\rangle-\frac\omega2\|\eta-\mu\|_2^2$ for all $\eta\in C$. This set coincides with the Fr\'echet and Clarke subdifferentials \cite{DavisDrusvyatskiy2019}.

\begin{theorem}[Weak convexity, subgradients, and directional derivatives]\label{thm:structure}
Let $r\in[0,1)$.
\begin{enumerate}
\item[(i)] For all $\mu,\eta\in\mathcal A_r$, $f\in\Phi_\mu$ and $f'\in\Phi_\eta$,
\begin{equation}\label{eq:two-sided}
\langle s_\mu(f),\eta-\mu\rangle-K_r\|\eta-\mu\|_1^2\ \le\ J(\eta)-J(\mu)\ \le\ \langle s_\mu(f'),\eta-\mu\rangle+K_r\|\eta-\mu\|_1^2 .
\end{equation}
\item[(ii)] $J$ is $\omega_r$-weakly convex on $\mathcal A_r$, and $s_\mu(f)\in\partial(J+\iota_{\mathcal A_r})(\mu)$ for every $\mu\in\mathcal A_r$ and $f\in\Phi_\mu$, where $\iota_{\mathcal A_r}$ is the indicator function of $\mathcal A_r$. The centred vector $\bar s_\mu(f):=s_\mu(f)-\frac1m\big(\sum_is_\mu(f)_i\big)\1$ satisfies $\|\bar s_\mu(f)\|_2\le G_r$.
\item[(iii)] If $\mathcal X$ is compact, then for every $\mu\in\mathcal A_r$ and every $h\in T_0$ with $\mu+th\in\mathcal A_r$ for all small $t>0$,
\begin{equation}\label{eq:danskin}
J'(\mu;h):=\lim_{t\downarrow0}\frac{J(\mu+th)-J(\mu)}t=\max_{f\in\Phi_\mu}\langle s_\mu(f),h\rangle .
\end{equation}
In particular, if all $s_\mu(f)$, $f\in\Phi_\mu$, coincide modulo multiples of $\1$ (e.g.\ if the optimal potential is unique), then $J'(\mu;h)=\langle s_\mu(f),h\rangle$ is linear in $h$.
\end{enumerate}
\end{theorem}

\begin{proof}
(i) Since $f\in\operatorname{Lip}_1$, \eqref{eq:KR} gives $J(\eta)\ge\pi_\eta(f)-\nu(f)=J(\mu)+\pi_\eta(f)-\pi_\mu(f)$. The lower bound then follows from Theorem~\ref{thm:linear-response}, because $C_{\mu,\eta}\le K_r$ by \eqref{eq:Dr} and $L(\mu),L(\eta)\le r$. Similarly, $J(\mu)\ge\pi_\mu(f')-\nu(f')$ and $J(\eta)=\pi_\eta(f')-\nu(f')$ give $J(\eta)-J(\mu)\le\pi_\eta(f')-\pi_\mu(f')$, and Theorem~\ref{thm:linear-response} gives the upper bound.

(ii) By the Cauchy--Schwarz inequality, $\|h\|_1^2\le m\|h\|_2^2$. Hence the lower bound in \eqref{eq:two-sided} gives $J(\eta)\ge J(\mu)+\langle s_\mu(f),\eta-\mu\rangle-\frac{\omega_r}2\|\eta-\mu\|_2^2$ for all $\eta\in\mathcal A_r$, which is the defining inequality of $s_\mu(f)\in\partial(J+\iota_{\mathcal A_r})(\mu)$. Adding $\frac{\omega_r}2\|\eta\|_2^2=\frac{\omega_r}2\|\mu\|_2^2+\omega_r\langle\mu,\eta-\mu\rangle+\frac{\omega_r}2\|\eta-\mu\|_2^2$ shows that $J+\frac{\omega_r}2\|\cdot\|_2^2$ has, at every $\mu\in\mathcal A_r$, an affine minorant on $\mathcal A_r$ that is exact at $\mu$. It is therefore a pointwise supremum of affine functions on $\mathcal A_r$, hence convex. Finally, $|s_\mu(f)_i-s_\mu(f)_j|=|\pi_\mu(u_f^\mu\circ w_i-u_f^\mu\circ w_j)|\le\operatorname{Lip}(u_f^\mu)\,D_\mu\le D_r/(1-r)$. Hence every entry of $\bar s_\mu(f)$ is bounded by $D_r/(1-r)$ in absolute value, and $\|\bar s_\mu(f)\|_2\le G_r$.

(iii) Let $f\in\Phi_\mu$. By (i), $J(\mu+th)-J(\mu)\ge t\langle s_\mu(f),h\rangle-K_rt^2\|h\|_1^2$, so the $\liminf$ of the difference quotients is at least $\sup_{f\in\Phi_\mu}\langle s_\mu(f),h\rangle$. Conversely, take $t_n\downarrow0$ that realize the $\limsup$, and $f_n\in\Phi_{\mu+t_nh}$. By (i), $t_n^{-1}(J(\mu+t_nh)-J(\mu))\le\langle s_\mu(f_n),h\rangle+K_rt_n\|h\|_1^2$. The $f_n$ are $1$-Lipschitz and vanish at $x_\circ$, so by the Arzel\`a--Ascoli theorem a subsequence converges uniformly to some $f\in\operatorname{Lip}_1$ with $f(x_\circ)=0$. Since $W_1(\pi_{\mu+t_nh},\pi_\mu)\to0$ by \eqref{eq:stability-general}, we get $\pi_\mu(f)-\nu(f)=\lim_n\big(\pi_{\mu+t_nh}(f_n)-\nu(f_n)\big)=\lim_nJ(\mu+t_nh)=J(\mu)$, i.e.\ $f\in\Phi_\mu$. Moreover, for $g:=f_n-f$ we have $|(P_\mu^kg)(y)-\pi_\mu(g)|\le\min\{2\|g\|_\infty,\ 2L(\mu)^k\operatorname{diam}(\mathcal X)\}$, so $u^\mu_{f_n}\to u^\mu_f$ uniformly, and hence $s_\mu(f_n)\to s_\mu(f)$. Therefore the $\limsup$ is at most $\langle s_\mu(f),h\rangle$, which proves \eqref{eq:danskin}. The same compactness argument shows that the maximum is attained.
\end{proof}

\begin{corollary}[First-order optimality: equal sensitivities]\label{cor:optimality}
Let $\mathcal X$ be compact, $r\in[0,1)$, and let $\mu^\star$ minimize $J$ over $\mathcal A_r$. Then there exists $g$ in the convex hull of $\{s_{\mu^\star}(f):f\in\Phi_{\mu^\star}\}$ such that $\langle g,\eta-\mu^\star\rangle\ge0$ for all $\eta\in\mathcal A_r$. If, in addition, $L(\mu^\star)<r$, then there is $\lambda\in\R$ with
\[
g_i\ge\lambda\quad\text{for all }i\in[m],\qquad g_i=\lambda\quad\text{whenever }\mu_i^\star>0 .
\]
\end{corollary}

\begin{proof}
Let $\mathcal S$ be the convex hull of $\{s_{\mu^\star}(f):f\in\Phi_{\mu^\star}\}$. It is compact and convex, by the compactness argument in the proof of Theorem~\ref{thm:structure}(iii). Optimality and \eqref{eq:danskin} give $\max_{g\in\mathcal S}\langle g,\eta-\mu^\star\rangle=J'(\mu^\star;\eta-\mu^\star)\ge0$ for every $\eta\in\mathcal A_r$. Since $\mathcal A_r$ is compact and convex and the pairing is bilinear, Sion's minimax theorem \cite{Sion1958} yields $g\in\mathcal S$ with $\min_{\eta\in\mathcal A_r}\langle g,\eta-\mu^\star\rangle\ge0$. If $L(\mu^\star)<r$, then $\mu^\star+t(\eta-\mu^\star)\in\mathcal A_r$ for every $\eta\in\Delta_m$ and all small $t>0$. Hence $\langle g,\eta-\mu^\star\rangle\ge0$ for all $\eta\in\Delta_m$, i.e.\ $\mu^\star$ minimizes $\langle g,\cdot\rangle$ over $\Delta_m$. This holds if and only if $\mu^\star$ is supported on $\arg\min_ig_i$, which is the claim with $\lambda:=\min_ig_i$.
\end{proof}

In control terms, Corollary~\ref{cor:optimality} is an \emph{equal-sensitivity} condition. At an optimal randomized policy, every mode in use has the same marginal effect on the shaping error, and moving probability to an unused mode cannot decrease it to first order. This is analogous to the equal-incremental-cost condition of economic dispatch.

\subsection{Identifiability and parameter recovery}

For a finite signed measure $\zeta$ on $\mathcal X$ with $\zeta(\mathcal X)=0$ and $\int\rho(x,x_\circ)\,|\zeta|(dx)<\infty$
we write $\|\zeta\|_{\rm KR}:=\sup_{f\in\operatorname{Lip}_1}\zeta(f)$ for its Kantorovich--Rubinstein norm, so that
$W_1(\alpha,\beta)=\|\alpha-\beta\|_{\rm KR}$ by~\eqref{eq:KR}. For $\mu\in\mathcal A_{<1}$ the size of the linear response
in a direction $h\in T_0$ is
\[
N_\mu(h):=\sup_{f\in\operatorname{Lip}_1}\langle D\pi_\mu[h],f\rangle=\sup_{f\in\operatorname{Lip}_1}\langle s_\mu(f),h\rangle ,
\]
the first-order change of $\pi_\mu$ in $W_1$. It is controlled by the signed measure $\zeta_h=\sum_ih_i(w_i)_\#\pi_\mu$
of~\eqref{eq:neumann}, which is explicit.

\begin{proposition}[Injectivity of the linear response]\label{prop:injectivity}
Let $\mu\in\mathcal A_{<1}$. For every $h\in T_0$,
\[
\frac{\|\zeta_h\|_{\rm KR}}{1+L(\mu)}\ \le\ N_\mu(h)\ \le\ \frac{\|\zeta_h\|_{\rm KR}}{1-L(\mu)} .
\]
In particular, the linear response at $\mu$ is injective on $T_0$ if and only if the push-forwards
$(w_1)_\#\pi_\mu,\dots,(w_m)_\#\pi_\mu$ are linearly independent, and in that case
\[
\min_{h\in T_0,\ \|h\|_1=1}N_\mu(h)\ \ge\ \frac1{1+L(\mu)}\ \min_{h\in T_0,\ \|h\|_1=1}\Big\|\sum_ih_i(w_i)_\#\pi_\mu\Big\|_{\rm KR}>0 .
\]
\end{proposition}

\begin{proof}
By Lemma~\ref{lem:mixture}, $|\zeta_h(g)|\le\frac12D_\mu\operatorname{Lip}(g)\|h\|_1$ for Lipschitz $g$, so $\|\zeta_h\|_{\rm KR}<\infty$ and,
applying the definition to $\pm g/\operatorname{Lip}(g)$, $|\zeta_h(g)|\le\|\zeta_h\|_{\rm KR}\operatorname{Lip}(g)$. For the upper bound, \eqref{eq:neumann}
and $\operatorname{Lip}(P_\mu^kf)\le L(\mu)^k$ give $\langle D\pi_\mu[h],f\rangle\le\sum_{k\ge0}\|\zeta_h\|_{\rm KR}L(\mu)^k$ for $f\in\operatorname{Lip}_1$. For
the lower bound, let $g\in\operatorname{Lip}_1$ and $f:=(g-P_\mu g)/(1+L(\mu))$, so that $f\in\operatorname{Lip}_1$. By~\eqref{eq:neumann}
the series telescopes:
\[
(1+L(\mu))\,\langle D\pi_\mu[h],f\rangle=\lim_{K\to\infty}\sum_{k=0}^{K-1}\zeta_h\big(P_\mu^kg-P_\mu^{k+1}g\big)
=\zeta_h(g)-\lim_{K\to\infty}\zeta_h\big(P_\mu^Kg\big)=\zeta_h(g),
\]
because $|\zeta_h(P_\mu^Kg)|\le\|\zeta_h\|_{\rm KR}L(\mu)^K\to0$. Taking the supremum over $g$ gives the lower
bound. Hence $N_\mu(h)=0$ if and only if $\zeta_h(g)=0$ for all Lipschitz $g$, that is, $\zeta_h=0$, because bounded
Lipschitz functions determine a finite signed measure. Since the $(w_i)_\#\pi_\mu$ are probability measures,
$\sum_ih_i(w_i)_\#\pi_\mu=0$ forces $\sum_ih_i=0$, so $\zeta_h=0$ for some $h\in T_0\setminus\{0\}$ if and only if the
push-forwards are linearly dependent. The last claim follows because $h\mapsto\|\zeta_h\|_{\rm KR}$ is then a norm on
$T_0$.
\end{proof}

Linear independence holds, for instance, when the push-forwards are mutually singular, as for the $8$-adic
system of Section~\ref{sec:exp-many}, whose maps have disjoint images, and for the ternary system of
Section~\ref{sec:exp-ternary}, whose images overlap only at $\frac13$ and $\frac23$, points charged by at most one
push-forward. Nearly dependent push-forwards suggest flat directions of $J$, although the two bounds of
Proposition~\ref{prop:injectivity} differ by the factor $(1+L(\mu))/(1-L(\mu))$, which is large when $L(\mu)$ is
close to one. In the population model of Section~\ref{sec:exp-tcl}, where $L(\mu)\approx0.99$, the push-forwards
are translates of one another by small amounts, and the flat direction can be read off directly: the maps
share the slope $a$ and have fixed points $c_i$, so the stationary mean $\sum_i\mu_ic_i$ changes at rate
$\sum_ih_ic_i$ along $h$, while the stationary variance $\frac{1-a}{1+a}\operatorname{Var}_\mu(c)$ changes at a rate of
order $\frac{1-a}{1+a}\sum_ih_ic_i^2$. Along designs with a common mean, $J$ is therefore much flatter than
across them.

\begin{proposition}[Local sharpness]\label{prop:sharpness}
Let $r\in[0,1)$, $\mu^\star\in\mathcal A_r$ and $\nu=\pi_{\mu^\star}$. Suppose that the linear response at $\mu^\star$ is injective, that is, by Proposition~\ref{prop:injectivity}, that the push-forwards $(w_i)_\#\pi_{\mu^\star}$ are linearly independent. Then
\[
\kappa_0:=\min_{h\in T_0,\ \|h\|_1=1}N_{\mu^\star}(h)\ \ge\ \frac1{1+r}\min_{h\in T_0,\ \|h\|_1=1}\Big\|\sum_ih_i(w_i)_\#\pi_{\mu^\star}\Big\|_{\rm KR}>0,
\]
and $J(\mu)\ge\frac{\kappa_0}2\|\mu-\mu^\star\|_1$ for every $\mu\in\mathcal A_r$ with $\|\mu-\mu^\star\|_1\le\kappa_0/(2K_r)$. If, moreover, $\mu\mapsto\pi_\mu$ is injective on $\mathcal A_r$, then there is $\kappa>0$ with $J(\mu)\ge\kappa\|\mu-\mu^\star\|_1$ for all $\mu\in\mathcal A_r$.
\end{proposition}

\begin{proof}
The bound on $\kappa_0$ is Proposition~\ref{prop:injectivity}, since $L(\mu^\star)\le r$. For $\mu\in\mathcal A_r$ and $h:=\mu-\mu^\star$, \eqref{eq:KR} and \eqref{eq:linear-response} give
$J(\mu)=\sup_{f\in\operatorname{Lip}_1}(\pi_\mu(f)-\pi_{\mu^\star}(f))\ge N_{\mu^\star}(h)-K_r\|h\|_1^2\ge\kappa_0\|h\|_1-K_r\|h\|_1^2$, because $N_{\mu^\star}$ is positively homogeneous. This is at least $\frac{\kappa_0}2\|h\|_1$ when $\|h\|_1\le\kappa_0/(2K_r)$. Under global injectivity, $J$ is continuous by \eqref{eq:stability-general} and positive on the compact set $\{\mu\in\mathcal A_r:\|\mu-\mu^\star\|_1\ge\kappa_0/(2K_r)\}$. Hence $J(\mu)/\|\mu-\mu^\star\|_1$ is bounded below there by a positive constant.
\end{proof}

\begin{corollary}[Parameter recovery from samples]\label{cor:recovery}
Let $\mathcal X$ be compact with diameter $D$ and $L=\max_iL_i<1$, and let $\nu=\pi_{\mu^\star}$ for some $\mu^\star\in\Delta_m$. Suppose $J(\mu)\ge\kappa\|\mu-\mu^\star\|_1$ for all $\mu\in\Delta_m$ (cf.\ Proposition~\ref{prop:sharpness}). If $\hat\mu\in\Delta_m$ satisfies $\hat J_{N,k}(\hat\mu)\le\inf_{\Delta_m}\hat J_{N,k}+\varepsilon_{\rm opt}$, then
\[
\|\hat\mu-\mu^\star\|_1\le\frac{2W_1(\hat\nu_N,\nu)+2L^kD+\varepsilon_{\rm opt}}{\kappa}.
\]
If $\mathcal X\subset\R$, then $\E\|\hat\mu-\mu^\star\|_1\le\big(D/\sqrt N+2L^kD+\varepsilon_{\rm opt}\big)/\kappa$.
\end{corollary}

\begin{proof}
Let $\eps_{N,k}:=W_1(\hat\nu_N,\nu)+L^kD$, which bounds $\sup_\mu|\hat J_{N,k}(\mu)-J(\mu)|$ by~\eqref{eq:uniform} and Corollary~\ref{cor:constants}(c), applied with $\cK=\cK'=\Delta_m$. Then
$\kappa\|\hat\mu-\mu^\star\|_1\le J(\hat\mu)\le\hat J_{N,k}(\hat\mu)+\eps_{N,k}\le\hat J_{N,k}(\mu^\star)+\varepsilon_{\rm opt}+\eps_{N,k}\le J(\mu^\star)+2\eps_{N,k}+\varepsilon_{\rm opt}$, and $J(\mu^\star)=0$. On the real line, $\E W_1(\hat\nu_N,\nu)\le D/(2\sqrt N)$, as in the proof of Corollary~\ref{cor:samplesize}; see also~\cite{BobkovLedoux2019}.
\end{proof}

\subsection{A projected subgradient method}

\begin{algorithm}[htbp]
\caption{Projected subgradient method for invariant-measure shaping}
\label{alg:psg}
\KwIn{Maps $w_1,\dots,w_m$; target $\nu$ (or $\hat\nu_N$); $r<1$; $\mu^0\in\mathcal A_r$; number of iterations $T$; step sizes $(\alpha_t)$; tolerance $\delta_\pi\ge0$ and number of terms $\bar k$ for the inner computations.}
\KwOut{The iterates $\mu^0,\dots,\mu^T$ and a selected iterate $\hat\mu$.}
\For{$t=0,1,\dots,T-1$}{
Compute $\pi^t\approx\pi_{\mu^t}$ by fixed-point iteration of $P^*_{\mu^t}$, warm-started at $\pi^{t-1}$, until $W_1(\pi^t,P^*_{\mu^t}\pi^t)\le(1-r)\delta_\pi$\;
Compute an optimal Kantorovich potential $f^t$ for $\pi^t$ and $\nu$, i.e.\ a dual solution of the transport problem between $\pi^t$ and $\nu$ (on $\R$: $(f^t)'=-\operatorname{sign}(F_{\pi^t}-F_\nu)$)\;
Compute $u^t:=\sum_{k=0}^{\bar k-1}P_{\mu^t}^kf^t$, the truncated series of Lemma~\ref{lem:poisson} (up to an additive constant)\;
$g^t_i\leftarrow\big((w_i)_\#\pi^t\big)(u^t)$ for $i\in[m]$, and $g^t\leftarrow g^t-\frac1m\big(\sum_ig^t_i\big)\1$\;
$\mu^{t+1}\leftarrow\operatorname{proj}_{\mathcal A_r}\big(\mu^t-\alpha_tg^t\big)$\;
}
\Return{$\hat\mu:=\mu^\tau$ with $\tau$ uniform on $\{0,\dots,T-1\}$ (the output analyzed in Theorem~\ref{thm:psg}(a)); in practice, the iterate with the smallest computed value of $J$.}
\end{algorithm}

The projection onto $\mathcal A_r=\Delta_m\cap\{\mu:\langle\ell,\mu\rangle\le r\}$, $\ell:=(L_1,\dots,L_m)$, is $\operatorname{proj}_{\Delta_m}(y-\lambda^\star\ell)$. Here $\lambda^\star\ge0$ is found by bisection, since $\lambda\mapsto\langle\ell,\operatorname{proj}_{\Delta_m}(y-\lambda\ell)\rangle$ is nonincreasing, and $\operatorname{proj}_{\Delta_m}$ is computed by sorting \cite{Condat2016}. If all $L_i<1$ and $r=\max_iL_i$, then $\mathcal A_r=\Delta_m$. Projections onto subsets of $\{\sum_i\mu_i=1\}$ are invariant under shifts of their argument along $\1$, so centering $g^t$ does not change the iterates. The stopping rule for $\pi^t$ certifies $W_1(\pi^t,\pi_{\mu^t})\le\delta_\pi$, because $W_1(\pi^t,\pi_{\mu^t})\le W_1(\pi^t,P_{\mu^t}^*\pi^t)+L(\mu^t)\,W_1(\pi^t,\pi_{\mu^t})$.

Let $\omega:=\omega_r$, $\varphi:=J+\iota_{\mathcal A_r}$, and let $\varphi_\lambda(\mu):=\min_\eta\{\varphi(\eta)+\frac1{2\lambda}\|\eta-\mu\|_2^2\}$ be the Moreau envelope of $\varphi$. We call $g\in\R^m$ an \emph{$\eps$-subgradient} of $\varphi$ at $\mu\in\mathcal A_r$ if
\begin{equation}\label{eq:eps-subgradient}
J(\eta)\ge J(\mu)+\langle g,\eta-\mu\rangle-\tfrac\omega2\|\eta-\mu\|_2^2-\eps\qquad\text{for all }\eta\in\mathcal A_r .
\end{equation}
By Theorem~\ref{thm:structure}(ii), the exact vectors $\bar s_\mu(f)$, $f\in\Phi_\mu$, are $0$-subgradients of norm at most $G_r$. The next proposition quantifies the effect of the two approximations made in Algorithm~\ref{alg:psg}.

\begin{proposition}[Inexact subgradients]\label{prop:inexact}
Let $r\in[0,1)$, $\mu\in\mathcal A_r$, and let $\widetilde\pi\in\mathcal P_1(\mathcal X)$ satisfy $W_1(\widetilde\pi,\pi_\mu)\le\delta_\pi$. Let $f\in\operatorname{Lip}_1$ satisfy $\widetilde\pi(f)-\nu(f)=W_1(\widetilde\pi,\nu)$, let $\bar k\in\N$, $\tilde u:=\sum_{k=0}^{\bar k-1}P_\mu^kf$, and let $\tilde g$ be the centred vector with entries $((w_i)_\#\widetilde\pi)(\tilde u)$. Then $\tilde g$ is an $\eps$-subgradient of $\varphi$ at $\mu$ with
\[
\eps:=2\delta_\pi+2\psi,\qquad \psi:=\frac{2L_{\max}\delta_\pi+r^{\bar k}\big(D_r+2L_{\max}\delta_\pi\big)}{2(1-r)},
\]
and $\|\tilde g\|_2\le G_r+2\sqrt m\,\psi$.
\end{proposition}

\begin{proof}
Let $u:=u_f^\mu$, $s:=s_\mu(f)$, $\eta\in\mathcal A_r$ and $h:=\eta-\mu\in T_0$. Since $f\in\operatorname{Lip}_1$ and $W_1(\widetilde\pi,\pi_\mu)\le\delta_\pi$, we have $\pi_\mu(f)-\nu(f)\ge W_1(\widetilde\pi,\nu)-\delta_\pi\ge J(\mu)-2\delta_\pi$, and therefore, by~\eqref{eq:KR} and Theorem~\ref{thm:linear-response},
\[
J(\eta)\ge\pi_\eta(f)-\nu(f)\ge J(\mu)-2\delta_\pi+\langle s,h\rangle-K_r\|h\|_1^2 .
\]
Let $v_i:=((w_i)_\#\pi_\mu)(u)-((w_i)_\#\widetilde\pi)(\tilde u)$. As in the proof of Lemma~\ref{lem:mixture}, $|\langle v,h\rangle|\le\frac12\|h\|_1\max_{i,j}|v_i-v_j|$, and
\[
v_i-v_j=\int\big(u\circ w_i-u\circ w_j\big)\,\dd(\pi_\mu-\widetilde\pi)+\int\big(t\circ w_i-t\circ w_j\big)\,\dd\widetilde\pi ,
\]
where $t:=\sum_{k\ge\bar k}(P_\mu^kf-\pi_\mu(f))$ differs from $u-\tilde u$ by a constant, which cancels. The first integral is at most $(L_i+L_j)\operatorname{Lip}(u)\,\delta_\pi\le2L_{\max}\delta_\pi/(1-r)$ in absolute value, by Lemma~\ref{lem:poisson}. For the second, $\operatorname{Lip}(t)\le\sum_{k\ge\bar k}L(\mu)^k\le r^{\bar k}/(1-r)$, and $x\mapsto\rho(w_i(x),w_j(x))$ is $(L_i+L_j)$-Lipschitz, so that $\int\rho(w_i(x),w_j(x))\,\widetilde\pi(\dd x)\le D_\mu+2L_{\max}\delta_\pi\le D_r+2L_{\max}\delta_\pi$. Hence $\max_{i,j}|v_i-v_j|\le2\psi$ and $|\langle s-\tilde g,h\rangle|=|\langle v,h\rangle|\le\psi\|h\|_1\le2\psi$. Together with $K_r\|h\|_1^2\le\frac\omega2\|h\|_2^2$ this gives~\eqref{eq:eps-subgradient}. Finally $|\tilde g_i-\tilde g_j|\le|s_i-s_j|+2\psi\le D_r/(1-r)+2\psi$ by the proof of Theorem~\ref{thm:structure}(ii), which bounds every entry of the centred vector $\tilde g$.
\end{proof}

\begin{theorem}[Convergence of Algorithm~\ref{alg:psg}]\label{thm:psg}
Let $r\in[0,1)$ and $\nu\in\mathcal P_1(\mathcal X)$.
\begin{enumerate}
\item[(a)] Suppose that every $g^t$ is an $\eps$-subgradient of $\varphi$ at $\mu^t$ with $\|g^t\|_2\le G'$. With constant steps $\alpha_t\equiv\gamma/\sqrt T$, $\gamma>0$,
\[
\frac1T\sum_{t=0}^{T-1}\big\|\nabla\varphi_{1/(2\omega)}(\mu^t)\big\|_2^2\le\frac{2\big(J(\mu^0)-\min_{\mathcal A_r}J+\omega G'^2\gamma^2\big)}{\gamma\sqrt T}+4\omega\eps .
\]
The same bound holds for $\min_{0\le t<T}\|\nabla\varphi_{1/(2\omega)}(\mu^t)\|_2^2$ and for $\E\|\nabla\varphi_{1/(2\omega)}(\mu^\tau)\|_2^2$ with $\tau$ uniform on $\{0,\dots,T-1\}$. With exact quantities, $\eps=0$ and $G'=G_r$; for the quantities computed by Algorithm~\ref{alg:psg}, $\eps$ and $G'$ are given by Proposition~\ref{prop:inexact}. Moreover, if $\|\nabla\varphi_{1/(2\omega)}(\mu)\|_2\le\vartheta$ for some $\mu\in\mathcal A_r$, then there is $\hat\mu\in\mathcal A_r$ with $\|\hat\mu-\mu\|_2\le\vartheta/(2\omega)$, $J(\hat\mu)\le J(\mu)$, and $\operatorname{dist}(0,\partial\varphi(\hat\mu))\le\vartheta$.
\item[(b)] Suppose $\nu=\pi_{\mu^\star}$ with $\mu^\star\in\mathcal A_r$, that $J(\mu)\ge\kappa\|\mu-\mu^\star\|_2$ for all $\mu\in\mathcal A_r$ with $\|\mu-\mu^\star\|_2\le R_\star$, and that exact quantities are used, $g^t=\bar s_{\mu^t}(f^t)$ with $f^t\in\Phi_{\mu^t}$. Use the Polyak steps $\alpha_t:=J(\mu^t)/\|g^t\|_2^2$ (and stop if $J(\mu^t)=0$). If $\|\mu^0-\mu^\star\|_2\le\min\{R_\star,\kappa/(2\omega)\}$, then
\[
\|\mu^t-\mu^\star\|_2^2\le\Big(1-\frac{\kappa^2}{2G_r^2}\Big)^t\|\mu^0-\mu^\star\|_2^2\qquad\text{for all }t\ge0.
\]
\end{enumerate}
\end{theorem}

\begin{proof}
(a) We follow \cite{DavisDrusvyatskiy2019}. Let $\hat\mu^t:=\arg\min_\eta\{\varphi(\eta)+\omega\|\eta-\mu^t\|_2^2\}$ and $d_t:=\mu^t-\hat\mu^t$. By Theorem~\ref{thm:structure}(ii), $\eta\mapsto\varphi(\eta)+\omega\|\eta-\mu^t\|_2^2$ is $\omega$-strongly convex, so $\hat\mu^t$ is well defined. Moreover, $\nabla\varphi_{1/(2\omega)}(\mu^t)=2\omega d_t$ and $2\omega d_t\in\partial\varphi(\hat\mu^t)$ \cite{DavisDrusvyatskiy2019}. Comparing the values of this strongly convex function at $\mu^t$ and at its minimizer gives $J(\mu^t)-J(\hat\mu^t)\ge\frac{3\omega}2\|d_t\|_2^2$. Together with~\eqref{eq:eps-subgradient} at $\mu^t$, applied with $\eta=\hat\mu^t$, this yields
\[
\langle g^t,d_t\rangle\ge J(\mu^t)-J(\hat\mu^t)-\tfrac\omega2\|d_t\|_2^2-\eps\ge\omega\|d_t\|_2^2-\eps .
\]
Since $\hat\mu^t\in\cA_r$, nonexpansiveness of the projection, $\|g^t\|_2\le G'$ and the last inequality give
\[
\|\mu^{t+1}-\hat\mu^t\|_2^2\le\|d_t-\alpha_tg^t\|_2^2\le(1-2\alpha_t\omega)\|d_t\|_2^2+2\alpha_t\eps+\alpha_t^2G'^2 .
\]
Hence
\[
\begin{aligned}
\varphi_{1/(2\omega)}(\mu^{t+1})&\le\varphi(\hat\mu^t)+\omega\|\mu^{t+1}-\hat\mu^t\|_2^2\\
&\le\varphi_{1/(2\omega)}(\mu^t)-\tfrac{\alpha_t}2\big\|\nabla\varphi_{1/(2\omega)}(\mu^t)\big\|_2^2+2\omega\alpha_t\eps+\omega G'^2\alpha_t^2 .
\end{aligned}
\]
Summing over $t=0,\dots,T-1$, and using $\varphi_{1/(2\omega)}(\mu^0)\le J(\mu^0)$ and $\varphi_{1/(2\omega)}\ge\min_{\mathcal A_r}J$, we obtain $\sum_t\alpha_t\|\nabla\varphi_{1/(2\omega)}(\mu^t)\|_2^2\le2(J(\mu^0)-\min_{\mathcal A_r}J)+4\omega\eps\sum_t\alpha_t+2\omega G'^2\sum_t\alpha_t^2$. Inserting $\alpha_t\equiv\gamma/\sqrt T$ and dividing by $\gamma\sqrt T$ proves the bound; the minimum and the expectation over $\tau$ are bounded by the average. For the last claim, take $\hat\mu:=\arg\min_\eta\{\varphi(\eta)+\omega\|\eta-\mu\|_2^2\}$. Then $\|\hat\mu-\mu\|_2=\|\nabla\varphi_{1/(2\omega)}(\mu)\|_2/(2\omega)$, $J(\hat\mu)\le J(\hat\mu)+\omega\|\hat\mu-\mu\|_2^2\le J(\mu)$, and $\nabla\varphi_{1/(2\omega)}(\mu)\in\partial\varphi(\hat\mu)$.

(b) Let $d_t:=\|\mu^t-\mu^\star\|_2$ and $J_t:=J(\mu^t)$, and suppose $d_t\le\min\{R_\star,\kappa/(2\omega)\}$. Theorem~\ref{thm:structure}(ii) at $\mu^t$, applied with $\eta=\mu^\star$, gives $0=J(\mu^\star)\ge J_t+\langle g^t,\mu^\star-\mu^t\rangle-\frac\omega2d_t^2$. Since $\omega d_t^2\le\frac\kappa2d_t\le\frac12J_t$, this implies $\langle g^t,\mu^t-\mu^\star\rangle\ge\frac34J_t>0$ whenever $J_t>0$, so $g^t\ne0$. By nonexpansiveness of the projection,
\[
\begin{aligned}
d_{t+1}^2&\le d_t^2-2\alpha_t\big(J_t-\tfrac\omega2d_t^2\big)+\alpha_t^2\|g^t\|_2^2=d_t^2-\frac{J_t}{\|g^t\|_2^2}\big(J_t-\omega d_t^2\big)\\
&\le d_t^2-\frac{J_t^2}{2\|g^t\|_2^2}\le\Big(1-\frac{\kappa^2}{2G_r^2}\Big)d_t^2 ,
\end{aligned}
\]
using $J_t\ge\kappa d_t$ and $\|g^t\|_2\le G_r$. In particular, $d_{t+1}\le d_t$, and the claim follows by induction.
\end{proof}

With inexact quantities, part~(a) thus reaches near-stationarity up to the floor $4\omega\eps$, and $\eps$ can be driven to zero by tightening the inner tolerance $\delta_\pi$ and increasing $\bar k$, at a cost that is only logarithmic in $1/\delta_\pi$ and $1/\eps$ because of the geometric decay of the fixed-point iteration and of the series. Part~(b) is more sensitive to errors, since the Polyak step is proportional to $J(\mu^t)$ and inversely proportional to $\|g^t\|_2^2$, and both are small near $\mu^\star$; in our experiments the inner tolerances are several orders of magnitude below the accuracies reached.

\begin{remark}[Practical aspects]\label{rem:psg-practice}
(1) Each iteration of Algorithm~\ref{alg:psg} requires one invariant-measure computation, one optimal-transport dual solution, one Poisson solve, and $m$ integrals. In the finite-state model with $n$ states and sparse transition matrices $W_i$ of the maps $w_i$, the fixed-point and Poisson computations cost $O(m\,\mathrm{nnz}\,\log(1/\delta_\pi)/(1-r))$ operations, where $\mathrm{nnz}$ is the number of nonzeros per $W_i$. By contrast, Algorithm~\ref{alg:shaping} requires $\binom{q+m-1}{m-1}$ invariant-measure computations.
(2) The diminishing steps $\alpha_t=\gamma/\sqrt{t+1}$ used in Section~\ref{sec:numerics} give the rate $O(\log T/\sqrt T)$ in (a), by the same argument. The constants $\omega_r$ and $G_r$ are worst-case bounds, so in practice we normalize the steps by $\|g^0\|_2$.
(3) Theorem~\ref{thm:structure} involves $\nu$ only through $\Phi_\mu$, and its constants do not depend on $\nu$. Hence Theorem~\ref{thm:structure}, Proposition~\ref{prop:inexact} and Theorem~\ref{thm:psg}(a) hold verbatim for the empirical objective $\mu\mapsto W_1(\hat\nu_N,\pi_\mu)$, and Corollary~\ref{cor:recovery} converts its approximate minimization into a parameter-recovery bound. Theorem~\ref{thm:psg}(b) does not transfer: it needs a realizable target, and an empirical measure is in general not an invariant measure of the system; in the ternary benchmark of Section~\ref{sec:exp-ternary}, for instance, $\pi_\mu$ has no atoms unless $\mu$ is a vertex.
(4) Since $J$ is nonconvex, Theorem~\ref{thm:psg}(a) guarantees only near-stationarity. A natural hybrid, used in Section~\ref{sec:numerics}, is to run Algorithm~\ref{alg:shaping} on a coarse net and to refine its output by Algorithm~\ref{alg:psg}.
(5) The factors $m$ in $\omega_r$ and $\sqrt m$ in $G_r$ come from comparing the $\ell_1$-norm, in which Theorem~\ref{thm:linear-response} is naturally stated, with the Euclidean norm; an entropic mirror-descent variant of Algorithm~\ref{alg:psg}, which works in the $\ell_1$-geometry of the simplex~\cite{BeckTeboulle2003}, avoids them. Moreover, by Theorem~\ref{thm:linear-response}, $J(\mu)=\|\nu-\pi_\mu\|_{\rm KR}$ is the composition of a convex Lipschitz function with a map whose linearization error is quadratic. This is the setting of prox-linear methods~\cite{DrusvyatskiyPaquette2019}, whose subproblems minimize the convex function $\eta\mapsto\|\nu-\pi_{\mu^t}-D\pi_{\mu^t}[\eta-\mu^t]\|_{\rm KR}+\frac1{2\alpha_t}\|\eta-\mu^t\|_2^2$ over $\mathcal A_r$, a convex quadratic program on a finite state space. We expect such methods to cope better with the ill-conditioning observed in Section~\ref{sec:exp-tcl}, but do not pursue them here.
\end{remark}

\section{Numerical experiments}\label{sec:numerics}

The experiments illustrate the parts of the theory: the geometry of $\cA$ and the structure of the
solution set, the performance of Algorithm~\ref{alg:shaping} without contraction, the sensitivity
analysis and Algorithm~\ref{alg:psg}, the behavior of $J$ at the missing boundary, and the scalability of
Algorithm~\ref{alg:psg} to many modes. Monte Carlo standard deviations or standard errors are given in
parentheses.

\emph{Implementation.} On an interval, all quantities are computed on a finite-state model: exact when the
maps preserve a finite lattice (Section~\ref{sec:exp-storage}) or act on digit expansions
(Section~\ref{sec:exp-ternary} and the $8$-adic system of Section~\ref{sec:exp-many}), and otherwise by the mean-preserving randomized
rounding of Proposition~\ref{prop:quantization}(ii). On a grid, $W_1$ is the $\ell_1$-distance of the
cumulative distribution functions and an optimal potential is given by
$f'=-\operatorname{sign}(F_\pi-F_\nu)$. In Algorithm~\ref{alg:psg}, the fixed-point iteration for $\pi^t$ is stopped
when $W_1(\pi^t,P_{\mu^t}^*\pi^t)<10^{-13}$, and the Poisson series is truncated when the oscillation
$\max P_{\mu^t}^kf^t-\min P_{\mu^t}^kf^t$ falls below $10^{-11}$; unless stated otherwise, the steps are
$\alpha_t=\gamma/(\|g^0\|_2\sqrt{t+1})$, with $\gamma$ reported for each experiment.

\subsection{Saturating storage: a non-contractive system}\label{sec:exp-storage}
We consider two storage models on $\cX=[0,1]$ built from the mechanisms
\begin{align*}
w_{+}(x)&=\min\{1,x+0.3\}, & w_{-}(x)&=\max\{0,x-0.2\},\\
w_{0}(x)&=x, & w_{=}(x)&=\max\{0,x-0.5\}
\end{align*}
(charge, discharge, hold, deep discharge). System $\mathsf S_1$ uses $(w_+,w_-,w_0)$ and is the system of
Example~\ref{ex:hold}; system $\mathsf S_2$ uses $(w_+,w_-,w_=)$. All maps have Lipschitz constant one, so no
sampling law is contractive on average. In $\mathsf S_2$ every singleton is point-accessible
($w_+^4\equiv1$, $w_-^5\equiv0$, $w_=^2\equiv0$), hence $\cA=\Delta_3$ by Theorem~\ref{thm:doeblin}. The lattice
$\{0,0.1,\dots,1\}$ is invariant under all maps and carries every invariant measure of an admissible law,
so with $\lambda_0=\delta_0$ all quantities below are computed exactly by finite linear algebra.

\emph{Landscapes.} Figure~\ref{fig:landscapes} shows $J$ on the simplex for the targets $\nu=\pi_{\mu^*}$ with
$\mu^*=(0.3,0.2,0.5)$ in $\mathsf S_1$ and $\mu^*=(0.47,0.31,0.22)$ in $\mathsf S_2$. In $\mathsf S_1$ the level
sets are the segments through $e_3$ predicted by Lemma~\ref{lem:lazy}; the minimizers form the half-open
segment from $(0.6,0.4,0)$ to the excluded vertex $e_3$, so the solution set is neither a singleton nor
closed, in agreement with Example~\ref{ex:hold}. In $\mathsf S_2$ the computed landscape has a single
minimizer, but the valley around it is flat in one direction.

\begin{figure}[t]
\centering
\includegraphics[width=\textwidth]{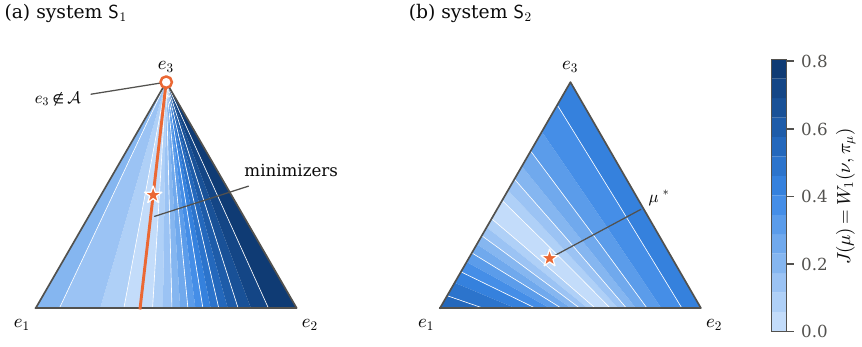}
\caption{The objective $J(\mu)=W_1(\nu,\pi_\mu)$ on $\Delta_3$ for the two storage systems of
Section~\ref{sec:exp-storage}, with $\nu=\pi_{\mu^*}$ (star). (a) System $\mathsf S_1$ (charge, discharge,
hold): $\cA=\Delta_3\setminus\{e_3\}$, and the minimizers form a half-open segment ending at the excluded
vertex. (b) System $\mathsf S_2$ (charge, discharge, deep discharge): $\cA=\Delta_3$.}
\label{fig:landscapes}
\end{figure}

\emph{Search over the $\delta$-net.} For $\mathsf S_2$ we ran Algorithm~\ref{alg:shaping} over the whole net
$\cN_\delta$ for $q\in\{12,30,60\}$, that is $\delta=2m/q\in\{0.5,0.2,0.1\}$ and
$|\cN_\delta|\in\{91,496,1891\}$, with $500$ independent samples of size $N$ for each configuration. The
target parameter $\mu^*$ does not belong to any of the three nets. Table~\ref{tab:gridN} reports the true
error $J(\hat\mu)$ of the returned law. It decreases with $N$ until it reaches the discretization floor
$\min_{\cN_\delta}J$, which itself decreases with $\delta$; the statistical term is of the order of
$\E\,W_1(\hat\nu_N,\nu)$, as predicted by~\eqref{eq:oracle}. Table~\ref{tab:gridk} shows the effect of the
number of iterations: the truncation error decays geometrically
(Figure~\ref{fig:gridsearch}(c)); at $k=32$ it is small compared with the statistical error for $N\le10^4$. Since $\nu$
and $\pi_\mu$ are known here, the bound~\eqref{eq:oracle} can be evaluated exactly; it held in every run, which
serves as a consistency check of the implementation.

\begin{table}[t]
\caption{System $\mathsf S_2$, $k=32$: mean (standard deviation) of $J(\hat\mu)$ over $500$ replications,
for three nets. Last two columns: mean of $W_1(\hat\nu_N,\nu)$, and mean parameter error
$\|\hat\mu-\mu^*\|_1$ for $q=60$. Last row: discretization floor.}
\label{tab:gridN}
\begin{tabular}{@{}lccccc@{}}
\hline
$N$ & $q=12$ & $q=30$ & $q=60$ & $W_1(\hat\nu_N,\nu)$ & $\|\hat\mu-\mu^*\|_1$ \\
\hline
$10^2$ & 0.0336 (0.0232) & 0.0316 (0.0214) & 0.0310 (0.0219) & 0.0364 & 0.333 \\
$10^3$ & 0.0181 (0.0044) & 0.0111 (0.0058) & 0.0100 (0.0066) & 0.0115 & 0.130 \\
$10^4$ & 0.0144 (0.0023) & 0.0074 (0.0011) & 0.0036 (0.0023) & 0.0037 & 0.059 \\
$10^5$ & 0.0138 (0.0000) & 0.0067 (0.0006) & 0.0022 (0.0004) & 0.0011 & 0.063 \\
\hline
$\min_{\mathcal N_\delta}J$ & 0.0138 & 0.0064 & 0.0021 & & \\
\hline
\end{tabular}
\end{table}

\begin{table}[t]
\caption{System $\mathsf S_2$, $q=60$: mean (standard deviation) of $J(\hat\mu)$ as a function of the number
of iterations $k$; last column: truncation error $\max_{\mu\in\cN_\delta}W_1(\pi_\mu^{(k)},\pi_\mu)$.}
\label{tab:gridk}
\begin{tabular}{@{}lcccc@{}}
\hline
$k$ & $N=10^3$ & $N=10^4$ & $N=10^5$ & truncation \\
\hline
4 & 0.1138 (0.0181) & 0.1120 (0.0076) & 0.1130 (0.0010) & $2.8\cdot10^{-1}$ \\
8 & 0.0311 (0.0150) & 0.0285 (0.0054) & 0.0300 (0.0022) & $1.0\cdot10^{-1}$ \\
16 & 0.0103 (0.0071) & 0.0042 (0.0025) & 0.0024 (0.0013) & $1.7\cdot10^{-2}$ \\
32 & 0.0100 (0.0066) & 0.0036 (0.0023) & 0.0022 (0.0004) & $5.1\cdot10^{-4}$ \\
\hline
\end{tabular}
\end{table}

\begin{figure}[t]
\centering
\includegraphics[width=\textwidth]{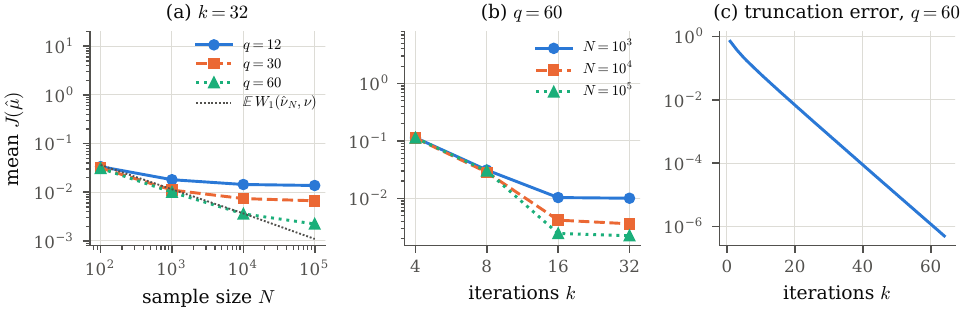}
\caption{Algorithm~\ref{alg:shaping} on system $\mathsf S_2$. (a) Mean error $J(\hat\mu)$ against the sample
size for three nets; the dotted line is the mean of $W_1(\hat\nu_N,\nu)$. (b) Mean error against the
number of iterations. (c) Truncation error $\max_{\mu\in\cN_\delta}W_1(\pi^{(k)}_\mu,\pi_\mu)$.}
\label{fig:gridsearch}
\end{figure}

Two further observations are worth recording. First, the guarantee of Theorem~\ref{thm:eps} concerns the
\emph{objective}, not the parameter: the last column of Table~\ref{tab:gridN} shows that $\hat\mu$ may stay
at a distance of several hundredths from $\mu^*$ while $J(\hat\mu)$ is of order $10^{-3}$, because of
the flat valley of Figure~\ref{fig:landscapes}(b). This is the correct notion of success for the shaping
problem, whose goal is the invariant measure. Second, the theoretical constants are conservative but
of the right nature. With $S=\{3\}$ and witness $(2,0)$, Proposition~\ref{prop:unifergodic}(ii) gives
$\dTV(\pi_\mu,\pi_\eta)\le\mu_3^{-2}\|\mu-\eta\|_1$; over $1.9\cdot10^5$ random pairs of points of the net with
$q=60$ the ratio of the two sides never exceeded $0.64$, so the bound was valid and within a factor two of
being attained. The resulting Lipschitz constant $L_\pi=\beta^{-2}$ on $\cK_\beta(\{3\})$, however, is far
above the largest ratio $W_1(\pi_\mu,\pi_\eta)/\|\mu-\eta\|_1\approx1.5$ observed on the net.

\subsection{An exact benchmark}
\label{sec:exp-ternary}

Let $w_i(x)=(x+i-1)/3$, $i=1,2,3$, on $[0,1]$, so $L=1/3$ and $D=1$. Then $\pi_\mu$ is the law of $\sum_{j\ge1}d_j3^{-j}$ with i.i.d.\ digits $d_j\sim\mu$. The finite-state model on the first $9$ ternary digits is exact: its invariant distribution is $\mu^{\otimes9}$, i.e.\ the masses of $\pi_\mu$ on the $3^9$ triadic intervals. Hence $J$ is computed with error at most $2\cdot3^{-9}<1.1\cdot10^{-4}$, and the sensitivity vector of Definition~\ref{def:sensitivity} can be compared with finite differences: over $20$ random directions the relative error is at most $1.4\cdot10^{-8}$.

Figure~\ref{fig:ternary}(a) shows $J$ on $\Delta_3$ for the bimodal target $\nu=\frac12\mathrm{Beta}(3,12)+\frac12\mathrm{Beta}(9,3)$. Grid search with $q=120$ gives $J=0.0318$ at $\mu=(0.433,\,0.175,\,0.392)$. Algorithm~\ref{alg:psg} started at $\frac13(1,1,1)$, with $\gamma=0.15$, reaches $J=0.0320$ in $40$ iterations. The minimizer of the convex collage surrogate $W_1(\nu,P_\mu^*\nu)$ has $J=0.0331$. Since $L=1/3$, the guarantee factor $(1+L)/(1-L)=2$ of the collage surrogate is mild, and indeed it is only 4\% worse than the optimum here.

Figure~\ref{fig:ternary}(b) concerns a realizable target $\nu=\pi_{\mu^\star}$ with $\mu^\star=(0.21,0.47,0.32)$, observed through $N$ samples. The estimator runs Algorithm~\ref{alg:shaping} on a coarse net ($q=30$) and refines its output with $150$ iterations of Algorithm~\ref{alg:psg} ($\gamma=0.03$) on $\mu\mapsto W_1(\hat\nu_N,\pi_\mu)$; we use $100$ replications per $N$. The three maps have disjoint images up to endpoints, so the linear response is injective at every $\mu^\star$ by Proposition~\ref{prop:injectivity}. The median error decays approximately like $N^{-1/2}$, in line with Corollary~\ref{cor:recovery}; at the largest $N$ it flattens slightly, where the optimization error of the refinement step becomes visible. The bound of the corollary, evaluated with the empirical $W_1(\hat\nu_N,\nu)$ and the sharpness constant $\kappa\approx0.131$ estimated on a net with $q=60$, is conservative by a roughly constant factor.

\begin{figure}[htbp]
\centering
\input{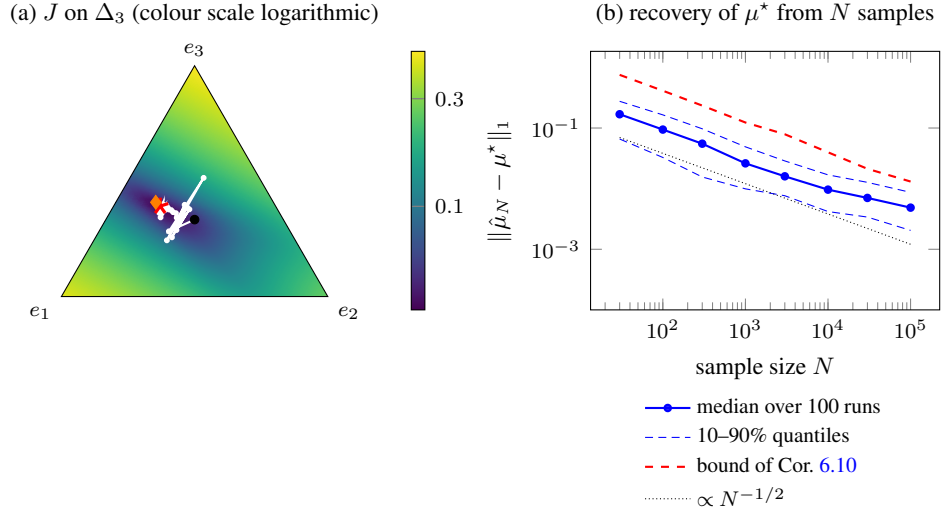}
\caption{Ternary benchmark. (a) $J$ on $\Delta_3$ for a bimodal target (logarithmic color scale), with the trajectory of Algorithm~\ref{alg:psg} (white; start $\bullet$), the grid-search minimizer ($\star$) and the collage minimizer ($\diamond$). (b) Recovery of $\mu^\star$ from $N$ samples, with the bound of Corollary~\ref{cor:recovery}.}
\label{fig:ternary}
\end{figure}

\subsection{Homotheties near the missing boundary}\label{sec:exp-affine}
We illustrate Theorem~\ref{thm:homothety} and Proposition~\ref{prop:localization} on $\cX=\R$ with two
systems and two paths approaching the critical set as $p\downarrow0$:
\begin{align*}
\text{(a)}\quad & w_1(x)=x+1,\ \ w_2(x)=0.5\,x,\ \ w_3(x)=0.3\,x-2, && \mu(p)=(1-p,\tfrac p2,\tfrac p2);\\
\text{(b)}\quad & w_1(x)=x+1,\ \ w_2(x)=x-1,\ \ w_3(x)=0.5\,x, && \mu(p)=(\tfrac{1-p}2,\tfrac{1-p}2,p).
\end{align*}
In (a), $\langle a,\mu(p)\rangle=1-0.6p$ and $\bar b(\mu(p))=1-2p$, so $\bar x(\mu(p))=(1-2p)/(0.6p)$ and
condition~\eqref{eq:C} holds with $\kappa_{\rm b}=1$. In (b), $\langle a,\mu(p)\rangle=1-0.5p$, the mean vanishes
and~\eqref{eq:C} fails, but the critical law $\mu(0)$ is a symmetric random walk, which has no invariant
probability measure; a direct computation gives $\int x^2\,\dd\pi_{\mu(p)}=(1-p)/(0.75p)$. For each $p$ we
simulated $2\cdot10^5$ independent chains started at $0$ for $\lceil40/(1-\langle a,\mu(p)\rangle)\rceil$
steps, after which the law of the state is within $e^{-40}M_1(\pi_{\mu(p)})$ of $\pi_{\mu(p)}$ in $W_1$, and
we estimated $J$ for the target $\nu=N(0,1)$ from the empirical quantiles.

The simulated moments agree with the closed forms to within Monte Carlo error, all $18$ estimates lying
within $2.2$ standard errors (Tables~\ref{tab:affa} and~\ref{tab:affb} in Appendix~\ref{sec:supp-affine}), and Figure~\ref{fig:coercivity} shows the two predicted behaviors: $J$ has an interior minimum along each path and diverges at the
missing boundary, like $p^{-1}$ in (a), as in Proposition~\ref{prop:localization}(i), and more slowly in (b),
where the first absolute moment is bounded by $(\int x^2\,\dd\pi_{\mu(p)})^{1/2}$ and is observed to grow like
$p^{-1/2}$. Existence of a minimizer in case (b) is covered by Theorem~\ref{thm:homothety} although the
drift condition~\eqref{eq:C} fails.

\begin{figure}[t]
\centering
\includegraphics[width=\textwidth]{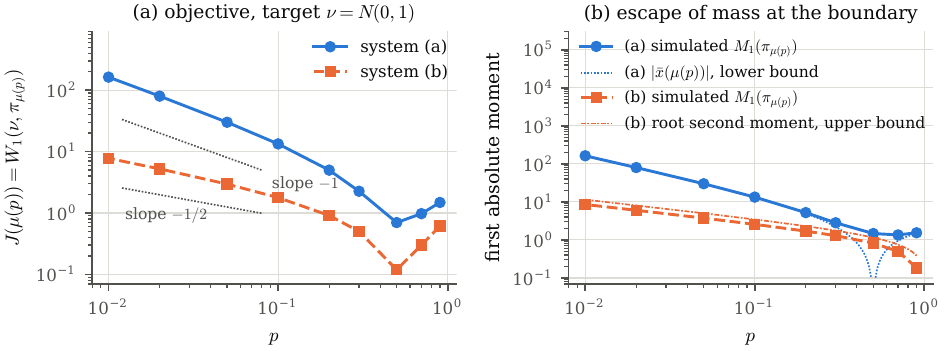}
\caption{Homotheties near the missing boundary (Section~\ref{sec:exp-affine}). (a) The objective along the
two paths. (b) First absolute moment of $\pi_{\mu(p)}$ with the closed-form lower bound $|\bar x(\mu(p))|$ for
system (a) and upper bound $(\int x^2\dd\pi_{\mu(p)})^{1/2}$ for system (b).}
\label{fig:coercivity}
\end{figure}

\subsection{Shaping the temperature distribution of a TCL population}
\label{sec:exp-tcl}

We use the discrete-time first-order thermal model of thermostatically controlled loads \cite{Callaway2009,Mathieu2013},
\[
\theta_{k+1}=a\,\theta_k+(1-a)\big(\theta_{\rm amb}-R\,P_{I_k}\big),\qquad a=e^{-\Delta t/(RC)},
\]
with thermal resistance $R=2\,^\circ$C/kW, capacitance $C=10$\,kWh/$^\circ$C, ambient temperature $\theta_{\rm amb}=32\,^\circ$C, and three cooling modes $P\in\{0,7,14\}$\,kW (off, half power, full power). The aggregator broadcasts $\mu$ every $\Delta t=15$\,min, and each agent, here a load, draws its mode independently. The three maps share the slope $a=0.9876$, so $L=a$ and the mixing time $1/(1-a)\approx81$ steps is long. All maps leave $[4,32]$ invariant, and we discretize this interval with step $h=0.01\,^\circ$C ($n=2801$), using the mean-preserving randomized rounding of Section~\ref{sec:quantization}. By Proposition~\ref{prop:quantization}(ii) with $\theta_2(\mu)=a^2$, the invariant measure of the discretized chain is within $(h/2)/\sqrt{1-a^2}\approx0.032\,^\circ$C of $\pi_\mu$ in $W_1$, for every design; the bound of Proposition~\ref{prop:quantization}(i) for deterministic rounding, of order $h/(1-a)$, would exceed the objective values themselves. Empirically the error is smaller still: halving $h$ changes $J$ by at most $0.0013$ for all designs below, and a Monte Carlo simulation of the continuous chain ($10^5$ agents) gives $W_1$ discrepancies of at most 0.005\,$^\circ$C. The target is the uniform distribution on the comfort band $[19,21]\,^\circ$C. Since the stationary mean is $\theta_{\rm amb}-R\,\E P_I$, every design with mean $20\,^\circ$C draws the same average power of $6$\,kW per agent; the designs differ only in how this power is distributed over the modes.

Table~\ref{tab:tcl} and Figure~\ref{fig:tcl} compare five designs. The on/off duty cycle, which only matches the mean, spreads the temperatures far beyond the band. The collage surrogate selects the minimum-variance mixture of the off and half-power modes and under-disperses. With $a\approx0.99$, its guarantee factor $(1+L)/(1-L)\approx160$ is uninformative: the collage distance ignores how variance accumulates over the long mixing time. A natural baseline that does account for it is moment matching~\cite{HandyMantica1990,Vrscay1991}: since the maps share the slope $a$, the stationary mean and variance are $\E c_I$ and $\frac{1-a}{1+a}\operatorname{Var}(c_I)$, where $c_i=\theta_{\rm amb}-RP_i$ is the fixed point of the $i$th map, and matching the mean $20\,^\circ$C and the variance $\frac13$ of the band is a pair of linear equations for $\mu$. It gives $J=0.090$, far better than the collage surrogate, because $\pi_\mu$ is nearly Gaussian when $a$ is close to one. $W_1$-shaping does better still, $J=0.080$, by choosing a slightly larger spread than that of the band (standard deviation $0.637$ versus $0.577$). The gaps between the first two designs and the last three exceed twice the certified discretization error; the gaps among the last three do not, but they are stable under halving $h$ ($0.0898$, $0.0860$ and $0.0795$). Note that $J$ measures the match of the whole distribution: the collage design keeps fewer agents outside the band, $1.0\%$ against $11.5\%$, but places almost all of them near $20\,^\circ$C rather than across the band. If the goal is to keep agents inside the band rather than to spread them across it, the probability of leaving the band, or a constraint on it, is the more appropriate objective.

Because $J$ is very steep across the set of designs with mean $20\,^\circ$C and flat along it, the subgradient method zig-zags (Figure~\ref{fig:tcl}(b)). This is consistent with the large constants $\omega_r$ and $G_r$ of Theorem~\ref{thm:psg} for $r$ close to $1$. Started at the barycenter with $\gamma=0.05$, Algorithm~\ref{alg:psg} reaches only $J=0.196$ after $200$ iterations. The hybrid of Remark~\ref{rem:psg-practice}(4) runs Algorithm~\ref{alg:shaping} on a net with $q=20$ ($231$ evaluations of $J$) and then Algorithm~\ref{alg:psg} with $\gamma=0.02$. It reaches $J=0.080$ (Monte Carlo: 0.079) after only 2 iterations, i.e.\ at the cost of about $235$ invariant-measure computations. Algorithm~\ref{alg:shaping} on the much finer net with $q=100$ needs $5151$ computations and attains only $0.086$, because the spacing of that net exceeds the width of the valley.

\begin{table}[htbp]
\centering
\caption{TCL population: five sampling laws, all with stationary mean $\approx20\,^\circ$C (average power $6$\,kW). ``Outside'': stationary fraction of agents outside the band, $\pi_\mu(\theta\notin[19,21])$; ``collage'': value of the collage surrogate $W_1(\nu,P_\mu^*\nu)$.}
\label{tab:tcl}
\small
\begin{tabular}{lcccccc}
\hline
design & $(\mu_{\rm off},\mu_{\rm half},\mu_{\rm full})$ & $J(\mu)$ & mean & std & outside & collage \\
\hline
on/off duty cycle & $(0.571,\,0.000,\,0.429)$ & 0.373 & 20.00 & 1.096 & 36.0\% & 0.0173 \\
collage surrogate & $(0.143,\,0.857,\,0.000)$ & 0.195 & 20.01 & 0.389 & 1.0\% & 0.0065 \\
moment matching & $(0.218,\,0.707,\,0.075)$ & 0.090 & 20.00 & 0.578 & 8.2\% & 0.0074 \\
Alg.~\ref{alg:shaping}, $q=100$ & $(0.250,\,0.640,\,0.110)$ & 0.086 & 19.96 & 0.646 & 12.1\% & 0.0081 \\
hybrid: $q=20$, Alg.~\ref{alg:psg} & $(0.247,\,0.649,\,0.104)$ & 0.080 & 20.00 & 0.637 & 11.5\% & 0.0080 \\
\hline
\end{tabular}

\end{table}

\begin{figure}[htbp]
\centering
\input{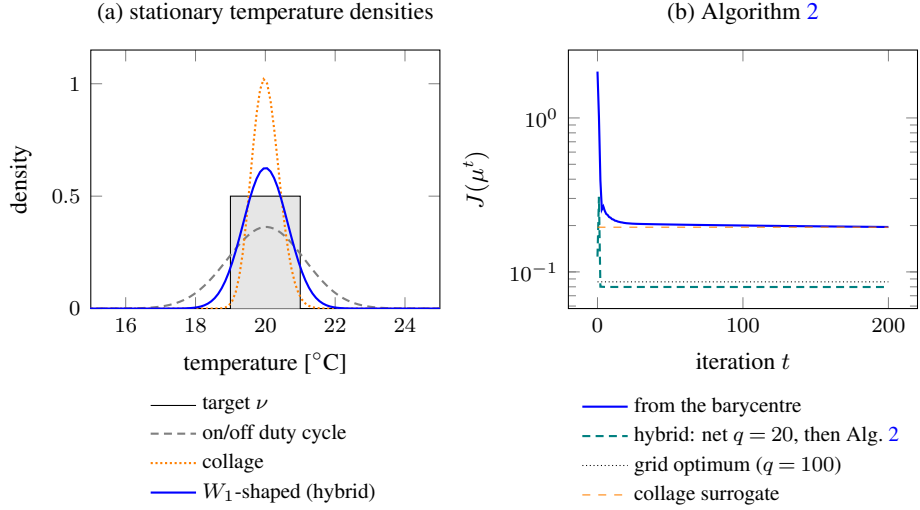}
\caption{TCL population. (a) Stationary temperature densities for the target band and three of the five
designs of Table~\ref{tab:tcl}; the moment-matching design and the design of Algorithm~\ref{alg:shaping} with $q=100$
are omitted, since their densities are close to that of the hybrid design. (b) Objective values along
Algorithm~\ref{alg:psg}.}
\label{fig:tcl}
\end{figure}

\subsection{Many modes}\label{sec:exp-many}
All systems above have $m=3$ modes, which makes landscapes visible and exhaustive search feasible. We now
take $m=12$ affine maps $w_i(x)=a_ix+(1-a_i)z_i$ of $[0,1]$, contracting towards equally spaced points
$z_i\in[0.04,0.96]$ with slopes $a_i\in[0.30,0.60]$, so that $L_{\max}=0.6$ and $\cA=\Delta_{12}$. The
finite-state model has $n=2001$ grid points, and by Proposition~\ref{prop:quantization}(ii) its invariant
measures are within $3.2\cdot10^{-4}$ of the true ones in $W_1$. Already the net with $q=20$ has
$\binom{31}{11}\approx8.5\cdot10^7$ points, so Algorithm~\ref{alg:shaping} is out of reach. The sensitivity
vectors agree with central finite differences to a relative error of at most $6.1\cdot10^{-7}$ over $10$ random
directions.

For the bimodal target $\nu=0.6\,\mathrm{Beta}(4,14)+0.4\,\mathrm{Beta}(16,5)$, Algorithm~\ref{alg:psg}
($\gamma=0.5$, $T=300$) was started at the barycenter and at $9$ random laws drawn from the uniform
distribution on $\Delta_{12}$. All ten runs end with $J\in[0.0393,0.0396]$, and the median of the best
value found is $0.0408$ after $50$ iterations and $0.0398$ after $100$ (Figure~\ref{fig:scalability}(a)). The
ten returned laws are supported on the same four modes, $\{1,3,11,12\}$, but differ by up to $0.13$ in
$\ell_1$ (median pairwise distance $0.05$): they lie in a flat valley rather than at one point. Random search
with twice as many evaluations of $J$ ($600$ laws drawn uniformly from $\Delta_{12}$) reaches only
$0.052$--$0.069$ (median $0.062$) over ten repetitions. Stochastic gradient descent that uses only values of $J$,
with the two-point random-direction gradient estimates of~\cite{NesterovSpokoiny2017}, the projection and
step-size rule of Algorithm~\ref{alg:psg}, and for each budget the best of the three constants
$\gamma\in\{0.1,0.5,2\}$, reaches from the same ten starts a median of $0.0418$ (range $0.0393$--$0.0470$) with
$600$ evaluations of $J$, roughly the cost of $300$ iterations of Algorithm~\ref{alg:psg}, and needs $3000$
evaluations to reach a median of $0.0393$ (range $0.0392$--$0.0403$); the exact sensitivity vectors thus save
a factor of about five in evaluations of $J$. The minimizer of the collage surrogate, computed as a linear
program, has $J=0.0443$, and the uniform law has $J=0.122$. The agreement of the runs suggests that, for this
contractive system, the landscape has no spurious local minima that matter, which is consistent with weak
convexity (Theorem~\ref{thm:structure}) but not implied by it.

For realizable targets, Figure~\ref{fig:scalability}(b) compares Polyak steps with diminishing steps. For
the $8$-adic system $w_i(x)=(x+i-1)/8$, $i\in[8]$, computed exactly on $4$ digits, and $\mu^\star$ drawn from a
Dirichlet$(2,\dots,2)$ law, the push-forwards have disjoint supports, so the linear response is injective
(Proposition~\ref{prop:injectivity}); a search over $10^5$ random directions gives $\kappa_0\lesssim0.046$. With
Polyak steps, $\|\mu^t-\mu^\star\|_1$ falls from $0.66$ to $4\cdot10^{-3}$ after $50$ iterations, $9\cdot10^{-5}$ after
$100$ and $4.5\cdot10^{-9}$ after $200$, a contraction by a factor of about $0.91$ per step, while diminishing
steps ($\gamma=0.5$) stall above $0.1$. This is consistent with Theorem~\ref{thm:psg}(b) but not explained by it:
with $r=\frac18$ one finds $\omega_r\approx1.31$ and $G_r\approx3.23$, and the estimate of $\kappa_0$ gives
$\kappa\approx\kappa_0/2\approx0.023$, so the theorem guarantees linear convergence only from within Euclidean
distance $\kappa/(2\omega_r)\approx0.009$ of $\mu^\star$, at the rate $1-\kappa^2/(2G_r^2)\approx1-2.7\cdot10^{-5}$ per
step, whereas the run starts at distance $0.26$. The worst-case constants $\omega_r$ and $G_r$ are far from
sharp here; a local analysis in the $\ell_1$-geometry of the simplex, in terms of $N_{\mu^\star}$, is left open. For the $12$-map system above and $\mu^\star$ drawn uniformly from $\Delta_{12}$,
the decrease is much slower: after $400$ Polyak steps $J(\mu^t)\approx3\cdot10^{-3}$, but
$\|\mu^t-\mu^\star\|_1\approx0.25$. The size of the linear response accounts for the difference: over $10^5$ random
directions $h\in T_0$ with $\|h\|_1=1$, $N_{\mu^\star}(h)$ ranges from $0.0063$ to $0.34$ (median $0.074$) for the
$12$-map system, against $0.046$ to $0.44$ (median $0.14$) for the $8$-adic one, so that $\kappa_0\lesssim0.0063$ is
about seven times smaller. The $12$ maps contract towards nearby points and their push-forwards overlap
strongly; the target then determines the invariant measure much better than the sampling law.

\begin{figure}[htbp]
\centering
\input{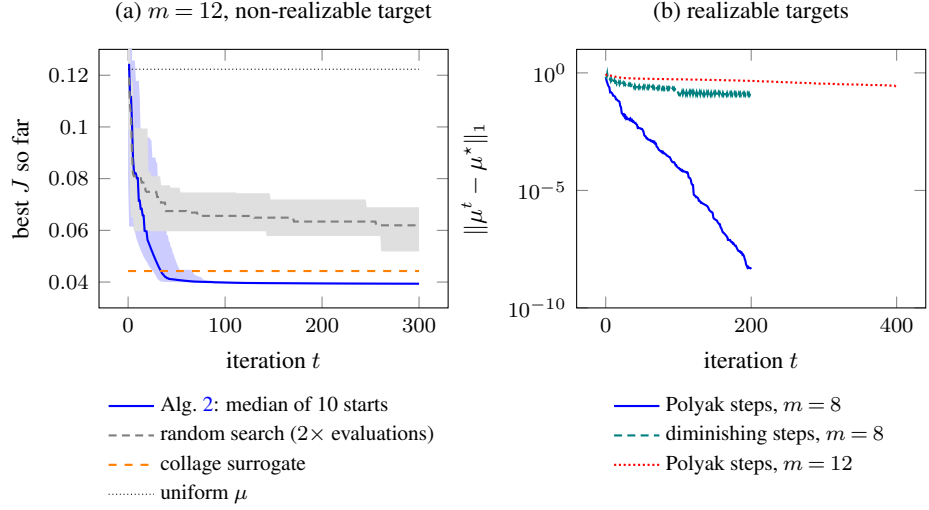}
\caption{Many modes (Section~\ref{sec:exp-many}). (a) $m=12$, non-realizable target: best value of $J$ found by
Algorithm~\ref{alg:psg} (median and range over $10$ starts) and by random search with twice as many
evaluations (median and range over $10$ repetitions), compared with the collage minimizer and the uniform
law. (b) Realizable targets: distance to $\mu^\star$ under Polyak and diminishing steps.}
\label{fig:scalability}
\end{figure}

\section{Conclusion and open problems}\label{sec:conclusion}

We have studied the problem of choosing the sampling law of an iterated function system so that its
invariant measure is as close as possible, in $W_1$, to a prescribed target. The feasible set of the
problem, the admissible set $\cA$ of laws with a unique globally attracting invariant measure, was
treated as an object of study in its own right rather than assumed to be the whole simplex. On bounded
state spaces, the regions in which admissibility is certified by a Doeblin or a splitting condition are
convex unions of relatively open faces, and they exhaust $\cA$ for all finite state spaces; $\cA$ is a
spectral sublevel set for affine systems with entrywise nonnegative matrices; and it has no regularity in
general, beyond being co-analytic on compact spaces. Existence of an optimal sampling law was reduced,
through a relaxation on the closure of $\cA$, to the set of limiting invariant measures at the missing
boundary; for compact state spaces solvability for every target is equivalent to closedness of the set of
reachable invariant measures. Boundary points without invariant measures of finite first moment are
harmless because they repel minimizing sequences, which gave an unconditional existence theorem for
homotheties. A search over a $\delta$-net of the simplex computes, from finite data and with high
probability, a law that is $\eps$-optimal over every regular search region, with constants that do not
involve contraction of the maps; it is exponential in the number of modes. In the contractive regime,
linear response, a criterion for its injectivity and weak convexity of the objective make a projected
subgradient method with inexact subgradients available, which handles twelve modes. We close with four
problems that we consider natural and open.

\begin{openproblem}[The admissible set of affine systems]\label{op:affine}
For affine maps $w_i(x)=A_ix+b_i$ on $\R^d$, Theorem~\ref{thm:affine} shows that
$\{\varrho_1<1\}\subseteq\cA\subseteq\{\rho(\bar A(\mu))<1\}$, with $\cA=\{\varrho_1<1\}$ when the matrices are
nonnegative or the system is linear. Is it true in general that $\cA=\{\mu:\varrho_1(\mu)<1\}$? Equivalently,
can the translation parts ever make the laws of two copies of the chain merge in $W_1$ although the
synchronous coupling does not contract in mean? In dimension one, the Kesten--Goldie tail asymptotics for affine
recursions~\cite{Kesten1973,Goldie1991,BuraczewskiDamekMikosch2016} support a positive answer in non-degenerate cases, and
for similarities with rotations the multidimensional tail theory of~\cite{BuraczewskiEtAl2009} may apply, but
we know of no argument that covers all affine systems. A related question is whether the existence
theorem for homotheties, Theorem~\ref{thm:homothety}, extends to entrywise nonnegative matrices; see the discussion
after its proof.
\end{openproblem}

\begin{openproblem}[Descriptive complexity and regularity of $\cA$]\label{op:borel}
For compact $\cX$ the sets $\cAu$ and $\cU$ are $G_\delta$ (Theorem~\ref{thm:topology}),
$\cAu\subseteq\cA\subseteq\cU$, and $\cA$ is co-analytic (Remark~\ref{rem:topology}(b)). Is $\cA$ always a Borel
subset of the simplex, and if so, of which class? Is
there a verifiable condition on the maps, weaker than the Doeblin and splitting conditions and applicable
to injective maps on a continuum, under which $\cA$ is a union of relatively open faces, or at least
convex? Example~\ref{ex:twolane} shows that compactness and Lipschitz continuity are not sufficient. Since
$P_\mu^n(x,\cdot)$ is atomic for an IFS with finitely many maps, a minorization condition on any set of states
again requires words that merge trajectories; weaker tools, such as drift conditions combined with
asymptotic coupling~\cite{HairerMattinglyScheutzow2011}, depend on $\mu$ through more than its support and
lead to sublevel sets, such as the half-space of Proposition~\ref{prop:contraction}, rather than to unions of
faces.
\end{openproblem}

\begin{openproblem}[Complexity and structure of the optimization problem]\label{op:complexity}
For a finite state space, $\pi_\mu$ is rational on $\cA$ (Theorem~\ref{thm:finite}), and~\eqref{eq:A} is
equivalent to minimizing $W_1(\nu,\pi)$ over pairs $(\mu,\pi)$ subject to the bilinear constraint
$\pi=\sum_i\mu_iW_i^\top\pi$, $\pi\ge0$, $\one^\top\pi=1$, where $W_i$ is the transition matrix of $w_i$, \emph{and} the
combinatorial constraint $\supp\mu\in\sD$; for fixed $\pi$ the bilinear constraint is linear in $\mu$, and for
fixed $\mu$ it is linear in $\pi$. The support constraint cannot be dropped: without it the program
minimizes over all invariant measures of all kernels, and at a vertex $e_j$ with $w_j=\id$ every $\pi$ is
feasible, so that its value is zero for every target; and even when the values coincide, the relaxation
may be attained only outside $\cA$, as in Example~\ref{ex:nonexistence}. Is Problem~\eqref{eq:A} solvable in time
polynomial in $m$, the number of states and $\log(1/\eps)$, or is it NP-hard? For a Dirac target $\nu=\delta_z$,
$J(\mu)=\sum_x\pi_\mu(x)\rho(x,z)$ is the long-run average cost of a randomized policy that ignores the
state, and hardness results for memoryless stochastic controllers of partially observable Markov decision
processes~\cite{VlassisLittmanBarber2012} may be relevant, although our transitions are deterministic. Are
there conditions under
which $J$ has no spurious local minima on $\cA$, so that the exhaustive search of
Algorithm~\ref{alg:shaping} can be replaced by a descent method such as Algorithm~\ref{alg:psg}, also
outside the contractive regime of Section~\ref{sec:sensitivity}?
\end{openproblem}

\begin{openproblem}[Identifiability and richer sampling schemes]\label{op:identifiability}
When is the shaping map $\mu\mapsto\pi_\mu$ injective on $\cA$, so that a reachable target determines the
sampling law? Example~\ref{ex:hold} shows that an idle mechanism destroys injectivity, and the flat
valley in Figure~\ref{fig:landscapes}(b) shows that injectivity can be numerically weak. A quantitative
answer, in the form of a lower bound $W_1(\pi_\mu,\pi_\eta)\ge c\,\|\mu-\eta\|_1$, turns
Theorem~\ref{thm:eps} into a statement about the estimation of $\mu^*$ from a sample of $\pi_{\mu^*}$
(Corollary~\ref{cor:recovery}). In the contractive regime, Propositions~\ref{prop:injectivity}
and~\ref{prop:sharpness} reduce the local question to linear independence of the push-forwards
$(w_i)_\#\pi_{\mu^\star}$, with a constant given by their Kantorovich--Rubinstein geometry; global injectivity,
non-contractive systems and minimax rates remain open. A related question is the extension of the theory to place-dependent
probabilities $\mu_i(x)$~\cite{BarnsleyEtAl1988} and to Markovian sampling of the maps, where the feasible
set is infinite-dimensional and where feedback, as in the ergodic control of
ensembles~\cite{FioravantiEtAl2019,KungurtsevEtAl2023}, enlarges the set of reachable invariant measures.
\end{openproblem}

\appendix

\section{Simulated moments near the missing boundary}\label{sec:supp-affine}
This appendix reports the simulations behind Section~\ref{sec:exp-affine}. For the two systems on $\cX=\R$
defined there,
\begin{align*}
\text{(a)}\quad & w_1(x)=x+1,\ \ w_2(x)=0.5\,x,\ \ w_3(x)=0.3\,x-2, && \mu(p)=(1-p,\tfrac p2,\tfrac p2);\\
\text{(b)}\quad & w_1(x)=x+1,\ \ w_2(x)=x-1,\ \ w_3(x)=0.5\,x, && \mu(p)=(\tfrac{1-p}2,\tfrac{1-p}2,p),
\end{align*}
Table~\ref{tab:affa} compares the closed-form mean $\bar x(\mu(p))=(1-2p)/(0.6p)$ of system (a), and
Table~\ref{tab:affb} the closed-form second moment $\int x^2\,\dd\pi_{\mu(p)}=(1-p)/(0.75p)$ of system (b),
with their Monte Carlo estimates. Each estimate uses $2\cdot10^5$ independent chains started at $0$ and run
for $\lceil40/(1-\langle a,\mu(p)\rangle)\rceil$ steps, after which the law of the state is within
$e^{-40}M_1(\pi_{\mu(p)})$ of $\pi_{\mu(p)}$ in $W_1$; the objective $J(\mu(p))$ for the target $\nu=N(0,1)$ is
estimated from the empirical quantiles. All $18$ estimates lie within $2.2$ standard errors of the closed forms, and $17$ of them
within $1.8$, as expected for independent Monte Carlo errors. In system (a), $M_1(\pi_{\mu(p)})$ and
$J(\mu(p))$ approach $|\bar x(\mu(p))|$, which grows like $p^{-1}$, as in
Proposition~\ref{prop:localization}(i); in system (b) the mean vanishes, and $M_1(\pi_{\mu(p)})$ grows like
$p^{-1/2}$, below the bound $(\int x^2\,\dd\pi_{\mu(p)})^{1/2}$.

\begin{table}[ht]
\caption{System (a): mean of $\pi_{\mu(p)}$ in closed form and by simulation (standard error), first
absolute moment, and $J(\mu(p))$ for $\nu=N(0,1)$.}
\label{tab:affa}
\begin{tabular}{@{}lccccc@{}}
\hline
$p$ & $\langle a,\mu(p)\rangle$ & $\bar x(\mu(p))$ & simulated mean & $M_1(\pi_{\mu(p)})$ & $J(\mu(p))$ \\
\hline
0.90 & 0.460 & $-1.481$ & $-1.484$ (0.002) & 1.543 & 1.485 \\
0.70 & 0.580 & $-0.952$ & $-0.953$ (0.003) & 1.359 & 0.976 \\
0.50 & 0.700 & 0.000 & 0.001 (0.004) & 1.474 & 0.695 \\
0.30 & 0.820 & 2.222 & 2.209 (0.007) & 2.824 & 2.249 \\
0.20 & 0.880 & 5.000 & 5.006 (0.012) & 5.258 & 5.006 \\
0.10 & 0.940 & 13.333 & 13.329 (0.024) & 13.368 & 13.329 \\
0.05 & 0.970 & 30.000 & 29.944 (0.049) & 29.947 & 29.944 \\
0.02 & 0.988 & 80.000 & 79.944 (0.123) & 79.944 & 79.944 \\
0.01 & 0.994 & 163.333 & 163.099 (0.247) & 163.099 & 163.099 \\
\hline
\end{tabular}

\end{table}

\begin{table}[ht]
\caption{System (b): second moment of $\pi_{\mu(p)}$ in closed form and by simulation (standard error),
first absolute moment, and $J(\mu(p))$ for $\nu=N(0,1)$.}
\label{tab:affb}
\begin{tabular}{@{}lccccc@{}}
\hline
$p$ & $\langle a,\mu(p)\rangle$ & $(1-p)/(0.75p)$ & simulated $\int x^2\dd\pi_{\mu(p)}$ & $M_1(\pi_{\mu(p)})$ & $J(\mu(p))$ \\
\hline
0.90 & 0.550 & 0.148 & 0.149 (0.001) & 0.184 & 0.614 \\
0.70 & 0.650 & 0.571 & 0.567 (0.003) & 0.497 & 0.304 \\
0.50 & 0.750 & 1.333 & 1.332 (0.006) & 0.825 & 0.121 \\
0.30 & 0.850 & 3.111 & 3.087 (0.014) & 1.292 & 0.494 \\
0.20 & 0.900 & 5.333 & 5.369 (0.024) & 1.714 & 0.916 \\
0.10 & 0.950 & 12.000 & 11.938 (0.052) & 2.573 & 1.775 \\
0.05 & 0.975 & 25.333 & 25.346 (0.110) & 3.748 & 2.951 \\
0.02 & 0.990 & 65.333 & 65.220 (0.281) & 6.026 & 5.228 \\
0.01 & 0.995 & 132.000 & 130.782 (0.565) & 8.548 & 7.750 \\
\hline
\end{tabular}

\end{table}

\section{A contractive benchmark}\label{sec:supp-battery}
For comparison with the classical setting we take three affine contractions of $[0,1]$,
\[
w_1(x)=0.4\,x+0.1,\qquad w_2(x)=0.5\,x+0.4,\qquad w_3(x)=0.3\,x+0.05,
\]
with $L_{\max}=0.5$, so that $\cA=\Delta_3$ and Corollary~\ref{cor:constants}(c) gives $L_\pi=1$ and
$r_k=2^{-k}$. The target is $\nu=\pi_{\mu^*}$ with $\mu^*=(0.2,0.5,0.3)$, which belongs to the net for
$q=60$, and we run Algorithm~\ref{alg:shaping} over the whole net. With $\lambda_0=\delta_0$, the atoms
$w_{\bi}(0)$, $\bi\in[3]^k$, of $\pi^{(k)}_\mu$ were enumerated once and rounded to a mesh of width $2^{-12}$,
which perturbs every $W_1$-distance below by at most $2^{-12}\approx2.4\cdot10^{-4}$; the reference measures
use $k=14$. Entries of Table~\ref{tab:battery} below this resolution (the value of $J(\hat\mu)$ for $N=10^5$,
$k=12$, and the truncation error for $k=12$) should be read as ``at most of order $10^{-4}$''.
Table~\ref{tab:battery} reports means over $200$ replications. The error $J(\hat\mu)$ follows the
statistical term once $2^{-k}$ is below it, and the parameter is recovered as well, the landscape being
less flat than in Section~\ref{sec:exp-storage}. On the net, the largest observed ratio
$W_1(\pi_\mu,\pi_\eta)/\|\mu-\eta\|_1$ was $0.49$, half of the theoretical constant $L_\pi=1$. As a consistency check of the implementation, the deterministic inequality~\eqref{eq:oracle} was evaluated exactly
in all $4{,}800$ runs, which include a coarser net with $q=20$ (not shown), and held in each.

\begin{table}[ht]
\caption{Contractive benchmark, $q=60$ ($|\cN_\delta|=1891$): mean of $J(\hat\mu)$ and of the parameter error
$\|\hat\mu-\mu^*\|_1$ over $200$ replications; last column: mean of $W_1(\hat\nu_N,\nu)$; last row:
truncation error $\max_{\mu\in\cN_\delta}W_1(\pi_\mu^{(k)},\pi_\mu)$.}
\label{tab:battery}
\begin{tabular}{@{}lccccccc@{}}
\hline
 & \multicolumn{2}{c}{$k=3$} & \multicolumn{2}{c}{$k=6$} & \multicolumn{2}{c}{$k=12$} & \\
$N$ & $J(\hat\mu)$ & error & $J(\hat\mu)$ & error & $J(\hat\mu)$ & error & $W_1(\hat\nu_N,\nu)$ \\
\hline
$10^2$ & 0.0296 & 0.173 & 0.0188 & 0.145 & 0.0186 & 0.146 & 0.0223 \\
$10^3$ & 0.0257 & 0.136 & 0.0066 & 0.053 & 0.0063 & 0.050 & 0.0074 \\
$10^4$ & 0.0257 & 0.132 & 0.0017 & 0.019 & 0.0011 & 0.014 & 0.0022 \\
$10^5$ & 0.0258 & 0.128 & 0.0014 & 0.017 & 0.0001 & 0.001 & 0.0007 \\
\hline
truncation & \multicolumn{2}{c}{$1.0\cdot10^{-1}$} & \multicolumn{2}{c}{$1.2\cdot10^{-2}$} & \multicolumn{2}{c}{$7.3\cdot10^{-5}$} & \\
\hline
\end{tabular}
\end{table}

\end{document}